\documentclass[letterpaper, 10pt, twocolumn, journal]{IEEEtran}

\usepackage{amsmath}
\usepackage{amsthm}
\usepackage{amssymb}
\usepackage{amsfonts}
\usepackage{bm}
\usepackage{cite,comment}
\usepackage{xcolor}
\usepackage{graphicx}
\usepackage{tikz}
\usepackage{pgfplots}
\pgfplotsset{compat=1.18}
\usepackage[mode=buildnew]{standalone}
\usepackage{subcaption}
\usepackage{scalefnt}

\usetikzlibrary{calc}

\newlength\fwidth
\newcommand{\R}[1]{\ensuremath{\mathbb{R}^{#1}}}
\renewcommand{\S}[1]{\ensuremath{\mathbb{S}_{#1}}}
\renewcommand{\P}[2][]{\mathbb{P}_{#1} ( #2 ) }

\newcommand{\cN}{\mathcal{N}}

\newcommand{\Zplus}{\mathbb{Z}\sub{+}}

\newcommand{\E}[1]{\mathbb{E}\left[#1\right]}
\newcommand{\Cov}[1]{\mathrm{Cov}\left( #1 \right)}
\newcommand{\tr}[1]{\ensuremath{\mathrm{tr} \left( #1 \right) }}
\newcommand{\vertcat}[1]{\ensuremath{\mathrm{vertcat} (#1)}}
\newcommand{\bdiag}[1]{\ensuremath{\mathrm{bdiag} (#1)}}

\DeclareMathOperator*{\argmin}{arg\,min}

\newcommand{\transpose}{\ensuremath{^{\mathrm{T}}}}
\newcommand{\sub}[1]{\ensuremath{_{#1}}}
\newcommand{\spr}[1]{\ensuremath{^{#1}}}

\newcommand{\curly}[1]{ \ensuremath{ \{ #1 \} } }

\newtheorem{theorem}{Theorem}
\newtheorem{lemma}{Lemma}
\newtheorem{proposition}{Proposition}
\newtheorem{problem}{Problem}
\theoremstyle{definition}
\newtheorem{remark}{Remark}

\newtheorem{assumption}{Assumption}
\newtheorem{definition}{Definition}
\newtheorem{corollary}{Corollary}

\renewenvironment{proof}
    {\noindent \textit{Proof.}}
    {\hfill \ensuremath{\blacksquare}}

\def\initColor{blue!30}
\def\termColor{green!30}
\def\desColor{red!30}
\def\ex{0.5}
\def\ey{0.5}

\newcommand{\drawellipseInit}[7]{
    \begin{scope}[shift={(#1, #2)}]
    \begin{scope}[rotate=#3]
        \fill[fill=\initColor, opacity=#7] (0, 0) ellipse (#4cm and #5cm);
    \end{scope};
    \node at (0, 0) {$\mu_#6^0, \Sigma_#6^0$};
    \end{scope};
}

\newcommand{\drawellipseDes}[7]{
    \begin{scope}[shift={(#1, #2)}]
    \begin{scope}[rotate=#3]
        \fill[fill=\desColor, opacity=#7] (0, 0) ellipse (#4cm and #5cm);
    \end{scope};
    \node at (0, 0) {$\mu_#6^\mathrm{d}, \Sigma_#6^\mathrm{d}$};
    \end{scope};
}

\newcommand{\drawellipseTerm}[7]{
    \begin{scope}[shift={(#1, #2)}]
    \begin{scope}[rotate=#3]
        \fill[fill=\termColor, opacity=#7] (0, 0) ellipse (#4cm and #5cm);
    \end{scope};
    \node at (0, 0) {$\mu_#6^N, \Sigma_#6^N$};
    \end{scope};
}

\title{Constrained Multi-Modal Density Control of Linear Systems via Covariance Steering Theory}

\author{Isin M. Balci, \and Efstathios Bakolas}
\date{November 2024}

\begin{document}

\maketitle

\begin{abstract}
    In this paper, we investigate finite-horizon optimal density steering problems for discrete-time stochastic linear dynamical systems whose state probability densities can be represented as Gaussian Mixture Models (GMMs). Our goal is to compute optimal controllers that can ensure that the terminal state distribution will match the desired distribution exactly (hard-constrained version) or closely (soft-constrained version) where in the latter case we employ a Wasserstein like metric that can measure the distance between different GMMs. Our approach relies on a class of randomized control policies which allow us to reformulate the proposed density steering problems as finite-dimensional optimization problems, and in particular, linear and bilinear programs. Additionally, we explore more general density steering problems based on the approximation of general distributions by GMMs and characterize bounds for the error between the terminal distribution under our policy and the approximated GMM terminal state distribution. Finally, we demonstrate the effectiveness of our approach through non-trivial numerical experiments.
\end{abstract}


\section{Introduction}\label{s:introduction}
In this paper, we address a class of optimal multi-modal density steering problems for discrete-time linear dynamical systems, in which the probability density of the state process is represented by Gaussian Mixture Models (GMMs), owing to their universal approximation property \cite[Chapter 3]{b:stergiopoulos2017advSignalProHandbook}. Such problems fall under the umbrella of ``optimal mass transport'' (OMT) \cite{p:chen2016OMToverLinearDS}. Throughout the paper, we explore various versions of the density steering problem, including those with input and state constraints, and cost functions designed as convex combinations of Wasserstein-like distance functions for GMMs \cite{p:chen2018OT-GMM} and quadratic cost functions (in terms of state and control inputs).

Typically, three different approaches can be employed to solve density steering problems. 
The goal in the first approach, which considers continuous-time state space models, is to control the evolution of the Fokker-Planck partial differential equation (PDE) describing the evolution of the probability density function (PDF) of the state \cite{p:eren2017velocityFieldDensityControl, p:zheng2022backsteppingDensityControl, p:sinigaglia2022density}. The second approach, which considers discrete state space models, employs Markov chain-based methods along with convex optimization tools to design a transition matrix for the Markov chain that will realize the transfer of the probability distribution to the desired one \cite{p:demir2015decentralizedProbabilisticDensityControl, p:djeumou2022probabilisticSwarmGTL}. Lastly, the third approach treats the density steering problem as a static mass transport problem which can be addressed by well-known optimal mass transport (OMT) algorithms \cite{p:chen2016OMToverLinearDS, p:debadyn2021discreteTimeLQRviaOT, p:ito2023entropicMPCOT, p:krishnan2018distributedOTforSwarm}.

Our approach 
is primarily aligned with OMT methods. 
However, instead of seeking the optimal transport map in a general discretized state-space, we utilize GMM along with covariance steering theory \cite{p:chen2015optimalCSI, p:goldshtein2017finitehorizonCS, p:balci2022exact} for a gridless approach. 
This enables us to formulate and solve a lower-dimensional linear program (LP) in the unconstrained case and a bilinear program for the constrained case, providing a computationally efficient method.

\begin{figure}
    \centering
    \begin{tikzpicture}
        \node (main image) at (0, 0) {\includegraphics[width=0.85\linewidth]{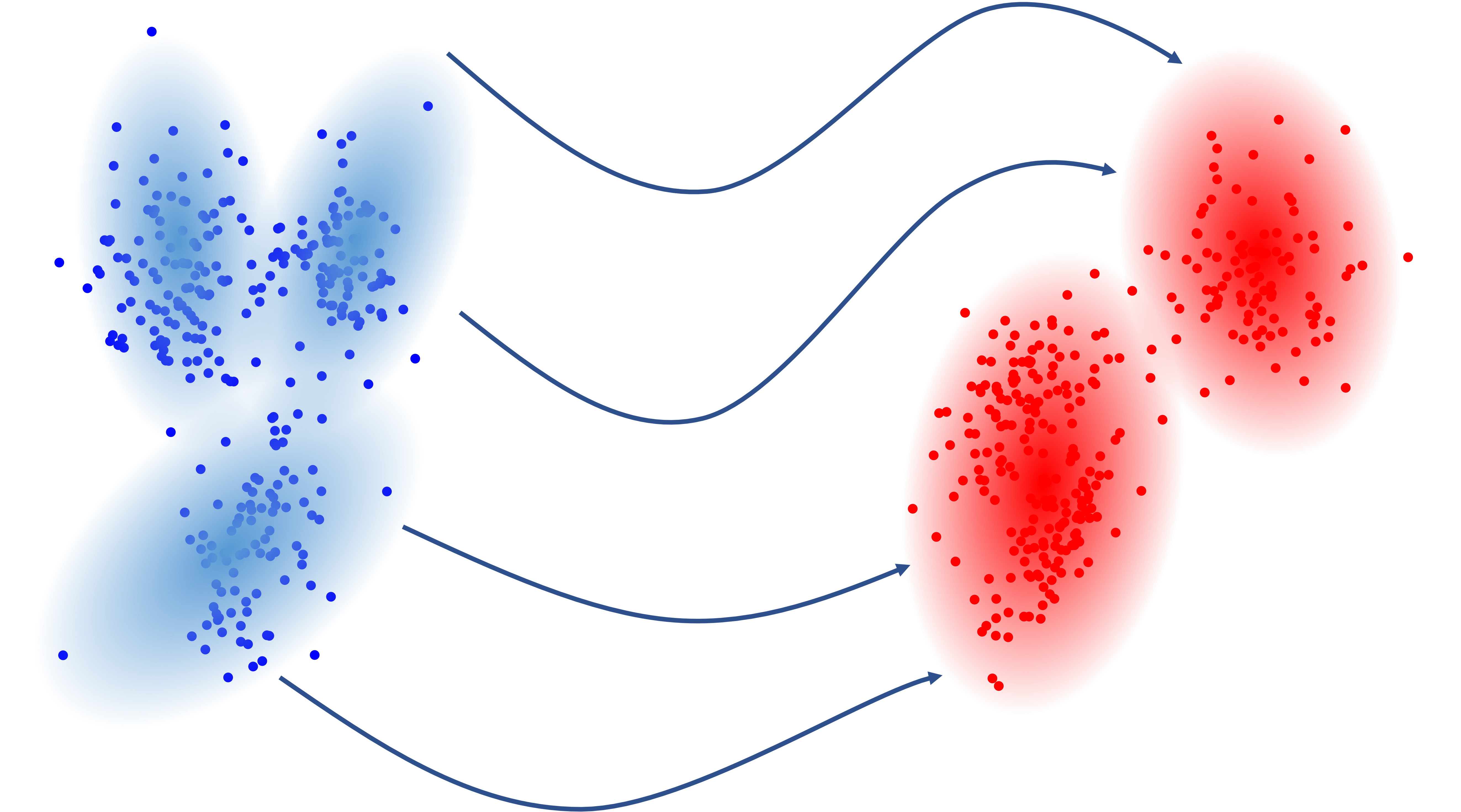}};
        \node (x0) at (-2, 2.2) {\small \textcolor{blue}{$x_0 \sim \mathbf{GMM}(\curly{p_i^0, \mu_i^0, \Sigma_i^0}_{i=0}^{2})$}};
        \node (xN) at (2.1, -2.2) {\small \textcolor{red}{$x_N \sim \mathbf{GMM}( \curly{p_i^N, \mu_i^N, \Sigma_i^N}_{i=0}^1 )$}};
    \end{tikzpicture}
    \caption{Graphical illustration of the multi-modal density control problem. 
    In this problem, a feedback control policy is sought to steer the uncertain initial state $x_0$ drawn from a (blue) GMM distribution to a terminal state $x_N$ drawn from another (red) GMM distribution in finite time. }
    \label{fig:GMM-Steering-illustration}
\end{figure}

\noindent \textbf{\textit{Literature Review:}} 
Density steering problems have received significant attention in the relevant literature. 
In \cite{p:debadyn2021discreteTimeLQRviaOT}, the authors consider a density steering problem with linear dynamics and a convex quadratic cost, recasting it as an optimal mass transport problem. The authors
of \cite{p:terpin2023DPinProbSpace} address the density steering problem treating it as a dynamic program over probability spaces. It is worth noting, however, that both of these approaches involve discretizing the continuous state-space, as computing the optimal transport map and solving recursively the Bellman equation requires a discrete state-space.

A special class of density steering problems is the so-called covariance steering (CS) problems \cite{p:chen2015optimalCSI,p:bakolas2018automatica} whose main objective is to steer the first two moments of the uncertain state of a dynamical system to desired quantities. CS problems have been studied extensively for both continuous-time \cite{p:chen2015optimalCSI} and discrete-time \cite{p:bakolas2018automatica,p:goldshtein2017finitehorizonCS} settings.
Constrained CS problems have been studied in \cite{p:bakolas2018automatica, p:okamoto2018chance, p:balci2024adf-journal}. Moreover, the squared Wasserstein distance has been used to formulate a soft constrained version of the CS problem in \cite{p:balci2022exact, p:halder2016finite-wasserstein}.
For computational efficiency, truncated affine disturbance feedback policies are used in CS problems in \cite{p:balci2024adf-journal}.
Both closed-form solutions \cite{p:chen2015optimalCSI, p:goldshtein2017finitehorizonCS} and optimization-based approaches \cite{p:bakolas2018automatica, p:balci2022exact, p:okamoto2018chance} to CS problems avoid state-space discretization and extensive sampling.
However, the main limitation of the CS methods is their inability to steer multi-modal distributions.


Density steering problems can be formulated over discrete state-spaces as Markov chain synthesis problems. The problem of characterizing a Markov chain that will realize the transfer to the desired probability distribution in the infinite horizon case can be cast as a convex semi-definite program (SDP) \cite{p:demir2015decentralizedProbabilisticDensityControl}, whereas in the finite horizon case, the problem is not necessarily convex and should be addressed as a general nonlinear program (NLP) \cite{p:hoshino2021MCWassersteinCarSharing, p:djeumou2022probabilisticSwarmGTL}.

In \cite{p:halder2021NonlinearDensity}, the authors address the density steering problem with nonlinear drift, deriving optimality conditions expressed through coupled PDEs which are solved by using the Feynman-Kac lemma and point cloud sampling. Meanwhile, \cite{p:sivaramakrishnan2022distributionSteeringGeneralDist} explores the distribution steering problem for linear systems under the influence of non-Gaussian noise, leveraging characteristic functions. Another approach, detailed in \cite{p:wu2023powerMomentSteering}, employs power moments to formulate convex optimization problems for steering general probability densities in one-dimensional settings. Furthermore, \cite{p:elamvazhuthi2015optimalCoverageControl} presents a PDE-based optimal robotic swarm coverage control policy, while \cite{p:bakshi2020stabilizingDensityControl} focuses on deriving the optimal density steering laws for systems with multiplicative noise for the infinite horizon case. 
Lastly, \cite{p:saravanos2023distributedLargeScale} introduces a hierarchical clustering-based density steering algorithm tailored for applications of distributed large-scale robotic networks.

The density steering problem using Gaussian mixture models has been investigated in \cite{p:abdulghafoor2023LargeScaleControl, p:zhu2021adaptive, p:yang2024riskaware-swarm}.
However, a notable limitation of \cite{p:abdulghafoor2023LargeScaleControl, p:zhu2021adaptive, p:yang2024riskaware-swarm} is that their proposed density steering methods do not explicitly account for the system's dynamics. Moreover, their approach relies on spatial discretization which results in high computational cost. Finally, it should be mentioned that a preliminary version of this paper \cite{p:balci2023density} studied the hard-constrained GMM steering problem in the absence of input/state constraints. In this extended version, we study constrained problem formulations and also present theoretical error bounds for the GMM approximations of the PDF of the (uncertain) state of linear systems.

\noindent \textbf{\textit{Main Contributions:}}
All the aforementioned methods for density steering, besides CS methods, involve state-space discretization, solving partial differential equations, or extensive sampling. 
In contrast, our approach offers a computationally efficient solution to the optimal multi-modal density steering problem without these complexities. 
Furthermore, except from \cite{p:ito2023entropicMPCOT}, none of the referenced papers consider control input or state constraints jointly with system dynamics. By contrast, the problem formulations in this paper address both constraints while explicitly accounting for system dynamics.

In this paper, we employ convex optimization techniques and covariance steering theory to tackle the proposed density steering problems. Initially, we revisit the unimodal density steering problem (covariance steering) and derive in closed-form the optimal state feedback policy that steers an initial state (Gaussian) distribution to a desired terminal state (Gaussian) distribution while minimizing a quadratic cost. Secondly, we introduce a class of randomized state feedback policies to reduce the primary problem into a finite-dimensional NLP. These policies ensure that the state density can be represented as a Gaussian mixture model throughout the time horizon. Lastly, we show that the finite-dimensional optimization problem obtained by utilizing the proposed policy corresponds to an LP for the hard-constrained problem whereas the other formulations give rise to bilinear programs. To solve these bilinear programs, we propose a block coordinate descent (BCD) solution technique.

Furthermore, we showcase the effectiveness of our methodology in steering arbitrary probability state distributions of linear dynamical systems approximated by GMMs computed by the expectation-maximization algorithm \cite{p:expectationMaximization}. 
We demonstrate that the error in the GMM approximation remains bounded after applying the proposed policy and provide the aforementioned bounds. 
Finally, we present the results of numerical experiments along with their computation times to demonstrate the efficacy of our approach.

\noindent \textbf{\textit{Organization of the Paper:}} 
Preliminary definitions, problem setup and formulations are given in Section \ref{s:problem-formulation}. In Section \ref{s:covariance-steering}, we revisit the optimal covariance steering for linear dynamical systems. In Section \ref{s:gmm-steering-policies}, we define a class of randomized policies under which the state distribution will remain a GMM at all future time steps. In Section \ref{s:formulation-nlp}, we demonstrate how the proposed density steering problems can be reduced to their corresponding NLPs which are subsequently solved by using BCD-based algorithms that are introduced in Section \ref{s:solution-procedure}. 
Error bounds for GMM approximations for arbitrary probability distributions are discussed in Section \ref{s:error}. 
Results of our numerical simulations are presented and discussed in Section \ref{s:numerical-simulations}.
Finally, the paper concludes with remarks in Section \ref{s:conclusion}.

\section{Problem Formulation}\label{s:problem-formulation}
\noindent \textbf{Notation:} $\R{n}$ ($\R{n \times m}$) denotes the space of $n$-dimensional real vectors ($n\times m$ matrices). $\Zplus$ represents the set of positive integers. 
The convex cone of positive definite (semi-definite) $n\times n$ matrices is denoted by $\S{n}^{++}$ ($\S{n}^{+}$). 
For any $x\in \R{n}$ and $Q \in \S{n}^{+}$, $\lVert x \rVert_Q := \sqrt{x\transpose Q x }$. $I\sub{n}$ denotes the identity matrix of size $n$. 
$\bm{1}_n$ denotes the $n$-dimensional vector of ones.
$\bm{0}$ denotes the zero matrix. Vertical concatenation of vectors or matrices $x_0, \dots, x_N$ is denoted as $\vertcat{x_0, \dots, x_N}$. 
The trace and determinant of a matrix $A \in \R{n \times n}$ are denoted by $\tr{A}$ and $\det (A)$, respectively. 
$\lVert A \rVert_{*}$ denotes the nuclear norm of a matrix $A \in \R{n \times m}$. $A \geq (\leq) \bm{0}$ denotes element-wise comparison for matrices $A \in \R{n \times m}$.
For symmetric matrices $A, B \in \R{n \times n} $, $A\succeq(\succ) B$ means that $A - B \in \S{n}^{+}(\S{n}^{++})$. $\bdiag{A_1, \dots, A_{N}}$ denotes the block diagonal matrix with diagonal blocks $A_1, \dots, A_{N}$. 
The expectation and covariance of a random variable $x$ are denoted as $\E{x}$ and $\Cov{x}$, respectively. 
The notation $x \sim \cN (\mu, \Sigma)$ means that $x$ is a Gaussian random variable with mean $\mu$ and covariance matrix $\Sigma$. $\Delta_n$ denotes the probability simplex in $\mathbb{R}^n$, where $\Delta_n:=\{[p_1,\dots, p_n]\transpose\in\mathbb{R}^n: \sum_{i=1}^n p_i = 1 ~\text{and}~ p_i\geq 0~ \forall i\}$.
When $x$ follows a Gaussian mixture model, we write $x \sim \mathbf{GMM}(\curly{p\sub{i}, \mu_i, \Sigma_i}\sub{i=0}\spr{n-1})$ where $\bm{p}_n := [p_0, \dots, p_{n-1}]\transpose \in \Delta_n$. 
The PDF of a random variable $x \in \R{n}$ evaluated at $x' \in \R{n}$ is $\P[x]{x'}$.
If $x\sim \mathcal{N}(\mu, \Sigma)$, we write $\P[\mathcal{N}]{x'; \mu, \Sigma}$ instead of $\P[x]{x'}$.
For random variables $x \in \R{n}$ and $y \in \R{m}$, we denote by $x|y=\Hat{y}$ the conditional random variable $x$ given $y=\Hat{y}$. $\mathcal{P} (\mathcal{A})$ denotes the set of all random variables over $\mathcal{A} \subseteq \R{n}$.

\noindent \textbf{Preliminaries:} 
Next, we provide some basic definitions from the OMT literature \cite{p:chen2018OT-GMM} that will be used throughout this paper. 
\begin{definition}\label{def:squared-wasserstein}
Let $\rho_x, \rho_y : \R{n} \rightarrow \R{}_{+}$ be PDFs of random variables $ \boldsymbol{x}, \boldsymbol{y} \in \R{n}$. The squared Wasserstein distance between $\rho_x$ and $\rho_y$ is denoted as $W_2^2(\rho_x, \rho_y)$ and defined as:
\begin{align}\label{eq:squared-wasserstein}
    W_2^2(\rho_x, \rho_y) = \inf_{\sigma \in \mathbb{H} (\rho_x, \rho_y)} \iint \lVert x - y \rVert_2^2 ~ \sigma(x, y)  \mathrm{d}x \mathrm{d}y 
\end{align}
where $\mathbb{H}(\rho_x, \rho_y)$ denotes the set of all PDFs over $\R{2 n}$ with finite second moments and marginals over $\rho_x$ and $\rho_y$, that is,
\begin{align*}
    \mathbb{H}(\rho_x, \rho_y) := &  \Big \{ \sigma(x, y) : \iint \lVert [x\transpose, y\transpose]\transpose \rVert_2^2 \sigma(x, y) \mathrm{d}x \mathrm{d} y < \infty,  \\
    & \int \sigma(x, y) \mathrm{d} x = \rho_y(y), \int \sigma(x, y) \mathrm{d}y = \rho_x(x) \Big \}.
\end{align*}
\end{definition}
To compute $W_2^2(\rho_x, \rho_y)$, one needs to solve the functional optimization problem in \eqref{eq:squared-wasserstein}, which is generally intractable except from a few special cases.
One such case occurs when $\rho_x(x) = \P[\cN]{x; \mu_x, \Sigma_x}$ and $\rho_y(y) := \P[\cN]{y; \mu_y, \Sigma_y}$, in which case
\begin{align}\label{eq:wasserstein-gaussian}
   W_\mathcal{N}^2(\mu_x, \Sigma_x, \mu_y, \Sigma_y) := & \lVert \mu_x - \mu_y \rVert_2^2 + \tr{\Sigma_x + \Sigma_y} \nonumber \\
    & -2 \tr{ ( \Sigma_y^{1/2} \Sigma_x \Sigma_y^{1/2} )^{1/2} }.
\end{align}

For arbitrary distributions over $\R{n}$ with $n > 1$, $W_2^2(\rho_x, \rho_y)$ must be computed numerically. For instance, one can represent the continuous PDFs $\rho_x$ and $\rho_y$ by finite sets of samples $\curly{x_i}_{i=1}^{N}$ and $\curly{y_i}_{i=1}^{M}$, respectively.
Additionally, let $\curly{p_i^x}_{i=1}^{N}$ and $\curly{p_i^{y}}_{i=1}^{M}$ represent the probability mass assigned to each $x_i$ and $y_i$, such that $\mathbf{p}^x = [p_1^x, \dots, p_N^x]\transpose \in \Delta_N$ and $\mathbf{p}^y = [p_1^y, \dots, p_M^y]\transpose \in \Delta_N$, respectively.
Then, $W_2^2(\rho_x, \rho_y)$ can be calculated by solving the following OMT problem: 
\begin{align}\label{eq:discrete-OMT-problem}
    \min_{\mathbf{M} \in \bm{\mathcal{M}}}  \sum_{i=1}^{N} \sum_{j=1}^{M} \mathbf{M}_{i, j} \mathbf{C}_{i, j}
\end{align}
where $\mathbf{C}_{i,j} = \lVert x_i - y_j \rVert_{2}^{2}$ and 
$\bm{\mathcal{M}} : = \curly{\mathbf{M} =[\mathbf{M}_{i, j}] \in \R{N \times M} : \mathbf{M} \geq \bm{0}, ~ \mathbf{M} \bm{1} = \mathbf{p}^x,   \mathbf{M}\transpose \bm{1} = \mathbf{p}^y }$.
Inspired by the discrete OMT problem \eqref{eq:discrete-OMT-problem}, the authors of \cite{p:chen2018OT-GMM} proposed a GMM-Wasserstein distance as a distance metric between GMMs. 

\begin{definition}\label{def:gmm-wasserstein}
    Let $ \bm{x} \sim \mathbf{GMM}(\curly{p_i^x, \mu_i^x, \Sigma_i^x}_{i=0}^{N-1} )$ and $\bm{y} \sim \mathbf{GMM}( \curly{p_j^y, \mu_j^y, \Sigma_j^y}_{j=0}^{M-1} )$ and let $\rho_x, \rho_y : \R{n} \rightarrow \R{}_+$ be their corresponding PDFs. 
    The GMM-Wasserstein distance between $\rho_x$ and $\rho_y$ is given as $W_{\mathrm{GMM}}(\rho_x, \rho_y) := \sqrt{ \sum_{i=1}^{N} \sum_{j=1}^{M} \mathbf{C}_{i, j} \mathbf{M}^{\star}_{i, j}}$ where $\mathbf{M}^{\star}_{i, j}$ is the optimal solution of the problem in \eqref{eq:discrete-OMT-problem}
    with $\mathbf{C}_{i, j} = W_\mathcal{N}^2 (\mu_i^x, \Sigma_i^x, \mu_j^y, \Sigma_j^y)$.
\end{definition}
Throughout the paper, we use the squared GMM-Wasserstein distance, given in Definition \ref{def:gmm-wasserstein}, to measure the distance between the PDFs of GMMs. 

\noindent \textbf{Problem Setup:}
We consider a discrete-time linear system:
\begin{equation}\label{eq:linear-system-eq}
    x\sub{k+1} = A_k x_k + B_k u_k
\end{equation}
where $x_k \in \R{n}$ and $u\sub{k} \in \R{m}$ are the state and input processes, respectively. 
We assume that $x\sub{0} \sim \mathbf{GMM}(\curly{p\sub{i}\spr{0}, \mu_{i}\spr{0}, \Sigma_{i}\spr{0}}\sub{i=0}\spr{r-1})$ such that $[p_0,\dots,p_{r-1}]\transpose \in \Delta_r$, $\mu\sub{i}\spr{0} \in \R{n}$ and $\Sigma\sub{i}\spr{0} \in \mathbb{S}\sub{n}\spr{++}$ for all $i \in \curly{0, \dots, n-1}$.
\begin{assumption}\label{assumption:linear-controllable}
    The system dynamics given in \eqref{eq:linear-system-eq} is controllable over a given problem horizon $N \in \mathbb{Z}_{+}$. 
    In other words, the controllability Grammian, $\mathcal{G}_{N:0}$, which is defined as:
    \begin{equation}\label{eq:ctrl-gram-def}
        \mathcal{G}_{N:0} = \sum_{k=0}^{N-1} \Phi_{N, k+1} B_k B_k\transpose \Phi_{N, k+1}\transpose  
    \end{equation}
    is non-singular with $\Phi_{k_2, k_1} := A_{{k_2}-1} A_{{k_2}-2} \dots A_{k_1}$, $\Phi_{k, k} = I_n$ for all $k_2, k_1 \in \mathbb{Z}^+$ such that $k_2 \geq k_1$.
\end{assumption}

A control policy with an horizon $N \in \mathbb{Z}\spr{+}$ for the system \eqref{eq:linear-system-eq} is defined as a sequence of control laws $\pi = \curly{\pi\sub{i}}\sub{i=0}\spr{N-1}$ where each $\pi\sub{i} : \R{n} \rightarrow \mathcal{P} (\R{m})$ is a function that maps the state $x_k$ to a random variable representing control inputs. 
The set of randomized control policies is denoted by $\Pi$.
Throughout the paper, we consider cost functions of the following form:
\begin{align}\label{eq:quadratic-cost-definition}
    J(X_{0:N}, U_{0:N-1}) = J_N(x_N) + \sum_{k=0}^{N-1} J_k(x_k, u_k) 
\end{align}
where $J_k(x_k, u_k) := \lVert u\sub{k} \rVert\sub{R\sub{k}}\spr{2} + \lVert x\sub{k} - x'\sub{k} \rVert\sub{Q\sub{k}}\spr{2}$ for all $k \in \curly{0, \dots, N-1}$, $J_{N}(x_N) := \lVert x_N - x'\sub{k} \rVert_{Q_N}^2$, $R_k \in \S{m}^{++}$ and $Q_k \in \S{n}^{+}$ for all $k$. Furthermore, $X\sub{0:N} = \curly{x\sub{k}}\sub{k=0}\spr{N}$ and $U\sub{0:N-1} = \curly{u\sub{k}}\sub{k=0}\spr{N-1}$. 
Next, we formulate the first problem considered in this paper:
\begin{problem}[Hard-Constrained GMM Density Steering]\label{prob:hard-GMM-density-steering}
Let $N \in \mathbb{Z}_{+}$, $A\sub{k} \in \R{n \times n}, B\sub{k} \in \R{n \times m}$, $R\sub{k} \in \S{m}\spr{++}, Q\sub{k} \in \S{n}\spr{+}$  be given for all $k \in \curly{0, \dots, N-1}$ and $Q_N \in \S{n}^{+}$. 
Also, let $r, t \in \mathbb{Z}_{+}$, $[p\sub{0}\spr{0}, \dots, p\sub{r-1}\spr{0}] \in \Delta_r$, $[p\sub{0}\spr{d}, \dots, p\sub{t-1}\spr{d}] \in \Delta_t$, $\curly{\mu\spr{0}\sub{i}}\sub{i=0}\spr{r-1}$, $\curly{\mu\spr{d}\sub{i}}\sub{i=0}\spr{t-1}$, 
$\curly{\Sigma\spr{0}\sub{i}}\sub{i=0}\spr{r-1}$, 
$\curly{\Sigma\spr{d}\sub{i}}\sub{i=0}\spr{t-1}$ such that $\mu\sub{i}\spr{0} , \mu\sub{i}\spr{d} \in \R{n}$ and $\Sigma\sub{i}\spr{0}, \Sigma\sub{i}\spr{d} \in \S{n}\spr{++}$ be given. 
Find an admissible control policy $\pi\spr{\star} \in \Pi$ that solves the following problem:
\begin{subequations}\label{eq:hard-GMM-density-steering-opt}
\begin{align}
    \min\sub{\pi \in \Pi} & ~~ \E{ J(X_{0:N}, U_{0:N-1}) } \\
    \mathrm{s.t.} & ~~ \eqref{eq:linear-system-eq} \nonumber \\
    & ~~ x\sub{0} \sim \mathbf{GMM}\left(\curly{p\sub{i}\spr{0}, \mu\spr{0}\sub{i}, \Sigma\spr{0}\sub{i}}\sub{i=0}\spr{r-1} \right) \label{eq:hard-problem-init-constr}\\
    & ~~ x\sub{N} \sim \mathbf{GMM}\left(\curly{p\sub{i}\spr{d}, \mu_i\spr{d}, \Sigma_i\spr{d}}\sub{i=0}\spr{t-1}\right) \label{eq:hard-problem-terminal-constr}\\
    & ~~ u_k = \pi\sub{k}(x\sub{k}) \label{eq:hard-problem-policy-constr}
\end{align}
\end{subequations}
\end{problem}

In the presence of state and/or input constraints, the feasibility of Problem \ref{prob:hard-GMM-density-steering} is not guaranteed.
For this reason, we also formulate a soft constrained version of the GMM density steering problem in which the main objective is to minimize the weighted sum of the quadratic cost given in \eqref{eq:quadratic-cost-definition} and the squared GMM-Wasserstein distance (Definition \ref{def:gmm-wasserstein}) between the terminal state distribution and the desired distribution. 
\begin{problem}[Soft-Constrained GMM Density Steering]\label{prob:soft-GMM-density-steering}
    Let $N \in \Zplus$, $\kappa > 0$, $A_k \in \R{n \times n}$, $B_k \in \R{n \times m}$, $R_k \in \S{m}^{++}$, $Q_k \in \S{n}^{+}$ be given for all $k \in \curly{0, \dots, N-1}$ and $Q_N \in \S{n}^{+}$. Also, let $r, t \in \Zplus$, $[p_0^0, \dots, p_{r-1}^0] \in \Delta_r$, $[p_0^d, \dots, p_{t-1}^d] \in \Delta_t$, 
    $\curly{\mu\spr{0}\sub{i}}\sub{i=0}\spr{r-1}$, $\curly{\mu\spr{d}\sub{i}}\sub{i=0}\spr{t-1}$, 
    $\curly{\Sigma\spr{0}\sub{i}}\sub{i=0}\spr{r-1}$, 
    $\curly{\Sigma\spr{d}\sub{i}}\sub{i=0}\spr{t-1}$ such that $\mu\sub{i}\spr{0} , \mu\sub{i}\spr{d} \in \R{n}$ and $\Sigma\sub{i}\spr{0}, \Sigma\sub{i}\spr{d} \in \S{n}\spr{++}$ be given. 
    Find an admissible control policy $\pi^{\star} \in \Pi$ that solves the following problem:
    \begin{subequations}
    \begin{align}\label{eq:soft-GMM-density-steering-opt}
        \min_{\pi \in \Pi} & ~~ \E{J(X_{0:N}, U_{0:N-1})} + \kappa W_\mathrm{GMM}^2 (\rho_N, \rho_d) \\
        \mathrm{s.t.} & ~~ \eqref{eq:linear-system-eq}, \eqref{eq:hard-problem-init-constr}, \eqref{eq:hard-problem-policy-constr} \nonumber \\
        & ~~ x_d \sim \mathbf{GMM}(\curly{p_i^d, \mu_i^d, \Sigma_i^d}_{i=0}^{t-1}) \label{eq:soft-prob-desired-constr}
    \end{align}
    \end{subequations}
    where $\rho_N, \rho_d$ are the PDFs of $x_N$ and $x_d$, respectively.
\end{problem}

Problem \ref{prob:soft-GMM-density-steering} can be solved to find a control policy such that constraints of the form $\E{J(X_{0:N,}, U_{0:N-1})} \leq c$, for some $c \geq 0$, are satisfied by adjusting appropriately the value of the parameter $\kappa$. However, the optimal value of the term $\E{J(X_{0:N,}, U_{0:N-1})}$ might be much smaller than $c$, if $\kappa$ is not large enough.

Next, we formulate another constrained density steering problem 
whose main objective is to minimize the squared GMM-Wasserstein distance between the terminal state distribution and the desired state distribution subject to a total quadratic cost constraint. 
\begin{problem}[Total Cost Constrained GMM-Density Steering]\label{prob:total-cost-constrained-GMM-density-steering}
    Let $N \in \Zplus$, $\kappa > 0$, $A_k \in \R{n \times n}$, $B_k \in \R{n \times m}$, $R_k \in \S{m}^{++}$, $Q_k \in \S{n}^{+}$ be given for all $k \in \curly{0, \dots, N-1}$ and $Q_N \in \S{n}^{+}$. 
    Also, let $r, t \in \Zplus$, $[p_0^0, \dots, p_{r-1}^0] \in \Delta_r$, $[p_0^d, \dots, p_{t-1}^d] \in \Delta_t$, 
    $\curly{\mu\spr{0}\sub{i}}\sub{i=0}\spr{r-1}$, $\curly{\mu\spr{d}\sub{i}}\sub{i=0}\spr{t-1}$, 
    $\curly{\Sigma\spr{0}\sub{i}}\sub{i=0}\spr{r-1}$, 
    $\curly{\Sigma\spr{d}\sub{i}}\sub{i=0}\spr{t-1}$ such that $\mu\sub{i}\spr{0} , \mu\sub{i}\spr{d} \in \R{n}$ and $\Sigma\sub{i}\spr{0}, \Sigma\sub{i}\spr{d} \in \S{n}\spr{++}$ be given. 
    Find an admissible control policy $\pi^{\star} \in \Pi$ that solves the following problem:
    \begin{subequations}
    \begin{align}\label{eq:total-cost-constrained-GMM-density-steering-opt}
        \min_{\pi \in \Pi} & ~~  W_\mathrm{GMM}^2 (\rho_N, \rho_d), ~~\mathrm{s.t.} ~~ \eqref{eq:linear-system-eq}, \eqref{eq:hard-problem-init-constr}, \eqref{eq:soft-prob-desired-constr}, \eqref{eq:hard-problem-policy-constr}  \\
        & ~~ \E{J(X_{0:N}, U_{0:N-1})} \leq \kappa
    \end{align}
    \end{subequations}
    where $\rho_N, \rho_d$ are the PDFs of $x_N$ and $x_d$, respectively.
\end{problem}

The last problem that we will study in this paper is similar to Problem \ref{prob:total-cost-constrained-GMM-density-steering} but in contrast with the latter, the state and input constraints are enforced separately for each time step.
\begin{problem}[Step Cost Constrained GMM-Density Steering]\label{prob:step-cost-constrained-GMM-density-steering}
    Let $N \in \Zplus$, $\kappa_k \in \R{}_{+}$, $A_k \in \R{n \times n}$, $B_k \in \R{n \times m}$, $R_k \in \S{m}^{++}$, $Q_k \in \S{n}^{+}$ be given for all $k \in \curly{0, \dots, N-1}$ and $Q_N \in \S{n}^{+}$. Also, let $r, t \in \Zplus$, $[p_0^0, \dots, p_{r-1}^0] \in \Delta_r$, $[p_0^d, \dots, p_{t-1}^d] \in \Delta_t$ 
    $\curly{\mu\spr{0}\sub{i}}\sub{i=0}\spr{r-1}$, $\curly{\mu\spr{d}\sub{i}}\sub{i=0}\spr{t-1}$, 
    $\curly{\Sigma\spr{0}\sub{i}}\sub{i=0}\spr{r-1}$, 
    $\curly{\Sigma\spr{d}\sub{i}}\sub{i=0}\spr{t-1}$ such that $\mu\sub{i}\spr{0} , \mu\sub{i}\spr{d} \in \R{n}$ and $\Sigma\sub{i}\spr{0}, \Sigma\sub{i}\spr{d} \in \S{n}\spr{++}$ be given. 
    Find an admissible control policy $\pi^{\star} \in \Pi$ that solves the following problem:
    \begin{subequations}
    \begin{align}\label{eq:step-cost-constrained-GMM-density-steering-opt}
        \min_{\pi \in \Pi} & ~~  W_\mathrm{GMM}^2 (\rho_N, \rho_d) ~~   \text{s.t.} ~~ \eqref{eq:linear-system-eq}, \eqref{eq:hard-problem-init-constr}, \eqref{eq:soft-prob-desired-constr}, \eqref{eq:hard-problem-policy-constr} \\
        & ~~ \E{J_k(x_k, u_k)} \leq \kappa_k,  \quad \forall k \in \curly{0, \dots, N-1} \label{eq:step-cost-constrained-GMM-density-steering-constr}
    \end{align}
    \end{subequations}
    where $\rho_N, \rho_d$ are the PDFs of $x_N$ and $x_d$, respectively.
\end{problem}

\begin{remark}
    It is worth mentioning that the solutions to Problems \ref{prob:soft-GMM-density-steering}-\ref{prob:step-cost-constrained-GMM-density-steering}, which may appear similar to each other, require different techniques, as will be shown in Section \ref{s:formulation-nlp}.
\end{remark}

\section{Optimal Covariance Steering for Linear Systems}\label{s:covariance-steering}
Optimal covariance steering problems for linear dynamical systems with quadratic cost functions have been extensively studied in the literature \cite{p:chen2015optimalCSI, p:goldshtein2017finitehorizonCS}. 
In this section, we expand upon the results of \cite{p:goldshtein2017finitehorizonCS} and derive a closed-form solution to the covariance steering problem, considering quadratic cost functions (in terms of the state and the control input). 
This results will be subsequently used to formulate the Problems defined in Section \ref{s:problem-formulation} as finite-dimensional optimization problems.
To begin, we revisit the formal definition of the Gaussian covariance steering problem:
\begin{problem}[Gaussian Covariance Steering]\label{prob:covariance-steering}
Let $\mu_0, \mu_d \in \R{n}$, $\Sigma\sub{0}, \Sigma\sub{d} \in \S{n}\spr{++}$, $R\sub{k} \in \S{m}\spr{++}$, $Q\sub{k} \in \S{n}\spr{+}$ for all $k \in \curly{0, \dots, N-1}$ and $Q_N \in \S{n}^{+}$ be given. 
Find an admissible control policy $\pi\spr{\star} \in \Pi$ that solves the following problem:
\begin{align}\label{eq:GaussianCovSteer-Opt-Problem}
    \min\sub{\pi \in \Pi} & ~~ \E{J(X_{0:N}, U_{0:N-1})} \\
    \mathrm{s.t.} & ~~ \eqref{eq:linear-system-eq}, ~~ x\sub{0} \sim \mathcal{N} (\mu_0, \Sigma_0),~~ x\sub{N} \sim \mathcal{N} (\mu_d, \Sigma_d). \nonumber
\end{align}
\end{problem}
In \cite{p:balci2022exact}, it is demonstrated that the optimal policy for the Gaussian covariance steering problem given in \eqref{eq:GaussianCovSteer-Opt-Problem} takes the form of a deterministic affine state feedback policy, expressed as $\pi_k(x_k) = \Bar{u}\sub{k} + K_k (x_k - \mu_k)$ where $\mu_k = \E{x_k}$. 
Furthermore, for deterministic linear systems, this affine state feedback policy can be equivalently expressed in terms of the initial state as $\pi_k(x_0) = \Bar{u}_k + L_k (x_0 - \mu_0)$ \cite{p:goldshtein2017finitehorizonCS}. 
Consequently, the optimal Gaussian covariance steering problem can be rewritten equivalently in terms of decision variables $\curly{\Bar{u}_k, L_k}_{k=0}^{N-1}$ as follows:
\begin{subequations}\label{eq:Gaussian-Cov-Steer-Finite-Dim}
\begin{align}
    \min\sub{\mathbf{\Bar{U}}, \mathbf{L}} & ~~ \mathbf{\Bar{U}}\transpose \mathbf{R} \mathbf{\Bar{U}} + \tr{ \mathbf{R} \mathbf{L} \Sigma_{0} \mathbf{L}\transpose } + \mathbf{\Tilde{X}}\transpose \mathbf{Q} \mathbf{\Tilde{X}} \nonumber \\
    & ~~ + \tr{\mathbf{Q} (\mathbf{\Gamma} + \mathbf{H_u} \mathbf{L}) \Sigma_0 (\mathbf{\Gamma} + \mathbf{H_u} \mathbf{L})\transpose} \\
    \text{s.t.} & ~~ \mu_d = \Phi_{N, 0} \mu_0 + \mathbf{B_N} \mathbf{\Bar{U}},  \label{eq:Gaussian-CS-constr-mean}\\
    & ~~ \Sigma_d = (\Phi_{N, 0} + \mathbf{B_N} \mathbf{L}) \Sigma_0 (\Phi_{N, 0} + \mathbf{B_N} \mathbf{L})\transpose, \label{eq:Gaussian-CS-constr-cov}
\end{align}
\end{subequations}
where $\Phi_{k_1, k_0} := A_{k_1-1} A_{k_1 - 2} \dots A_{k_0}$ for all $k_1 > k_0$, $\Phi_{k_0, k_0} = I_n$ for all $k_0 \in \Zplus$. 
$\mathbf{X} := \vertcat{x_0, \dots, x_N}$, $\mathbf{\Bar{X}} := \E{\mathbf{X}}$, $\mathbf{\Tilde{X}} := \mathbf{\Bar{X}} - \mathbf{X'}$, $\mathbf{X'} := \vertcat{x_0', \dots, x_N'}$, $\mathbf{U} := \vertcat{u_0, \dots, u_{N-1}}$, $\mathbf{\Bar{U}} := \E{\mathbf{U}} = \vertcat{\Bar{u}_0, \dots, \Bar{u}_{N-1}}$, $\mathbf{Q}:=\bdiag{Q_0, \dots, Q_{N}}$, $\mathbf{R} := \bdiag{R_0, \dots, R_{N-1}}$, $\mathbf{L}:=\vertcat{L_0, \dots, L_{N-1}}$, $\mathbf{B_N}:= [\Phi_{N, 1}B_0, \Phi_{N, 2}B_1, \dots, \Phi_{N, N}B_{N-1}]$. Note that $\mathbf{B_N} \mathbf{B}_{\mathbf{N}}\transpose = \mathcal{G}_{N:0}$ where $\mathcal{G}_{N:0}$ is defined in \eqref{eq:ctrl-gram-def}.
Furthermore, $\mathbf{\Gamma} := \vertcat{\Phi_{0,0}, \Phi_{1, 0}, \dots, \Phi_{N,0}}$ and $\mathbf{H_u}$ is given as:
\begin{align}\label{eq:Hu-def}
    \mathbf{H_u} = \left[ \begin{smallmatrix}
        \bm{0} & \bm{0} & \cdots & \bm{0} \\
        \Phi_{1,1} B_0 & \bm{0} & \cdots & \bm{0} \\
        \Phi_{2, 1} B_{0}& \Phi_{2, 2}B_{1} & \cdots & \bm{0} \\
        \vdots & \vdots &  \ddots & \vdots \\
        \Phi_{N, 1} B_0 & \Phi_{N,2}B_1 & \cdots & \Phi_{N, N}B_{N-1}
    \end{smallmatrix} 
    \right ]
\end{align}

Moreover, the concatenated vectors $\mathbf{X}$ and $\mathbf{U}$ satisfy 
\begin{subequations}\label{eq:concat-dynamics}
\begin{align}
    \mathbf{X}  = \mathbf{\Gamma} x_0 + \mathbf{H_u U}, ~~ \mathbf{U}  = \mathbf{L} \hspace{0.02cm} (x_0 - \mu_0)  + \mathbf{\Bar{U}}.
\end{align}
\end{subequations}

The constraints in \eqref{eq:Gaussian-CS-constr-mean} and \eqref{eq:Gaussian-CS-constr-cov} correspond to the mean and covariance steering constraints. 
Since the state mean depends on $\mathbf{\Bar{U}}$, the state covariance depends on $\mathbf{L}$, and the objective function is separable in $\mathbf{\Bar{U}}$ and $\mathbf{L}$, 
we conclude that the mean and covariance steering problems can be decoupled. 
In particular, the mean steering problem is formulated as:
\begin{align}\label{eq:Mean-Steering-Problem}
    \min_{\mathbf{\Bar{U}}} & ~~ J_{\mathrm{mean}}(\mathbf{\Bar{U}}; \mu_0) :=  \mathbf{\Bar{U}}\transpose \mathbf{R} \mathbf{\Bar{U}} \nonumber \\
    & ~ + (\mathbf{\Gamma}\mu_0 + \mathbf{H_u \Bar{U}} - \mathbf{X'})\transpose \mathbf{Q} \left( \boldsymbol{\Gamma}  \mu_0 + \mathbf{H_u} \mathbf{\Bar{U}} -\mathbf{X'} \right) \\
    \text{s.t.} & ~~ \eqref{eq:Gaussian-CS-constr-mean}. ~ \nonumber
\end{align}
Note that the problem in \eqref{eq:Mean-Steering-Problem} is a strictly convex quadratic program with affine equality constraints (since $R_k \in \S{m}\spr{++}$), whose closed-form solution can be obtained using the KKT conditions \cite{b:boyd_vandenberghe_2004}. The following proposition provides the optimal feed-forward control input $\mathbf{\Bar{U}}$ for the mean steering problem in \eqref{eq:Mean-Steering-Problem}.

\begin{proposition}\label{prop:optimal-mean-steering}
Under Assumption \ref{assumption:linear-controllable}, the optimal control sequence $\mathbf{\Bar{U}}\spr{\star}$ that solves problem \eqref{eq:Mean-Steering-Problem} is given by:
\begin{align}
    \Lambda\spr{\star} = &  2 (\mathbf{B_N} M^{-1} \mathbf{B}_{\mathbf{N}}\transpose)^{-1} \times \nonumber \\
    & ~~ (\mathbf{B_N} M^{-1} \mathbf{H_u}\transpose \mathbf{Q} Y + (\mu_d - \Phi_{N,0} \mu_0 )), \\
    \mathbf{\Bar{U}}\spr{\star} = &  (1/2) M^{-1} (\mathbf{B_N}\transpose \Lambda\spr{\star} - 2 \mathbf{H_u}\transpose \mathbf{Q} Y),
\end{align}
where $M = \mathbf{R} + \mathbf{H_u Q H_u}\transpose$, $Y = \mathbf{\Gamma} \mu_0 - \mathbf{X'}$.
Furthermore, the optimal value of the objective function is given as:
\begin{align}\label{eq:optimal-mean-steer-expr1}
    J_\mathrm{mean}^{\star}(\mu_0, \mu_d) := J_\mathrm{mean}(\mathbf{\Bar{U}}^{\star}; \mu_0) 
\end{align}
\end{proposition}

Next, we formulate the covariance steering problem:
\begin{align}\label{eq:Cov-Steering-Problem}
    \min_{\mathbf{L}} & ~~ J_{\mathrm{cov}}(\mathbf{L}; \Sigma_0) := \tr{\mathbf{R} \mathbf{L}  \Sigma_0 \mathbf{L}\transpose} \nonumber \\
    & ~~+\tr{ \mathbf{Q} (\mathbf{\Gamma} + \mathbf{H_u} \mathbf{L})\Sigma_0 (\mathbf{\Gamma} + \mathbf{H_u} \mathbf{L})\transpose } \\
    \text{s.t.} & ~~ \eqref{eq:Gaussian-CS-constr-cov}, ~ \nonumber
\end{align}
The objective function $J_{\mathrm{cov}}(\mathbf{L}; \Sigma_0)$ of the covariance steering problem given in \eqref{eq:Cov-Steering-Problem} is a convex quadratic function of the decision variable $\mathbf{L}$. 
However, the terminal covariance constraint in \eqref{eq:Gaussian-CS-constr-cov} is a non-convex quadratic equality constraint. The next proposition provides the closed-form solution to problem given in \eqref{eq:Cov-Steering-Problem} in terms of the problem parameters.
\begin{proposition}\label{prop:optimal-covariance-steering}
Under Assumption \ref{assumption:linear-controllable}, the optimal sequence of the feedback controller gains $\mathbf{L}\spr{\star}$ that solves the problem in \eqref{eq:Cov-Steering-Problem} is given by:
\begin{subequations}
\begin{align}
    \mathbf{L}\spr{\star} & = \mathbf{h} + \mathbf{D} Z, \\
    \mathbf{h} & = \mathbf{B}_{\mathbf{N}}\transpose \mathcal{G}_{N:0}\spr{-1} \left( \Sigma_{d}^{1/2} T \Sigma_0^{-1/2} - \Phi_{N,0} \right), \\
    Z & = - \left( \mathbf{D}\transpose M \mathbf{D} \right)^{-1} \mathbf{D}\transpose \left( M \mathbf{h} + \mathbf{H}_{\mathbf{u}}\transpose \mathbf{Q} \mathbf{\Gamma} \right), \\
    T & = - V_{\Omega} U\transpose_{\Omega}, \\
    \Omega & = \Sigma_0^{1/2} \left( \Theta_5\transpose \mathbf{Q} \mathbf{H_u} - \Theta_4\transpose \mathbf{R} \right) \Theta_1 \mathbf{B_N}\transpose \mathcal{G}_{N:0}\spr{-1} \Sigma_{d}\spr{1/2}, \label{eq:omega-def}
\end{align}
\end{subequations}
where $\mathbf{D} \in \R{mN \times mN - n}$ is an arbitrary full-rank matrix whose columns are orthogonal to the range space of $\mathbf{B}_{\mathbf{N}}\transpose$ (i.e., $\mathbf{B_N} \mathbf{D} = \bm{0}$). 
$M = \mathbf{R} + \mathbf{H}_\mathbf{u}\transpose \mathbf{Q} \mathbf{H_u}$, $\Omega = U_{\Omega} \Lambda_{\Omega} V_{\Omega}\transpose$ is the singular value decomposition of matrix $\Omega$, $\Theta_1 := I_{Nm} - \mathbf{D}(\mathbf{D}\transpose M \mathbf{D})^{-1} \mathbf{D}\transpose M$, $\Theta_2 := \mathbf{D}(\mathbf{D}\transpose M \mathbf{D})^{-1} \mathbf{D}\transpose \mathbf{H}_{\mathbf{u}}\transpose \mathbf{Q} \mathbf{\Gamma}$, $\Theta_3 := \mathbf{B}_{\mathbf{N}}\transpose \mathcal{G}_{N:0}\spr{-1} \Phi_{N:0}$, $\Theta_4:= \Theta_1 \Theta_3 + \Theta_2$, $\Theta_5 := \mathbf{\Gamma} - \mathbf{H_u} \Theta_4$.
Furthermore, the optimal value of the objective function is given as:
\begin{align}\label{eq:optimal-cov-steer-expr1}
    J_{\mathrm{cov}}\spr{\star} (\Sigma_0, \Sigma_d) = & ~\tr{ \mathbf{R} \left( \Theta_1 \mathcal{Z} \Theta_1\transpose + \Theta_4 \Sigma_0 \Theta_4\transpose \right) } \nonumber \\ 
    & + \tr{ \mathbf{Q} \left( \mathbf{H_u} \Theta_1 \mathcal{Z} \Theta_1\transpose \mathbf{H}_\mathbf{u}\transpose + \Theta_5 \Sigma_0 \Theta_5\transpose \right) } \nonumber \\
    & - 2 \lVert \Omega \rVert_{*},
\end{align}
where $ \mathcal{Z} := \mathbf{B}_{\mathbf{N}}\transpose \mathcal{G}_{N:0}\spr{-1} \Sigma_d \mathcal{G}_{N:0}\spr{-1} \mathbf{B_N} $.
\end{proposition}
\begin{proof}
Observe that \eqref{eq:Gaussian-CS-constr-cov} can be equivalently written as 
\begin{subequations}
\begin{align}\label{eq:prop-opt-cov-steer-rewritten-constr-cov}
    T = \Sigma_d^{-1/2} (\Phi_{N:0} + \mathbf{B_N} \mathbf{L}) \Sigma_0^{1/2}, ~~  T T\transpose = I_{n}.
\end{align}
Also, let $\mathbf{D} \in \R{mN \times mN - n}$ be a full-rank matrix such that $\mathbf{B_N} \mathbf{D} = \bm{0}$. 
Using $\mathbf{D}$, we can write $\mathbf{L} = \mathbf{B}_{\mathbf{N}}\transpose Y + \mathbf{D} Z$ where $Y\in \R{n \times n}$ and $Z \in \R{Nm - n \times n}$. 
Note that there is a one-to-one mapping between $\mathbf{L}$ and the pair $Y, Z$ since both $\mathbf{B_N}$ and $\mathbf{D}$ are full-rank and have orthogonal columns. 
Thus, we can rewrite \eqref{eq:prop-opt-cov-steer-rewritten-constr-cov} as $T =  \Sigma_d^{-1/2} (\Phi_{N:0} + \mathcal{G}_{N:0}Y)\Sigma_0$. 
Hence, 
\begin{equation}\label{eq:LofTandZ}
    \mathbf{L} = \mathbf{B}_N\transpose \mathcal{G}_{N:0}\spr{-1} \left( \Sigma_d\spr{1/2} T \Sigma_0\spr{-1/2} - \Phi_{N:0}\right) + \mathbf{D} Z.
\end{equation}
and thus, 
the problem in \eqref{eq:Cov-Steering-Problem} can be written as follows:
\begin{align}
    \min_{T \in \mathbf{{T}}, Z} & ~~ J_1(T, Z) 
\end{align}
where $J_1(T, Z) = J_{\mathrm{cov}} \big( \mathbf{B}_N\transpose \mathcal{G}_{N:0}\spr{-1} \big( \Sigma_d\spr{1/2} T \Sigma_0\spr{-1/2} - \Phi_{N:0} \big) + \mathbf{D} Z \big)$, $\mathbf{T} := \{T ~ | ~ T T\transpose = I_{n} \}$. 
Note that the objective function $J_1(T, Z)$ is jointly convex in $(T, Z)$. 
For a fixed $T$, $\min_{Z} J_1(T, Z)$ is an unconstrained convex quadratic program whose (global) minimizer $Z^{\star}(T)$ is given by:
\begin{align}
    Z^{\star}(T) = - (\mathbf{D}\transpose M \mathbf{D})^{-1} \mathbf{D}\transpose (M h(T) + \mathbf{H}_{\mathbf{u}}\transpose \mathbf{Q} \mathbf{\Gamma} ),
\end{align}
where $h(T) = \mathbf{B_N} \mathcal{G}_{N:0}\spr{-1} (\Sigma_d\spr{1/2} T \Sigma_0^{-1/2} - \Phi_{N,0})$.
By plugging the expression of $Z\spr{\star}(T)$ back into $J_{1}(T, Z)$, we obtain the following optimization problem:
\begin{align}
    \min_{T \in \mathbf{T}} & ~~ J_2(T) = J_1(T, Z\spr{\star}(T)).
\end{align}
Expanding $J_2(T)$, we obtain that $J_2(T) = C + 2 \tr{\Omega T} $ where the constant term $C = \tr{ \mathbf{R} \left( \Theta_1 \mathcal{Z} \Theta_1\transpose + \Theta_4 \Sigma_0 \Theta_4\transpose \right) } + \tr{ \mathbf{Q} \left( \mathbf{H_u} \Theta_1 \mathcal{Z} \Theta_1\transpose \mathbf{H}_\mathbf{u}\transpose + \Theta_5 \Sigma_0 \Theta_5\transpose \right) }$. 
Finally, from Von Neuman trace inequality \cite{p:mirsky1975vonNeumanTrace}, we obtain that $T^{\star} = \arg \min_{T T\transpose = I_n} \tr{\Omega T} = - V_{\Omega} U_{\Omega}\transpose$ and $\tr{\Omega T\spr{\star}} = \tr{\Lambda_{\Omega}} = \sum_i \sigma_i(\Omega) = \lVert \Omega \rVert_{*}$. 
\end{subequations}
\end{proof}

The optimal value of the performance index of the covariance steering problem in \eqref{eq:Cov-Steering-Problem}, $J\spr{\star}_\mathrm{cov}( \Sigma_0, \Sigma_d )$, can be alternatively computed by solving an associated SDP.
Before we proceed, we will introduce the following lemmas:
\begin{lemma}\label{lemma:sdp-nuc-norm}
    Let $\Omega \in \R{n \times n}$ be a non-singular matrix.
    Then, $-\lVert \Omega \rVert_*$ is equal to the optimal value of the following SDP: $\min_{L \in \R{n \times n}} \tr{L}$ s.t. $\left[ \begin{smallmatrix}
        \Omega \Omega\transpose & L \\
            L\transpose & I_n
    \end{smallmatrix} \right] \succeq \mathbf{0}$.
\end{lemma}
\begin{proof}
    The constraint of the SDP given in Lemma \ref{lemma:sdp-nuc-norm} can be equivalently written as $\Omega \Omega\transpose - L L\transpose \succeq 0$ using the Schur's complement lemma. 
    By multiplying $\Omega^{-1}$ from the left and  $\Omega^{-\mathrm{T}}$ from the right, we obtain $I_n - \Omega^{-1} L L\transpose \Omega^{- \mathrm{T}} \succeq 0$.
    Then, applying variable transformation $Y = \Omega^{-1} L $, the SDP given in Lemma \ref{lemma:sdp-nuc-norm} can be rewritten as $\min_{Y} \tr{\Omega Y}$ s.t. $ I_n \succeq Y Y\transpose$.
    Finally, from Von Neumann trace inequality \cite{p:mirsky1975vonNeumanTrace}, the (global) minimizer of the latter SDP is given as $Y^\star = -V_{\Omega} U\transpose_{\Omega}$ where $U_\Omega\transpose D_\Omega V_\Omega$ is the SVD decomposition of $\Omega$ and $\tr{\Omega Y\spr{\star}} = -\lVert \Omega \rVert_{*} $.
\end{proof}
\begin{lemma}\label{lemma:spd-equivalent-form}
    Let $A \in \R{n \times n}$ be non-singular and $M \in \S{n}^{++}$. 
    Then, the SDP: $\min_{L \in \R{n \times n}} \tr{L}$ s.t. $\left[ \begin{smallmatrix}
        A M A\transpose & L \\ L\transpose & I_n
    \end{smallmatrix} \right] \succeq \mathbf{0}$ is equivalent to the SDP: $\min_{L \in \R{n \times n}} \tr{L}$ s.t. $\left[ \begin{smallmatrix}
        M & L \\ L\transpose & A A\transpose
    \end{smallmatrix} \right] \succeq \mathbf{0}$.  
\end{lemma}
\begin{proof}
    By applying Schur's complement lemma and multiplying the resulting inequality with $A^{-1}$ from the left and $A^{-\mathrm{T}}$ from the right, we obtain $M - A^{-1} L L\transpose A^{-\mathrm{T}} \succeq 0$. 
    Then, by applying the variable transformation $L := A Y A^{-1}$ and Schur's complement lemma, we obtain the SDP constraint $\left[ \begin{smallmatrix}
        M & L \\
        L\transpose & A A\transpose
    \end{smallmatrix} \right] \succeq 0$. 
    By the cyclic permutation property of the trace operator, the objective function can be rewritten as $\tr{AYA^{-1}} = \tr{Y}$ which concludes the proof.
\end{proof}

The following result follows readily from Lemmas \ref{lemma:sdp-nuc-norm} and \ref{lemma:spd-equivalent-form}.
\begin{corollary}\label{corrolary:alternative-cov-steer-expr}
    The optimal covariance steering cost $J_\mathrm{cov}^\star (\Sigma_0, \Sigma_d)$ defined in \eqref{eq:optimal-cov-steer-expr1} for Problem \ref{prob:covariance-steering} from initial covariance matrix $\Sigma_0 \in \S{n}^{++}$ to $\Sigma_d \in \S{n}^{++}$ for the linear dynamical system \eqref{eq:linear-system-eq} is equal to the optimal value of the following SDP:
    \begin{subequations}\label{eq:alternative-cov-steer-expr}
    \begin{align}
        \min_{L \in \R{n \times n}} & ~ \tr{\Theta_6 \Sigma_d} + \tr{\Theta_7 \Sigma_0}  +  \tr{ L + L\transpose } \label{eq:corollary-alternative-cov-steer-expr-objective}\\
        \text{s.t.} & ~  \begin{bmatrix}
            \Theta_8 \Sigma_d \Theta_8\transpose & L \\
            L\transpose & \Sigma_0
        \end{bmatrix} \succeq \mathbf{0}. \label{eq:corollary-alternative-cov-steer-expr-constr}
    \end{align}
    \end{subequations}
    where $L \in \R{n \times n}$, $\Theta_6 := \mathcal{G}_{N:0}^{-1} \mathbf{B_N} \Theta_1\transpose M \Theta_1 \mathbf{B_N}\transpose  \mathcal{G}_{N:0}^{-\mathrm{T}}$, $M = \mathbf{R}  + \mathbf{H_u} \mathbf{Q} \mathbf{H_u}\transpose$, $\Theta_7 := \Theta_4\transpose \mathbf{R} \Theta_4 + \Theta_5\transpose \mathbf{Q} \Theta_5$ and $\Theta_8 := (\Theta_5 \mathbf{Q} \mathbf{H_u} - \Theta_4\transpose \mathbf{R}) \Theta_1 \mathbf{B_N}\transpose \mathcal{G}_{0:N}^{-1}$.
\end{corollary}
\begin{proof}
    First, expand the term $\mathcal{Z}$ which is defined in Proposition \ref{prop:optimal-covariance-steering}. 
    By using the cyclic permutation property of the trace operator, we obtain the first two terms in the objective function in \eqref{eq:corollary-alternative-cov-steer-expr-objective}. 
    Then, observe that $\Omega = \Sigma_0^{1/2} \Theta_8 \Omega_d^{1/2}$ where $\Omega$ is defined in \eqref{eq:omega-def}. 
    Lemma \ref{lemma:sdp-nuc-norm} implies that the third term $-2 \lVert \Omega \rVert_*$ is equal to the optimal value of the optimization problem: $\min_{L} 2\tr{L}$ s.t. $\left[ \begin{smallmatrix}
        \Theta_8 \Sigma_d \Theta_8\transpose & L \\ L\transpose & I_n
    \end{smallmatrix} \right] \succeq 0$. 
    Finally, applying Lemma \ref{lemma:spd-equivalent-form} to the obtained SDP and using the equality $\tr{L} = \tr{L\transpose}$, we obtain \eqref{eq:alternative-cov-steer-expr}.
\end{proof}

The SDP in Corollary \ref{corrolary:alternative-cov-steer-expr} will be used to formulate Problems \ref{prob:soft-GMM-density-steering}-\ref{prob:step-cost-constrained-GMM-density-steering} as finite-dimensional optimization problems in Section \ref{s:formulation-nlp}.

\begin{remark}
In the special case in which $\mathbf{Q} = 0$, the problem in \eqref{eq:Cov-Steering-Problem} is equivalent to the covariance steering problem studied in \cite{p:goldshtein2017finitehorizonCS}.
Thus, the optimal policy derived in Proposition \ref{prop:optimal-covariance-steering} is the same optimal policy defined in \cite[Eq. (20)]{p:goldshtein2017finitehorizonCS}.
\end{remark}

\section{GMM Steering Policies}\label{s:gmm-steering-policies}
To reduce Problems \ref{prob:hard-GMM-density-steering}-\ref{prob:step-cost-constrained-GMM-density-steering} into tractable finite-dimensional optimization problems, we propose an admissible set of control policies $\Pi_{r} \subset \Pi$ consisting of randomized control policies, where each $\pi \in \Pi_r$ is a sequence of  control laws, $\curly{\pi_{0}, \pi_{1} , \dots, \pi_{N-1} }$, such that each $\pi_{k} : \R{n} \rightarrow \mathcal{P} (\R{m})$, with
\begin{align}\label{eq:gmm-policy-definition}
    \pi_{k} (x_0) =  L_k^{i, j} (x_0 - \Bar{\mu}^{i}_0) + \Bar{u}_k^{i, j} ~~ \text{w.p.} ~~ \gamma_{i, j}(x_0)
\end{align}
where $L_{k}\spr{i, j} \in \R{m \times n}$, $\Bar{u}_k\spr{i, j} \in \R{m}$ for all $k \in \curly{0, \dots, N-1}, i \in \curly{0, \dots, r-1}$ and $j \in \curly{0, \dots, t-1}$.
Furthermore, $\gamma_{i, j} : \R{n} \rightarrow \R{} $ is given by:
\begin{align}\label{eq:gamma-def}
\gamma_{i, j} (x_0) = \lambda_{i, j} \ell_{i}(x_0),
\end{align}
where $\ell_i(x_0) := \frac{p_i \P[\mathcal{N}]{x_0; \Bar{\mu}_0\spr{i}, \Bar{\Sigma}_0\spr{i}}}{\sum_{i=0}\spr{r-1} p_i \P[\mathcal{N}]{x_0 ; \Bar{\mu}_0^{i}, \Bar{\Sigma}_0\spr{i}}}$, $\sum_{j} \lambda_{i, j} = 1$ $\forall i$ and $\lambda_{i, j}, \geq 0$ $\forall (i,j)$. Note that $\sum_{i, j} \gamma_{i,j} = 1$ and $\gamma_{i, j} \geq 0$ $\forall (i,j)$, and in addition, the policy $\pi$ determines a probability distribution over the control sequences for a given $x_0$.

The policy defined in \eqref{eq:gmm-policy-definition} is based on the intuition that under affine state feedback control policies, the state distribution $x_k$ will remain Gaussian, provided that the initial state distribution is also Gaussian. 
The term $\ell_i(x_0)$ in \eqref{eq:gamma-def} represents the likelihood that the initial state $x_0$ will be drawn from the $i$th component of the GMM with components $( \curly{p_i, \Bar{\mu}_i, \Bar{\Sigma}_i}\sub{i=0}\spr{r-1})$.
Consequently, the control policy in \eqref{eq:gmm-policy-definition} is a valid choice for steering probability distributions described by GMMs.
The following proposition states that if a policy $\pi \in \Pi_r$ is applied to the dynamical system in \eqref{eq:linear-system-eq} whose initial state is sampled from $\mathbf{GMM}(\curly{p_i, \mu^{0}_i, \Sigma^{0}_i}_{i=0}\spr{r-1})$, then the terminal state $x_N \sim \mathbf{GMM}(\curly{q_i, \mu_i\spr{f}, \Sigma_i\spr{f}}_{i=0}\spr{t-1})$ whose parameters are determined by the policy $\pi \in \Pi_r$.

\begin{proposition}\label{prop:gmm-steering-proposition}
Let $x_0 \in \R{n}$ be the initial state of the system given in \eqref{eq:linear-system-eq} such that $x_0 \sim \mathbf{GMM}(\curly{p_i, \mu_i\spr{0}, \Sigma_i\spr{0}}_{i=0}\spr{r-1})$ and $\pi \in \Pi_r$ with parameters $(\curly{\Bar{\mu}_i, \Bar{\Sigma}_i}_{i=0}\spr{r-1},$ $\curly{\lambda_{i, j}}_{i=0, j=0}\spr{r-1, t-1},$ $\curly{\Bar{u}\sub{k}\spr{i,j}, L\sub{k}\spr{i,j}}_{i=0, j=0, k=0}\spr{r-1, t-1, N-1})$ and
$\mu_i\spr{0} = \Bar{\mu}_{i}$, $\Sigma_i\spr{0} = \Bar{\Sigma}_i$ $\forall i \in \curly{0, \dots, r-1}$. 
Furthermore, let $u_k = \pi_k(x_0)$ $\forall k \in \curly{0, \dots, N-1}$, then $x_N \sim \mathbf{GMM}(\curly{q_j, \mu_j\spr{f}, \Sigma_j\spr{f}})$ such that 
\begin{subequations}
\begin{align}
    q_j & = \sum_{i=0}\spr{r-1} p_i \lambda_{i, j} \label{eq:prop-gmm-steer-eq1}, \\
    \mu\spr{f}_{j} & = \Phi_{N:0} \mu_{i}\spr{0} + \mathbf{B_N} \mathbf{\Bar{U}}\sub{i, j},  \label{eq:prop-gmm-steer-eq2} \\
    \Sigma\spr{f}_{j} & = (\Phi_{N:0} + \mathbf{B_N} \mathbf{L}\sub{i, j}) \Sigma_{i}\spr{0} (\Phi_{N:0} + \mathbf{B_N} \mathbf{L}\sub{i, j})\transpose, \label{eq:prop-gmm-steer-eq3}
\end{align}
\end{subequations}
where \eqref{eq:prop-gmm-steer-eq1}-\eqref{eq:prop-gmm-steer-eq3} hold  for all $j \in \curly{0, \dots, t-1}$, $\mathbf{L}_{i, j} = \vertcat{L_0\spr{i, j}, \dots L_{N-1}^{i,j}}$ and $\mathbf{\Bar{U}}_{i, j} = \vertcat{\Bar{u}_0\spr{i, j}, \dots, \Bar{u}_{N-1}\spr{i, j}}$.
\end{proposition}

\begin{proof}
By virtue of Bayes' Theorem on conditional probability densities, the PDF of $x_N$ can be written as follows:
\begin{align}\label{eq:pdf-x-n}
    \P[x_N]{\Hat{x}} & =  \int_{\R{n}} \int_{\R{m N}}  \P[x_N | x_0 = \Hat{x}_0, U = \Hat{U}]{\Hat{x}} ~ \P[U| x_0=\Hat{x}]{\Hat{U}}  \nonumber \\
    & \qquad \qquad \qquad \times\P[x_0]{\Hat{x}_0 }  \, \mathrm{d}\Hat{U} \, \mathrm{d}\Hat{x}_0 .
\end{align}
Furthermore, the conditional PDFs in \eqref{eq:pdf-x-n} are given by:
\begin{align}
    & \P[x_N | x_0  = \Hat{x}_0, U=\Hat{U}]{\Hat{x}} = \delta (\Hat{x} = \Phi_{N:0} \Hat{x}_0 + \mathbf{B_N} \Hat{U}) \label{eq:conditional-xn-given-x0-u} \\
    & \qquad ~ \, \P[U | x_0 = \Hat{x}_0]{\Hat{U}} = \nonumber \\ 
    &  \quad \sum_{i=0}\spr{r-1} \sum_{j=0}\spr{t-1} \gamma_{i, j} (\Hat{x}_0) \delta ( \Hat{U} = \mathbf{L}_{i, j} (\Hat{x}_0 - \mu_i\spr{0}) + \mathbf{\Bar{U}}\sub{i, j} )  \label{eq:conditional-U-given-x0}
\end{align}
where $\delta(\cdot)$ denotes the Dirac delta function. Eq.~\eqref{eq:conditional-xn-given-x0-u} can be derived readily from the system dynamics \eqref{eq:linear-system-eq} whereas \eqref{eq:conditional-U-given-x0} is a direct consequence of the definition of the policy set $\Pi_r$ given in \eqref{eq:gmm-policy-definition}. 
Let us now analyze the inner integral in \eqref{eq:pdf-x-n}. We observe that $\P[x_0]{\Hat{x}_0}$ does not depend on $\Hat{U}$ and thus can be factored out of the inner integral, which can be written as:
\begin{align}\label{eq:inner-integral}
    l(\Hat{x}_0) := \int_{\R{mN}} \P[x_N | x_0 = \Hat{x}_0, U=\Hat{U}]{\Hat{x}} \P[U|x_0 = \Hat{x}]{\Hat{U}} \mathrm{d} \Hat{U}
\end{align}
Now, plug the expressions in \eqref{eq:conditional-xn-given-x0-u} and \eqref{eq:conditional-U-given-x0} into \eqref{eq:inner-integral} to obtain:
\begin{align}
    l(\Hat{x}_0)& = \int_{\R{mN}} \delta (\Hat{x} = \Phi_{N:0} \Hat{x}_0 + \mathbf{B_N}\Hat{U}) \times \nonumber \\ 
    & \hspace{-0.1cm} \sum_{i=0, j=0}\spr{r-1, t-1} \gamma_{i, j} (\Hat{x}_0) \delta ( \Hat{U} = \mathbf{L}_{i, j} (\Hat{x}_0 - \mu_i\spr{0}) + \mathbf{\Bar{U}}\sub{i, j} ) \, \mathrm{d}\Hat{U} \\
    & = \sum_{i=0, j=0}\spr{r-1, t-1} \gamma_{i, j} (\Hat{x}_0) \delta \big(\Hat{x} = \Phi_{N:0} \Hat{x}_0 \nonumber \\
    &  \qquad \qquad + \mathbf{B_N} (\mathbf{L}_{i,j} (\Hat{x}_0 - \mu_i\spr{0}) + \mathbf{\Bar{U}}_{i, j}) \big). \label{eq:inner-integral-resolved}
\end{align}
Equation~\eqref{eq:inner-integral-resolved} is obtained by using the linearity of the integral operator and the properties of the Dirac delta function. 
The expression in \eqref{eq:inner-integral-resolved} can be rewritten as $\sum_{i, j} \gamma_{i, j}(\Hat{x}_0) \delta (\mathbf{H}_{i, j} \Hat{x}_0 - h_{i, j} = \Hat{x})$ for brevity, where 
\begin{align}
    \mathbf{H}_{i, j} & := \Phi_{N:0} + \mathbf{B_N} \mathbf{L}_{i, j}, \label{eq:H_ij_def} \\
    h_{i, j} &:=  \mathbf{B_N} (\mathbf{L}_{i, j} \mu_{i}\spr{0} - \mathbf{\Bar{U}}_{i, j}). \label{eq:h_ij_def}
\end{align}
Observe that the denominator of $\gamma_{i, j}(\Hat{x}_0)$ defined in \eqref{eq:gamma-def} is equal to $\P[x_0]{\Hat{x}_0}$. By using this fact, it follows that 
\begin{align}\label{eq:prop-gmm-proof-PxN}
    & \P[x_N]{\Hat{x}} = \sum_{i=0}\spr{r-1} \sum_{j=0}\spr{t-1} p_i \lambda_{i, j} g_{i, j} (\Hat{x}),
\end{align}
where 
\begin{align}
    \hspace{-0.20cm} g_{i, j} (\Hat{x}) & := \int_{\R{n}} \P[\mathcal{N}]{\Hat{x}_0; \Bar{\mu}_{0}\spr{i}, \Bar{\Sigma}_{0}\spr{i}} \delta(\mathbf{H}_{i, j} \Hat{x}_0 - h_{i, j} = \Hat{x})  \mathrm{d} \Hat{x}_0 \label{eq:prop-gmm-proof-gij-0}, \\
    & = \int_{\R{n}} \P[\mathcal{N}]{\mathbf{H}_{i, j}^{-1}(z_{i, j} + h_{i, j}) ; \Bar{\mu}_0\spr{i}, \Bar{\Sigma}_0\spr{i}} \nonumber \\
    & \qquad  \times \det (\mathbf{H}_{i, j}\spr{-1})  \delta (z_{i,j} = \Hat{x}) \mathrm{d} z_{i, j}, \label{eq:prop-gmm-proof-gij-1}\\
    & = \P[\mathcal{N}]{\mathbf{H}_{i,j}\spr{-1} (\Hat{x} + h_{i,j});\Bar{\mu}_0\spr{i}, \Bar{\Sigma}_{0}\spr{i}}  \det (\mathbf{H}_{i, j}^{-1} ), \label{eq:prop-gmm-proof-gij-2}\\
    & = \P[\mathcal{N}]{\Hat{x}; \mathbf{H}_{i,j} \mu_{i}\spr{0} - h_{i,j}, \mathbf{H}_{i,j} \Sigma_i\spr{0} \mathbf{H}_{i,j}\transpose} , \label{eq:prop-gmm-proof-gij-3}
\end{align}
Equation \eqref{eq:prop-gmm-proof-gij-1} is obtained by applying the variable transformation $z_{i, j} = \mathbf{H}_{i,j} \Hat{x}_0 - h_{i, j} $ to \eqref{eq:prop-gmm-proof-gij-0}.
Then, the standard property of the Dirac delta function yields \eqref{eq:prop-gmm-proof-gij-2}.
Expanding $\P[\mathcal{N}]{\mathbf{H}_{i,j}\spr{-1} (\Hat{x} + h_{i,j});\Bar{\mu}_0\spr{i}, \Bar{\Sigma}_{0}\spr{i}}$ leads to \eqref{eq:prop-gmm-proof-gij-3}.
Consequently, we expand $\mathbf{H}_{i,j}$ in \eqref{eq:H_ij_def}, $h_{i,j}$ in \eqref{eq:h_ij_def} and define \eqref{eq:prop-gmm-steer-eq2} and \eqref{eq:prop-gmm-steer-eq3} to obtain $g_{i,j}(\Hat{x}) := \P[\mathcal{N}]{\Hat{x}; \mu_j\spr{f}, \Sigma_j\spr{f}}$.
Finally, plugging $g_{i,j}(\Hat{x})$ into \eqref{eq:prop-gmm-proof-PxN} and defining \eqref{eq:prop-gmm-steer-eq1} concludes the proof.
\end{proof}

\begin{remark}
In the statement of Proposition \ref{prop:gmm-steering-proposition}, it is given that for each $j \in \curly{0, \dots, t-1}$ every $\mathbf{\Bar{U}}_{i, j}$ should satisfy \eqref{eq:prop-gmm-steer-eq2} and every $\mathbf{L}_{i, j}$ should satisfy \eqref{eq:prop-gmm-steer-eq3}. 
This seems like an extra condition for the policy to satisfy and limits the applicability of the policy proposed in \eqref{eq:gmm-policy-definition}. 
However, the number of terminal Gaussian components $t$ is actually determined by the parameters $\mathbf{\Bar{U}}_{i, j}$ and $ \mathbf{L}_{i,j}$.
To see this, take an arbitrary set $\curly{\mathbf{\Bar{U}}_\ell, \mathbf{L}_\ell }_{\ell=0}\spr{s}$ for each $i \in \curly{0, \dots, r-1}$ and $\ell \in \curly{0, \dots, s}$, define $\mu_{i\times s + \ell}\spr{f} = \Phi_{N:0} \mu_i\spr{0} + \mathbf{B_N} \mathbf{\Bar{U}}_{\ell}$ and $\Sigma_{i\times s + \ell}\spr{f} = (\Phi_{N:0} + \mathbf{B_N} \mathbf{L}_{\ell}) \Sigma_{i}\spr{0} (\Phi_{N:0} + \mathbf{B_N} \mathbf{L}_{\ell})\transpose$. 
Thus, we can set $t = s \times r$ to obtain \eqref{eq:prop-gmm-steer-eq2} and \eqref{eq:prop-gmm-steer-eq3}.
\end{remark}

\section{Finite-Dimensional Optimization Problem Formulations of GMM Steering Problems}\label{s:formulation-nlp}
Since we have demonstrated that, within the set of policies $\Pi_r$ defined in \eqref{eq:gmm-policy-definition}, the initial GMM state distribution is transformed into another GMM, and these policies in $\Pi_r$ are parameterized by a finite number of decision variables, we can utilize the set of policies $\Pi_r$ to formulate finite-dimensional optimization problems whose solutions will allow us to solve Problems \ref{prob:hard-GMM-density-steering}-\ref{prob:step-cost-constrained-GMM-density-steering}.
A graphical illustration of Problems \ref{prob:soft-GMM-density-steering}
-\ref{prob:step-cost-constrained-GMM-density-steering} under the class of policies defined in \eqref{eq:gmm-policy-definition} is given in Figure \ref{fig:problem-illustration}.

\subsection{Reduction of Unconstrained Problem to a Linear Program}\label{ss:reduction-lp}
Problem \ref{prob:hard-GMM-density-steering} under the admissible randomized control policies defined in \eqref{eq:gmm-policy-definition} corresponds to the following finite-dimensional NLP with decision variables $\mathcal{S}_1 := \curly{\bm{\lambda}, \curly{ \mathbf{\Bar{U}}_{i,j}, \mathbf{L}_{i,j}}_{i=0, j=0}\spr{r-1, t-1} }$:
\begin{subequations}\label{eq:nlp-hard-GMM-density-steering}
\begin{align}
    \min_{ \mathcal{S}_1 } & ~ \mathcal{J}_1 ( \mathcal{S}_1 ) \\
    \text{s.t.} & ~ \bm{\lambda}_{i, :} \in \Delta_t \label{eq:nlp-hard-GMM-density-steering-constr-0}\\
    & ~ \sum_{i=0}\spr{r-1} p_i\spr{0} \lambda_{i,j} = p_j\spr{d},  \label{eq:nlp-hard-GMM-density-steering-constr-2}\\
    & ~ \mu_j\spr{d} = \Phi_{N:0} \mu_{i}\spr{0} + \mathbf{B_N} \mathbf{\Bar{U}}\sub{i, j}, 
    ~ \Sigma_j\spr{d} = \mathbf{H}_{i,j} \Sigma_{i}\spr{0} \mathbf{H}_{i,j}\transpose,  \label{eq:nlp-hard-GMM-density-steering-constr-34}
\end{align}
\end{subequations}
where $\bm{\lambda} \in \R{r \times t}$, $\mathbf{\Bar{U}}_{i, j} \in \R{mN}$, $\mathbf{L}_{i, j} \in \R{mN \times n}$, $\bm{\lambda}_{i, :} = \vertcat{\lambda_{i, 0}, \dots, \lambda_{i, t-1}}$ for all $i \in \curly{0, \dots, r-1}$, and $\lambda_{i, j}$ is the $(i, j)$ entry of $\bm{\lambda}$.
Furthermore, $\mathbf{H}_{i, j}$ is given in \eqref{eq:H_ij_def}, $\mathcal{J}_1( \mathcal{S}_1 ) = \E{J(X_{0:N}, U_{0:N-1})}$.
The constraint in \eqref{eq:nlp-hard-GMM-density-steering-constr-0} is due to the parametrization of the control policy in \eqref{eq:gmm-policy-definition}. The constraints in \eqref{eq:nlp-hard-GMM-density-steering-constr-2} and \eqref{eq:nlp-hard-GMM-density-steering-constr-34} 
are obtained by making the right hand side of \eqref{eq:prop-gmm-steer-eq1}, \eqref{eq:prop-gmm-steer-eq2} and \eqref{eq:prop-gmm-steer-eq3} equal to $p_j\spr{d}$, $\mu\sub{j}\spr{d}$ and $\Sigma_j\spr{d}$, respectively.
Furthermore, constraints \eqref{eq:nlp-hard-GMM-density-steering-constr-0} and \eqref{eq:nlp-hard-GMM-density-steering-constr-2} are enforced for all $i \in \curly{0, \dots, r-1}$ and for all $j \in \curly{0, \dots, t-1}$, respectively.  
Constraints \eqref{eq:nlp-hard-GMM-density-steering-constr-34} 
are enforced for all $i, j$ in $\curly{0, \dots, r-1} \times \curly{0, \dots, t-1}$.
By using the law of iterated expectations, the objective function $\mathcal{J}_1 (\mathcal{S}_1)$ can be written as:
\begin{align}\label{eq:nlp-hard-GMM-density-steering-objective}
    \mathcal{J}_1 (\mathcal{S}_1) := \sum_{i=0}^{r-1} \sum_{j=0}^{t-1} p_i^0 \lambda_{i,j} (J\spr{i}_\mathrm{mean}(\mathbf{\Bar{U}}_{i,j}) + J\spr{i}_\mathrm{cov}(\mathbf{L}_{i,j}))
\end{align}
where $J_\mathrm{mean}\spr{i}(\mathbf{\Bar{U}}_{i, j}) := J_\mathrm{mean}(\mathbf{\Bar{U}}_{i, j}; \mu_0^i)$ and $J\spr{i}_\mathrm{cov}(\mathbf{L}_{i, j}) := J_\mathrm{cov}(\mathbf{L}_{i, j}; \Sigma_0^i)$, where 
$J_\mathrm{mean}(\cdot)$ and $J_\mathrm{cov}(\cdot)$ are defined as in \eqref{eq:Mean-Steering-Problem} and \eqref{eq:Cov-Steering-Problem}, respectively.
The parameters $\curly{\Bar{\mu}_{0}\spr{i}, \Bar{\Sigma}_{0}\spr{i}}_{i=0}\spr{r-1}$ are also decision variables for the randomized policy in \eqref{eq:gmm-policy-definition}. 
However, we take them to be constant and equal to $\curly{\mu_i\spr{0}, \Sigma_i\spr{0}}$ to invoke Proposition \ref{prop:gmm-steering-proposition} and formulate the NLP in \eqref{eq:nlp-hard-GMM-density-steering}.

The non-convexity of the NLP in \eqref{eq:nlp-hard-GMM-density-steering} is due to the non-convexity of the objective function $\mathcal{J}_1 (\mathcal{S}_1)$ and the equality constraint \eqref{eq:nlp-hard-GMM-density-steering-constr-34}. 
We observe that when $\bm{\lambda}$ is fixed, the objective function in \eqref{eq:nlp-hard-GMM-density-steering-objective} becomes separable for each $(i,j)$ pair. 
These separated optimization problems for all $(i,j)$ are linear Gaussian covariance steering problems, and optimal policies for each one of them can be found by invoking Propositions \ref{prop:optimal-mean-steering} and \ref{prop:optimal-covariance-steering}. 
The following theorem summarizes the main result of this section and describes how the optimal policy can be extracted from the solution to the LP in \eqref{eq:hardLP}.

\begin{theorem}\label{theorem:hardtoLP}
The optimal parameters of the policy $\pi \in \Pi_r$ given in \eqref{eq:gmm-policy-definition} that solves Problem \ref{prob:hard-GMM-density-steering} can be obtained by solving the following LP:
\begin{subequations}\label{eq:hardLP}
\begin{align}
    \min_{\boldsymbol{\Tilde{\lambda}} \in \R{r\times t}} & ~~ \tr{ \mathbf{C}\transpose \boldsymbol{\Tilde{\lambda}} }  \label{eq:hardLP-obj}\\
    \text{s.t.} & ~~ \boldsymbol{\Tilde{\lambda}} \mathbf{1}_t = \mathbf{p}_0, ~ \boldsymbol{\Tilde{\lambda}}\transpose \mathbf{1}_r = \mathbf{p}_d, ~ \boldsymbol{\Tilde{\lambda}} \geq 0, \label{eq:hardLP-constr}
\end{align}
\end{subequations}
where $\mathbf{C}_{i,j} := J\spr{\star}_{\mathrm{mean}}(\mu\spr{0}_{i}, \mu_{j}\spr{d}) + J\spr{\star}_{\mathrm{cov}}(\Sigma_i\spr{0}, \Sigma\spr{d}_{j} )$, is the $(i, j)$ entry of $\mathbf{C}$, $J_{\mathrm{mean}}\spr{\star}(\mu_i\spr{0}, \mu_j\spr{d})$ and $J_{\mathrm{cov}}\spr{\star}(\Sigma_i\spr{0}, \Sigma_j\spr{d})$ denote the optimal values of the optimal mean and covariance steering problems with $x_0 \sim \mathcal{N}(\mu_i\spr{0}, \Sigma_i\spr{0})$ and $x_N \sim \mathcal{N} (\mu_j\spr{d}, \Sigma_j\spr{d})$. 
Furthermore, the optimal mixing weights $\lambda_{i,j}\spr{\star}$ of the policy $\pi \in \Pi_r$ are given by:
$\lambda_{i,j}\spr{\star} = \frac{\Tilde{\lambda}_{i,j}}{p_i\spr{0}}$,
where $\Tilde{\lambda}_{i,j}\spr{\star}$ is the optimal solution of the LP in \eqref{eq:hardLP}.
\end{theorem}

The formal proof of Theorem \ref{theorem:hardtoLP} is omitted since the reduction of Problem \ref{prob:hard-GMM-density-steering} to the LP defined in \eqref{eq:hardLP} has already been described in detail in Section \ref{ss:reduction-lp}.
\vspace{-0.3cm}

\subsection{Reduction of the Soft Constrained Problem to an NLP}\label{ss:reduction-soft-nlp}
Solving Problem \ref{prob:soft-GMM-density-steering} over randomized policies defined in \eqref{eq:gmm-policy-definition} can be done similarly to the NLP formulation presented in \eqref{eq:nlp-hard-GMM-density-steering}. 
The set of decision variables for this minimization problem is denoted by $\mathcal{S}_{2} := \big( \curly{p_i^N, \mu_i^N, \Sigma_i^N}_{i=0}^{q-1},$ $\boldsymbol{\beta},$ $\curly{\lambda_{i, j}, \mathbf{\Bar{U}}_{i, j},$ $ \mathbf{L}_{i, j}  }_{i=0, j=0}^{r-1, q-1} \big)$.
The resulting NLP is presented as follows:
\begin{subequations}\label{eq:nlp-soft-GMM-density-steering}
\begin{align}
    \min_{\mathcal{S}_2 } & ~ \mathcal{J}_2(\mathcal{S}_2) ~~\text{s.t.} ~~ \eqref{eq:nlp-hard-GMM-density-steering-constr-0} \\
    & ~ \sum_{i=0}^{r-1} p_i^0 \lambda_{i, j} = p_j^N , \label{eq:nlp-soft-GMM-density-steering-constr-0}\\
    & ~ \mu_j^N = \Phi_{N:0} \mu_i^0 + \mathbf{B_N} \mathbf{\Bar{U}}_{i, j}, \label{eq:nlp-soft-GMM-density-steering-constr-1} \\
    & ~ \Sigma_{j}^N = \mathbf{H}_{i, j} \Sigma_i^0 \mathbf{H}_{i, j}\transpose,  \label{eq:nlp-soft-GMM-density-steering-constr-2} \\
    & ~ \boldsymbol{\beta}\mathbf{1}_t = \mathbf{p}_N, ~ \boldsymbol{\beta}\transpose\mathbf{1}_q = \mathbf{p}_{d} , ~ \boldsymbol{\beta} \geq 0,\label{eq:nlp-soft-GMM-density-steering-constr-3}
\end{align}
\end{subequations}
where $\mathbf{H}_{i, j}$ is given in \eqref{eq:H_ij_def}, $\mathbf{p}_N := [p_0^N, \dots, p_{q-1}^N]$, and $\mathbf{p}_\mathrm{d} := [p_0^\mathrm{d}, \dots, p_{t-1}^{\mathrm{d}}]$.
The constraints in \eqref{eq:nlp-soft-GMM-density-steering-constr-0}, \eqref{eq:nlp-soft-GMM-density-steering-constr-1} and \eqref{eq:nlp-soft-GMM-density-steering-constr-2} are obtained from the relationship between the terminal state GMM distribution and the policy parameters. 
The constraint in \eqref{eq:nlp-soft-GMM-density-steering-constr-3} is due to the expression of the Wasserstein-GMM distance given in Definition \ref{def:gmm-wasserstein}.
In this formulation, the number of mixture components of the GMM corresponding to the terminal state distribution is fixed and equal to $q \in \Zplus$.
Thus, the constraint \eqref{eq:nlp-hard-GMM-density-steering-constr-0} is imposed for all $i \in \curly{0, \dots, r-1}$. 
Furthermore, constraints \eqref{eq:nlp-soft-GMM-density-steering-constr-1}, \eqref{eq:nlp-soft-GMM-density-steering-constr-2} are imposed for all $(i, j) \in \curly{0, \dots, r-1} \times \curly{0, \dots, q-1}$ pairs. 
Additionally, the objective function $\mathcal{J}_2(\mathcal{S}_2)$ is written as follows:
\begin{align}\label{eq:nlp-soft-GMM-density-steering-objective}
    \mathcal{J}_2(\mathcal{S}_2) := & ~  \sum_{i=0}^{r-1} \sum_{j=0}^{q-1}  p_i^0 \lambda_{i, j} (J^i_\mathrm{mean}(\mathbf{\Bar{U}}_{i, j}) + J^i_\mathrm{cov}(\mathbf{L}_{i, j})) \nonumber \\
    & ~ + \kappa \sum_{i=0}^{q-1} \sum_{j=0}^{r-1} \beta_{i, j}  W_\mathcal{N}^2(\mu_i^0, \Sigma_i^0, \mu_j^N, \Sigma^N_j) 
\end{align}
where $W_\mathcal{N}^2(\mu_i^0, \Sigma_i^0, \mu_j^N, \Sigma^N_j)$ denotes the squared Wasserstein distance between the PDFs $\rho_i^0(x) := \P[\mathcal{N}]{x; \mu_i^0, \Sigma_i^0}$, $\rho_j^N(x) := \P[\mathcal{N}]{x; \mu_j^N, \Sigma_j^N}$ whose closed-form expression is given in \eqref{eq:wasserstein-gaussian}. 
We were able to use the latter expression in \eqref{eq:wasserstein-gaussian} since the state distribution remains GMM under the policy given in \eqref{eq:gmm-policy-definition}.
Note that the $\beta_{i, j}$ variables appear in the objective function since they are used to evaluate the squared GMM-Wasserstein distance given in Definition \ref{def:gmm-wasserstein}.

Similar to the problem given in \eqref{eq:nlp-hard-GMM-density-steering}, the problem in \eqref{eq:nlp-soft-GMM-density-steering} is non-convex due to the non-convexity of the objective function in \eqref{eq:nlp-soft-GMM-density-steering-objective} and the terminal covariance assignment constraints in \eqref{eq:nlp-soft-GMM-density-steering-constr-2}. 
We observe, however, that by fixing the variables $\lambda_{i, j}$, $\beta_{i, j}$, $\mu_j^N$, and $\Sigma_j^N$, the second term in the objective function becomes constant. 
Consequently, the problem in \eqref{eq:nlp-soft-GMM-density-steering} can be decoupled into a mean steering problem and a covariance steering problem for each pair $(i, j) \in \curly{0, \dots, r-1} \times \curly{0, \dots, q-1}$. 
By solving each of these individual covariance steering problems, we can eliminate constraints \eqref{eq:nlp-soft-GMM-density-steering-constr-2} and \eqref{eq:nlp-soft-GMM-density-steering-constr-3}, along with the decision variables $\curly{\mathbf{\bar{U}}_{i, j},\mathbf{L}_{i, j}}_{i=0, j=0}^{r-1, q-1}$. We denote the new decision variable as $\mathcal{S}_2^{\prime}$. Thus, the objective function in \eqref{eq:nlp-soft-GMM-density-steering-objective} can be rewritten in terms of $\mathcal{S}_2^\prime$ as follows:
\begin{align}\label{eq:nlp-soft-GMM-density-steering-objective2}
    \mathcal{J}_3(\mathcal{S}_2^{\prime}) := & \sum_{i=0}^{r-1}\sum_{j=0}^{q-1} p^0_i \lambda_{i, j} \big( J^\star_\mathrm{mean}(\mu_i^0, \mu_j^N) + J^\star_\mathrm{cov}(\Sigma_i^0, \Sigma_j^N) \big) \nonumber \\
    & + \kappa \sum_{i=0}^{q-1}\sum_{j=0}^{t-1} \beta_{i, j} W_\mathcal{N}^2 (\mu_i^N, \Sigma_i^N, \mu_j^d, \Sigma_j^d).
\end{align}
Note that $J^\star_\mathrm{mean}(\mu_x, \mu_y)$ is a convex quadratic function and $J^\star_\mathrm{cov}(\Sigma_x, \Sigma_y)$ is a convex function.
Furthermore, $W_\mathcal{N}^2(\mu_x, \Sigma_x, \mu_y, \Sigma_y)$ is also a convex function. 
Then, to isolate the non-convex terms in the objective function, we define additional decision variables $\curly{C_{i, j}}_{i=0, j=0}^{r-1, q-1}$ and $\curly{T_{i, j}}_{i=0, j=0}^{q-1, t-1}$ and enforce the constraints $C_{i, j} = J^\star_\mathrm{mean}(\mu^0_i, \mu^N_j) + J^\star_\mathrm{cov}(\Sigma^0_i, \Sigma^N_j)$ and $T_{i, j} = W_\mathcal{N}^2(\mu_i^N, \Sigma_i^N, \mu_j^d, \Sigma_j^d)$.
Finally, relaxing these equality constraints into inequality constraints yields an optimization problem with convex constraints and a bilinear objective function. 
The following theorem presents the final form of the optimization problem. 
\begin{theorem}\label{theorem:soft-to-bilinear}
    The optimal parameters of the policy $\pi \in \Pi_r$ given in \eqref{eq:gmm-policy-definition} that solves Problem \ref{prob:soft-GMM-density-steering} can be obtained by solving the following bilinear program over $\mathcal{S}_{\mathrm{soft}} = (\boldsymbol{\Tilde{\lambda}}, \mathbf{C}, \boldsymbol{\beta}, \mathbf{T}, \mathbf{p}_{N}, \curly{L_{i, j}}_{i=0, j=0}^{r-1, q-1}, \curly{Y_{i, j}}_{i=0, j=0}^{q-1, t-1}, \curly{\mu_i^{N}, \Sigma_i^N}_{i=0}^{q-1})$:
    \begin{subequations}\label{eq:softBilinear}
    \begin{align}
        \min_{\mathcal{S}_\mathrm{soft}} & ~ \mathrm{tr} \big( \mathbf{C}\transpose \boldsymbol{\Tilde{\lambda}} \big) + \kappa \tr{\mathbf{T}\transpose \boldsymbol{\beta}} \label{eq:softBilinear-objective}\\
        \mathrm{s.t.} & ~  \boldsymbol{\Tilde{\lambda}} \mathbf{1}_q = \mathbf{p}_0, ~ \boldsymbol{\Tilde{\lambda}}\transpose \mathbf{1}_r = \mathbf{p}_N, ~ \boldsymbol{\Tilde{\lambda}} \geq 0,  \label{eq:softBilinear-constr0}\\
        & ~ \boldsymbol{\beta}\mathbf{1}_t = \mathbf{p}_N, ~ \boldsymbol{\beta}\transpose\mathbf{1}_q = \mathbf{p}_{d} , ~ \boldsymbol{\beta} \geq 0, \label{eq:softBilinear-constr1} \\
        & ~ \mathbf{C}_{i, j} \geq J_\mathrm{mean}^\star (\mu_i^0, \mu_j^N) + \tr{\Theta_6 \Sigma_i^0} \nonumber \\
        & \qquad \quad ~ + \tr{\Theta_7 \Sigma_j^N} + \tr{L_{i, j} + L\transpose_{i, j}}, \label{eq:softBilinear-constr2}\\
        & ~ \begin{bmatrix}
            \Theta_8 \Sigma_j^N \Theta_8\transpose & L_{i, j} \\
            L_{i, j}\transpose & \Sigma_i^0
        \end{bmatrix} \succeq 0,  \label{eq:softBilinear-constr3} \\ 
        & ~ \mathbf{T}_{i, j} \geq \lVert \mu_j^d - \mu_i^N \rVert_2^2 + \tr{\Sigma_i^N} + \tr{\Sigma_j^d} \nonumber \\
        & ~ \qquad \quad ~ + \tr{Y_{i, j} + Y_{i, j}\transpose}, \label{eq:softBilinear-constr4} \\
        & ~ \begin{bmatrix}
            \Sigma_i^N & Y_{i, j} \\
            Y_{i, j}\transpose & \Sigma_{j}^d
        \end{bmatrix} \succeq 0, \label{eq:softBilinear-constr5}
    \end{align}
    \end{subequations}
    where $J_\mathrm{mean}\spr{\star}(\mu_x, \mu_y)$ is a convex quadratic function defined in Proposition \ref{prop:optimal-mean-steering}.
    In addition, $\boldsymbol{\Tilde{\lambda}}, \mathbf{C} \in \R{r \times q}$, $\boldsymbol{\beta}, \mathbf{T} \in \R{q \times t}$, $L_{i,j} , Y_{i, j} \in \R{n \times n}$. 
    $\mathbf{p}_0 := [p_0^{0}, \dots, p_{r-1}^{0}] \in \Delta_{r}$, $\mathbf{p}_{N} := [ p_0^{N}, \dots, p_{q-1}^N] \in \Delta_{q}$ and $\mathbf{p}_d := [p_0^d, \dots, p_{t-1}^d] \in \Delta_{t}$. 
    $\mathbf{C}_{i,j}$ and $\mathbf{T}_{i, j}$ denote the $(i, j)$ entry of matrices $\mathbf{C}$ and $\mathbf{T}$, respectively.
    The constraints specified in \eqref{eq:softBilinear-constr2} and \eqref{eq:softBilinear-constr3} are imposed for all pairs $(i, j) \in \curly{0, \dots, r-1} \times \curly{0, \dots, q-1}$, while \eqref{eq:softBilinear-constr4} and \eqref{eq:softBilinear-constr5} are enforced for all pairs $(i, j) \in \curly{0, \dots, q-1} \times \curly{0, \dots, t-1}$.
    Furthermore, the optimal mixing weights $\lambda_{i,j}\spr{\star}$ of the policy $\pi \in \Pi_r$ are given by $\lambda_{i,j}\spr{\star} = \frac{\Tilde{\lambda}_{i,j}}{p_i\spr{0}}$, where $\Tilde{\lambda}_{i,j}\spr{\star}$ is the optimal solution of the optimization problem in \eqref{eq:softBilinear}.
\end{theorem}
\begin{proof}
    First, we show that the inequality constraints in \eqref{eq:softBilinear-constr2} and \eqref{eq:softBilinear-constr4} are tight at optimality. 
    To see that, let all decision variables except $\mathbf{C}$ and $\mathbf{T}$ be fixed. 
    Since $\lambda_{i, j} \geq 0$, the objective function in \eqref{eq:softBilinear-objective} is minimized if $\mathbf{C}_{i, j}$ is equal to the right hand side of \eqref{eq:softBilinear-constr2}. 
    Thus, for any fixed $\boldsymbol{\Tilde{\lambda}} \geq 0$, the objective function is minimized, if the right hand side of \eqref{eq:softBilinear-constr2} is minimized with respect to $L_{i, j}$ subject to \eqref{eq:softBilinear-constr3}. 
    From Corollary \ref{corrolary:alternative-cov-steer-expr}, the optimal value of this minimization problem is an alternative representation of $J^{\star}_\mathrm{mean}(\mu_i^0, \mu_j^N) + J_\mathrm{cov}^{\star}(\Sigma_i^0, \Sigma_j^N)$.
    Thus, we have shown that $\mathbf{C}_{i, j} = J^{\star}_\mathrm{mean}(\mu_i^0, \mu_j^N) + J_\mathrm{cov}^{\star}(\Sigma_i^0, \Sigma_j^N)$ at optimality. The same arguments can be applied to \eqref{eq:softBilinear-constr4} and \eqref{eq:softBilinear-constr5} along with Lemmas \ref{lemma:sdp-nuc-norm} and \ref{lemma:spd-equivalent-form} to show that $\mathbf{T}_{i, j}$ is equal to the right hand side of \eqref{eq:wasserstein-gaussian}. 
    Therefore, the objective function is equal to $\mathcal{J}_{3}(\mathcal{S}_2^{\prime})$ given in \eqref{eq:nlp-soft-GMM-density-steering-objective2} and the proof is complete.
\end{proof}

\begin{remark}
    The number of mixture components for the terminal state GMM distribution is not considered a decision variable in the optimization problem presented in \eqref{eq:softBilinear}. 
    Instead, it is taken to be a parameter that is determined prior to solving the problem and remains fixed throughout the optimization process. 
    Typically, choosing $q = \mathrm{max} (r, t)$ yields satisfactory results in numerical simulations.
\end{remark}

\subsection{Reduction of Total Cost Constrained Problem to an NLP}\label{ss:reduction-effort-nlp}
Solving Problem \ref{prob:total-cost-constrained-GMM-density-steering} over the set of policies given in \eqref{eq:gmm-policy-definition} is similar to solving Problem \ref{prob:soft-GMM-density-steering}. 
Both problems have the same set of decision variables and share most of the constraints. 
An NLP formulation for Problem \ref{prob:total-cost-constrained-GMM-density-steering} is given as follows:
\begin{subequations}\label{eq:nlp-total-cost-GMM-density-steering}
\begin{align}
    \min_{\mathcal{S}_2} & ~ \mathcal{J}_{4}(\mathcal{S}_2) ~~  \text{s.t.} ~ \eqref{eq:nlp-hard-GMM-density-steering-constr-0}, \eqref{eq:nlp-soft-GMM-density-steering-constr-0} \text{-} \eqref{eq:nlp-soft-GMM-density-steering-constr-3}, \\
    & ~ \sum_{i=0}^{r-1} \sum_{j=0}^{q-1} p_i^0 \lambda_{i, j} \big ( J_\mathrm{mean}^i(\mathbf{\Bar{U}}_{i, j}) + J^i_\mathrm{cov}(\mathbf{L}_{i, j}) \big) \leq \kappa \label{eq:nlp-total-cost-GMM-density-steering-constr}
\end{align}
\end{subequations}
where $\mathcal{J}_4(\mathcal{S}_2) := \sum_{i=0}^{q-1} \sum_{j=0}^{t-1} \beta_{i, j} W_2^2 (\mu_i^N, \Sigma_i^N, \mu_j^d, \Sigma_j^d)$.
Similar to the NLPs in \eqref{eq:nlp-hard-GMM-density-steering} and \eqref{eq:nlp-soft-GMM-density-steering}, the NLP in \eqref{eq:nlp-total-cost-GMM-density-steering} is not convex. 
To eliminate variables $\mathbf{\Bar{U}}_{i, j}$ and $\mathbf{L}_{i, j}$ in \eqref{eq:nlp-total-cost-GMM-density-steering}, we use Propositions \ref{prop:optimal-mean-steering}, \ref{prop:optimal-covariance-steering} and Corollary \ref{corrolary:alternative-cov-steer-expr} as described in Section \ref{ss:reduction-effort-nlp}.
The final form of the optimization problem is given in the following theorem. 
\begin{theorem}\label{theorem:totalBilinear}
    The optimal parameters of the optimal policy $\pi \in \Pi_r$ that solves Problem \ref{prob:total-cost-constrained-GMM-density-steering} can be obtained by solving the following bilinear program:
    \begin{subequations}\label{eq:totalBilinear}
    \begin{align}
        \min_{\mathcal{S}_\mathrm{total}} & ~ \tr{\mathbf{T}\transpose \boldsymbol{\beta} } ~~
        \mathrm{s.t.} ~~~\eqref{eq:softBilinear-constr0} \text{-} \eqref{eq:softBilinear-constr5} \\
        & ~ \mathrm{tr} \big( \mathbf{C}\transpose \boldsymbol{\Tilde{\lambda}} \big) \leq \kappa, \label{eq:totalBilinear-constr}
    \end{align}
    \end{subequations}
    where $\boldsymbol{\Tilde{\lambda}}, \mathbf{C} \in \R{r \times q}$, $\boldsymbol{\beta}, \mathbf{T} \in \R{q \times t}$ and the set of decision variables $\mathcal{S}_\mathrm{total}$ is equal to $\mathcal{S}_{\mathrm{soft}}$ from Theorem \ref{theorem:soft-to-bilinear}.
\end{theorem}
The proof of Theorem \ref{theorem:totalBilinear} is similar to the proof of Theorem \ref{theorem:soft-to-bilinear} and is thus omitted.
The optimization problem presented in \eqref{eq:totalBilinear} has a bilinear objective function and a bilinear constraint given in \eqref{eq:totalBilinear-constr}. 
Furthermore, all other constraints yield a convex set of decision variables.

\subsection{Reduction of Step Cost Constrained Problem to NLP}\label{ss:reduction-step-nlp}
The presence of the constraint given in \eqref{eq:step-cost-constrained-GMM-density-steering-constr} prevents us from using the results derived in Propositions \ref{prop:optimal-mean-steering} and \ref{prop:optimal-covariance-steering} to address Problem \ref{prob:step-cost-constrained-GMM-density-steering}. This is because the constraints in \eqref{eq:step-cost-constrained-GMM-density-steering-constr} are functions of the mean and covariance of both the state and the control processes at time step $k \in \curly{0, \dots, N-1}$. 
Since the other constraints of Problem \ref{prob:step-cost-constrained-GMM-density-steering} are shared with Problem \ref{prob:total-cost-constrained-GMM-density-steering}, we will focus on \eqref{eq:step-cost-constrained-GMM-density-steering-constr} and write $\E{J_k (x_k, u_k)}$ in terms of the decision variables $\lambda_{i, j}, \mathbf{\Bar{U}}_{i, j}$, and $ \mathbf{L}_{i, j}$.
To do so, we will use the law of iterated expectations and obtain $\E{J_k (x_k, u_k)} = \sum_{i=0}^{r-1} \sum_{j=0}^{q-1} p_i^0 \lambda_{i, j} h_{k}( \mathbf{\Bar{U}}_{i, j}, \mathbf{L}_{i, j}; \mu_i^0, \Sigma_i^0 )  $, where
\begin{align}
    h_k (\mathbf{\Bar{U}}_{i, j}, \mathbf{L}_{i, j} ; \mu_i^0, \Sigma_i^0 )  & :=   \lVert \mathbf{P}_k^u \mathbf{\Bar{U}}_{i, j} \rVert_{R_k}^2\nonumber \\ 
    & + \tr{R_k \mathbf{P}_k^u \mathbf{L}_{i, j} \Sigma_i^0  \mathbf{L}_{i, j}\transpose \mathbf{P}_k^{u \mathrm{T}} } \nonumber \\
    & + \lVert \mathbf{P}_k^x (\mathbf{\Gamma} \mu_i^0 + \mathbf{H_u} \mathbf{\Bar{U}}_{i, j}) - x'_k \rVert_{Q_k}^2 \nonumber \\
    & + \tr{Q_k \mathbf{P}_k^x \mathbf{G}_{i,j} \Sigma_i^0 \mathbf{G}_{i, j} \mathbf{P}_k^{x \mathrm{T}} }, \label{eq:h_func_def}
\end{align}
$\mathbf{G}_{i, j} := \mathbf{\Gamma} + \mathbf{H_u} \mathbf{L}_{i, j}$, 
$\mathbf{P}_k^x \in \R{n \times n(N+1)}$ and $\mathbf{P}_k^u \in \R{m \times mN}$ are block matrices whose $k$th block is equal to $I_n$ and $I_m$, respectively, such that $x_k = \mathbf{P}_k^x \mathbf{X}$ and $u_k = \mathbf{P}_k^u \mathbf{U}$.
Now, we can formulate the NLP associated with Problem \ref{prob:step-cost-constrained-GMM-density-steering} as follows:
\begin{subequations}\label{eq:nlp-step-cost-GMM-density-steering}
\begin{align}
    \min_{\mathcal{S}_2} & ~ \mathcal{J}_4 (\mathcal{S}_2) ~~ \text{s.t.} ~~  \eqref{eq:nlp-hard-GMM-density-steering-constr-2}, \eqref{eq:nlp-soft-GMM-density-steering-constr-0} \text{-} \eqref{eq:nlp-soft-GMM-density-steering-constr-3}, \\
    & \sum_{i=0}^{r-1} \sum_{j=0}^{q-1} p_i^0 \lambda_{i, j} h_{k}(\mathbf{\Bar{U}}_{i, j}, \mathbf{L}_{i, j}; \mu_i^0, \Sigma_i^0) \leq \kappa_k, \label{eq:nlp-step-cost-GMM-density-steering-constr} 
\end{align}
\end{subequations}
where the constraint in \eqref{eq:nlp-step-cost-GMM-density-steering-constr} is enforced for all $k \in \curly{0, \dots, N-1}$.
Note that the functions $h_{k}(\mathbf{\Bar{U}}_{i, j}, \mathbf{L}_{i, j}; \mu_i^0, \Sigma_i^0)$ are jointly convex in $(\mathbf{\Bar{U}}_{i, j}, \mathbf{L}_{i, j})$. 
However, the constraint given in \eqref{eq:nlp-soft-GMM-density-steering-constr-2} is non-convex. The constraint in \eqref{eq:nlp-soft-GMM-density-steering-constr-2} will be convexified in the final form of the associated optimization problem given in the following theorem (the convexification procedure is described in its proof). 

\begin{theorem}\label{theorem:stepBilinear}
    The optimal parameters of the policy $\pi \in \Pi_r$ given in \eqref{eq:gmm-policy-definition} that solve Problem \ref{prob:step-cost-constrained-GMM-density-steering} can be obtained by solving the following NLP with decision variables $\mathcal{S}_\mathrm{step} := \big( \boldsymbol{\Tilde{\lambda}},  \boldsymbol{\beta}, \mathbf{T}, \mathbf{p}_{N},$$\curly{\mu_i^{N}, \Sigma_i^{N}}_{i=0}^{q-1},$ $\curly{\mathbf{\Bar{U}}_{i, j},  \mathbf{M}_{i, j}, \mathbf{Y}_{i, j}, \mathbf{\Sigma}_{i, j}}_{i=0, j=0}^{r-1, q-1}, \curly{Y_{i, j}}_{i=0, j=0}^{q-1, t-1} \big)$:
    \begin{subequations}\label{eq:stepBilinear}
    \begin{align}
        \min_{\mathcal{S}_\mathrm{step}} & ~ \tr{\boldsymbol{\beta}\transpose \mathbf{T}} ~
        \mathrm{s.t.}~ \eqref{eq:nlp-soft-GMM-density-steering-constr-1}, \eqref{eq:softBilinear-constr0}, \eqref{eq:softBilinear-constr1}, \eqref{eq:softBilinear-constr4}, \eqref{eq:softBilinear-constr5}, \label{eq:stepBilinear-objective}\\
        & ~ \Sigma_j^N = \Phi_{N:0} \Sigma_i^0 \Phi_{N:0}\transpose + \Phi_{N:0} \mathbf{Y}_{i, j}\transpose \mathbf{B_N}\transpose \nonumber \\
        & ~ \qquad \qquad + \mathbf{B_N} \mathbf{Y}_{i, j} \Phi_{0:N}\transpose + \mathbf{B_N} \mathbf{M}_{i, j} \mathbf{B_N}\transpose, \label{eq:stepBilinear-constr0} \\ 
        & ~ \mathbf{\Sigma}_{i, j} = \mathbf{\Gamma} \Sigma_i^0 \mathbf{\Gamma}\transpose + \mathbf{\Gamma} \mathbf{Y}_{i, j}\transpose \mathbf{H_u}\transpose + \mathbf{H_u} \mathbf{Y}_{i, j} \mathbf{\Gamma}\transpose \nonumber \\
        & ~ \qquad \qquad + \mathbf{H_u} \mathbf{M}_{i, j} \mathbf{H_u}\transpose, \label{eq:stepBilinear-constr1} \\
        & \begin{bmatrix}
            \mathbf{M}_{i, j} & \mathbf{Y}_{i, j} \\
            \mathbf{Y}_{i, j}\transpose & \Sigma_{i}^0
        \end{bmatrix} \succeq 0,\label{eq:stepBilinear-constr2} \\
        & ~ \sum_{i=0}^{r-1} \sum_{j=0}^{q-1} \Tilde{\lambda}_{i,j} \Big ( \lVert \mathbf{P}_k^u \mathbf{\Bar{U}}_{i, j} \rVert_{\Tilde{R}_k}^2 + \tr{\Tilde{R}_k \mathbf{P}_k^u \mathbf{M}_{i, j} \mathbf{P}_k^{u \mathrm{T}} } \nonumber \\
        & ~ \qquad + \lVert \mathbf{P}_k^x (\mathbf{\Gamma} \mu_i^0 + \mathbf{H_u} \mathbf{\Bar{U}}_{i, j}) - x'_k \rVert_{\Tilde{Q}_k}^2 \nonumber \\
        & ~ \qquad + \tr{\Tilde{Q}_k \mathbf{P}_k^x \mathbf{\Sigma}_{i,j} \mathbf{P}_k^{x \mathrm{T}} } \Big) \leq 1 \label{eq:stepBilinear-constr3},
    \end{align}
    \end{subequations}
    where $\mathbf{M}_{i, j} \in \S{m N}^{+}$, $\mathbf{Y}_{i, j} \in \R{mN \times n}$, $\Tilde{R}_k = (1/\kappa_k) R_k$, $\Tilde{Q}_k = (1/\kappa_k) Q_k$. 
    The constraints in \eqref{eq:stepBilinear-constr0}-\eqref{eq:stepBilinear-constr2} are imposed for all $(i, j) \in \curly{0, \dots, r-1} \times \curly{0, \dots, q-1}$ and the constraint in \eqref{eq:stepBilinear-constr3} is imposed for all $k \in \curly{0, \dots, N-1}$.
    Furthermore, the optimal feedback policy parameters $\mathbf{L}_{i, j} = \vertcat{L^{i, j}_{0}, \dots, L^{i, j}_{N-1}}$ for each $(i, j, k) \in \curly{0, \dots, r-1} \times \curly{0, \dots, q-1} \times \curly{0, \dots, N-1}$ are given by
    \begin{subequations}
    \begin{align}
        L_{k}^{i, j} & = K_k^{i, j} \Tilde{A}^{i, j}_{k-1} \Tilde{A}^{i, j}_{k-2} \dots \Tilde{A}^{i, j}_{0}, \label{eq:state-initial0}\\
        \Tilde{A}^{i, j}_{k} & = (A_k + B_k K^{i, j}_{k}), \label{eq:state-initial1}\\
        K^{k}_{i, j} & = \Bar{Y}^{i, j}_{k} (\Sigma_i^k)^{-1},\label{eq:state-initial2}
    \end{align}
    \end{subequations}
    where $\Bar{Y}_{k}^{i, j}$ is the respective component of the global minimizer of the following SDP:
    \begin{subequations}\label{eq:stepBilinear-correction}
    \begin{align}
        \min_{\mathcal{S}^{i, j, k}_{\mathrm{aux}}} & ~ \tr{R_k M^{i, j}_k} \label{eq:stepBilinear-correction-objective}\\
        \mathrm{s.t.} & ~ \Sigma^{k+1}_{i, j} = A_k \Sigma^{k}_{i, j} A_k\transpose + A_k Y^{i, j \, \mathrm{T}}_{k} B_k\transpose \nonumber \\ 
        & ~ + B_k Y^{i, j}_{k} A_k\transpose + B_k M^{i, j}_{k} B_k\transpose, \label{eq:stepBilinear-correction-constr0}\\
        & ~ \begin{bmatrix}
            M^{i, j}_{k} & Y^{i, j}_{k} \\
            Y^{i, j}_{k\mathrm{T}} & \Sigma_{i, j}^{k}
        \end{bmatrix} \succeq 0, \label{eq:stepBilinear-correction-constr1}
    \end{align}
    \end{subequations}
    where $\mathcal{S}_{\mathrm{aux}}^{i, j, k} := \curly{M_k^{i, j} \in \S{m},  Y_k^{i, j} \in \R{m \times n}}$,  $\Sigma^{k+1}_{i, j} = \mathbf{P}_{k+1}^x \mathbf{\Bar{\Sigma}}_{i, j} \mathbf{P}_{k+1}^{x\mathrm{T}}$, $\Sigma^{k}_{i, j} = \mathbf{P}_k^x \mathbf{\Bar{\Sigma}}_{i, j} \mathbf{P}_k^{x\mathrm{T}}$, and $\mathbf{\Bar{\Sigma}}_{i, j}$ is the optimal value of $\mathbf{\Sigma}_{i, j}$ obtained by solving the problem in \eqref{eq:stepBilinear}.
\end{theorem}
\begin{proof}
    The derivations of the expressions for the objective function and constraints \eqref{eq:softBilinear-constr0}, \eqref{eq:softBilinear-constr1}, \eqref{eq:softBilinear-constr4} and \eqref{eq:softBilinear-constr5} are similar to the derivation of \eqref{eq:totalBilinear}. 
    To convexify the constraint in \eqref{eq:nlp-soft-GMM-density-steering-constr-2}, we expand the term $\mathbf{H}_{i,j}$ given in \eqref{eq:H_ij_def} and  rewrite the constraint as $\Sigma_j^N = \Phi_{N:0} \Sigma_i^0 \Phi_{N:0} + \Phi_{N:0} \Sigma_i^0 \mathbf{L}_{i, j}\transpose \mathbf{B_N}\transpose + \mathbf{B_N} \mathbf{L}_{i, j} \Sigma_i^0 \Phi_{N:0}\transpose +\mathbf{B_N} \mathbf{L}_{i, j} \Sigma_{i}^0 \mathbf{L}_{i, j}\transpose \mathbf{B_N}\transpose$. 
    Then, we introduce the following variable transformations:
    \begin{align}\label{eq:transformation12}
    \mathbf{Y}_{i, j} & = \mathbf{L}_{i, j} \Sigma_i^0,~~~ 
    \mathbf{M}_{i, j} = \mathbf{Y}_{i, j} (\Sigma_i^{0})^{-1} \mathbf{Y}_{i, j}\transpose  
    \end{align}
    for all $(i, j) \in \curly{0, \dots, r-1} \times \curly{0, \dots, q-1}$. 
    Applying the variable transformation in \eqref{eq:transformation12} to constraint \eqref{eq:nlp-soft-GMM-density-steering-constr-2}, we obtain \eqref{eq:stepBilinear-constr0}.
    For notational simplicity, we add the constraint $\mathbf{\Sigma}_{i, j} = \mathbf{G}_{i, j} \Sigma_i^0 \mathbf{G}_{i, j}\transpose$, expand $\mathbf{G}_{i, j}$ and apply the variable transformations \eqref{eq:transformation12} to obtain \eqref{eq:stepBilinear-constr3} after dividing both sides by $\kappa_k$. 
    Lastly, by relaxing the second equality in \eqref{eq:transformation12} to an inequality and applying Schur's complement, we obtain \eqref{eq:stepBilinear-constr2}.

Note that we need the second equality in \eqref{eq:transformation12} to hold at optimality. This condition is satisfied, if the matrix $V_{i, j} = \Big[ \begin{smallmatrix} \mathbf{M}_{i, j} & \mathbf{Y}_{i, j} \\ \mathbf{Y}_{i, j}\transpose & \Sigma^{0}_{i} \end{smallmatrix} \Big]$ has rank $n$. Otherwise, $\mathbf{M}_{i, j} - \mathbf{Y}_{i, j} (\Sigma_i^0)^{-1} \mathbf{Y}_{i, j}\transpose = \boldsymbol{\mathcal{W}}_{i, j} \in \S{mN}^{+}$ with $\boldsymbol{\mathcal{W}}_{i, j}\neq 0$. Thus, $(\Phi_{N:0} + \mathbf{B_N}\mathbf{L}_{i, j}) \Sigma_i^0 (\Phi_{N:0} + \mathbf{B_N}\mathbf{L}_{i, j})\transpose$ will not be equal to $\Sigma_j^N$, since the feedback policy matrix is obtained as $\mathbf{L}_{i, j} = \mathbf{Y}_{i, j} (\Sigma_i^0)^{-1}$. However, even if $\boldsymbol{\mathcal{W}}_{i, j} \neq \mathbf{0}$, there still exists a randomized policy $\pi_h$ that steers the state covariance of system \eqref{eq:linear-system-eq} from $\Sigma_i^0$ at $k=0$ to $\Sigma_j^N$ at $k=N$, where 
    \begin{align}\label{eq:theorem-stepBilinear-helper-policy}
        \pi_h(x_0) = L_k^{i, j} (x_0 - \mu_i^0) + \Bar{u}_{k}^{i,j} + w_k^{i, j},
    \end{align}
    where $w_k^{i, j}  = \mathbf{P}_k^u \varpi_{i, j}$, $\varpi_{i, j} \sim \mathcal{N}(\mathbf{0}, \boldsymbol{\mathcal{W}}_{i, j})$. 
    Now, observe that the policy in \eqref{eq:theorem-stepBilinear-helper-policy} satisfies the constraints \eqref{eq:stepBilinear-constr0}, \eqref{eq:stepBilinear-constr1} and thus, the constraint in \eqref{eq:stepBilinear-constr3} is trivially satisfied. 
    From \cite[Lemma 1]{p:balci2022exact}, there exists a state feedback policy of the form:
    \begin{align}
        \pi_s(x_k) = K_k^{i, j} (x_k - \mu_i^0) + u_{k}^{i, j} + v^{i,j}_k
    \end{align}
    where $v_k^{i, j} \sim \mathcal{N}(\bm{0}, V_{k}^{i, j})$ such that the first two moments of the state process $x_k$ and control process $u_k$ under policy $\pi_h$ are identical
    with the respective processes under policy $\pi_s$, $\forall k$.
    Finally, \cite[Theorem 3]{p:balci2022exact} implies that the  constraint in \eqref{eq:stepBilinear-correction-constr1} is tight at optimality and thus, there exists a state feedback policy that enforces the constraint \eqref{eq:stepBilinear-constr3}, $\forall k$. 
    Lastly, the initial state parametrized feedback matrices $L_k^{i, j}$ can be obtained from \eqref{eq:state-initial0}-\eqref{eq:state-initial2}.
    This completes the proof.
\end{proof}

\begin{remark}
    Note that the optimization problems presented in \eqref{eq:softBilinear}, \eqref{eq:totalBilinear}, and \eqref{eq:stepBilinear} are non-convex because of the bilinear objective functions and constraints. 
    However, when one set of variables is fixed, the objective function and constraints become linear in the remaining decision variables. Leveraging this property enables us to address these problems using block coordinate descent schemes \cite{p:tseng2001block}, as elaborated in Section \ref{s:solution-procedure}.
\end{remark}

\begin{remark}
    The expectation-type constraints outlined in \eqref{eq:step-cost-constrained-GMM-density-steering-constr} offer a means to approximate chance constraints that are of the form $\P[]{ \lVert u_k \rVert_2^2 \leq u_{\text{max}}} \geq 1 -\delta$ for $\delta \in (0, 1)$ through the application of Markov's inequality. However, a comprehensive discussion on this matter falls beyond the scope of this paper.
\end{remark}

\begin{figure}
    \centering
    \resizebox{\linewidth}{!}{
        \begin{tikzpicture}
    \drawellipseInit{0}{5}{-10}{1.0}{1.5}{0}{0.5}
    \drawellipseInit{0.0}{1.5}{-20}{1.5}{0.6}{1}{0.5}

    \drawellipseTerm{4}{4.3}{10}{1.0}{1.3}{0}{0.8}
    \drawellipseTerm{4}{1.5}{20}{1.5}{0.4}{1}{0.2}

    \drawellipseDes{8}{5.5}{-20}{1}{0.5}{0}{0.5}
    \drawellipseDes{8}{4.1}{60}{1}{1}{1}{0.3}
    \drawellipseDes{8}{1.9}{60}{2}{0.5}{2}{0.2}

    \coordinate (i0) at (1.0, 5.0);
    \coordinate (i1) at (0.6, 1.9);
    \coordinate (t0) at (3, 4.2);
    \coordinate (t1) at (3.7, 1.8);

    \coordinate (t00) at (5.0, 4.3);
    \coordinate (t11) at (5.4, 2.0);
    
    \coordinate (d0) at (7.1, 5.7);
    \coordinate (d1) at (7.1, 3.9);
    \coordinate (d2) at (7.6, 2.2);

    \coordinate (O) at (0, 0);
    \coordinate (e0) at (\ex, 0);
    \coordinate (e1) at (0, \ey);
    
    \node at ($(i0)+(0.4, 0.2)$) {$\lambda_{0, 0}$};
    \draw [line width=2, dashed, color=black, ->] (i0) -- (t0);
    \node at ($(i0)+(0.2, -0.7)$) {$\lambda_{0, 1}$};
    \draw [line width=1, dashed, color=black, ->] (i0) -- (t1);
    \node at ($(i1)+(0.2, 0.8)$) {$\lambda_{1, 0}$};
    \draw [line width=2, dashed, color=black, ->] (i1) -- (t0);
    \node at ($(i1)+(1.2, -0.3)$) {$\lambda_{1, 1}$};
    \draw [line width=1, dashed, color=black, ->] (i1) -- (t1);

    \node at ($(t00)+(0.4, 0.7)$) {\textcolor{blue}{$\beta_{0, 0}$}};
    \draw [line width=2, dashed, color=blue, ->] (t00) -- (d0);
    \node at ($(t00)+(0.8, 0.1)$) {\textcolor{blue}{$\beta_{0, 1}$}};
    \draw [line width=2, dashed, color=blue, ->] (t00) -- (d1);
    \node at ($(t00)+(0.5, -0.7)$) {\textcolor{blue}{$\beta_{0, 2}$}};
    \draw [line width=1, dashed, color=blue, ->] (t00) -- (d2);

    \node at ($(t11)+(-0.1, 0.7)$) {\textcolor{blue}{$\beta_{1, 0}$}};
    \draw [line width=1, dashed, color=blue, ->] (t11) -- (d0);
    \node at ($(t11)+(0.9, 0.5)$) {\textcolor{blue}{$\beta_{1, 1}$}};
    \draw [line width=1, dashed, color=blue, ->] (t11) -- (d1);  
    \node at ($(t11)+(0.7, -0.3)$) {\textcolor{blue}{$\beta_{1, 2}$}};
    \draw [line width=2, dashed, color=blue, ->] (t11) -- (d2);  
\end{tikzpicture}
    }
    \caption{
     2-$\sigma$ confidence ellipses corresponding to the Gaussian components of the GMM representing the PDF of the initial, terminal and desired states, respectively. The opacity of the ellipses reflects the weights of the GMMs whereas the thickness of the black and blue arrows reflects the values of policy weight parameters $\lambda_{i, j}$ and Wasserstein-GMM parameters $\beta_{i, j}$, respectively. }
    \label{fig:problem-illustration}
\end{figure}
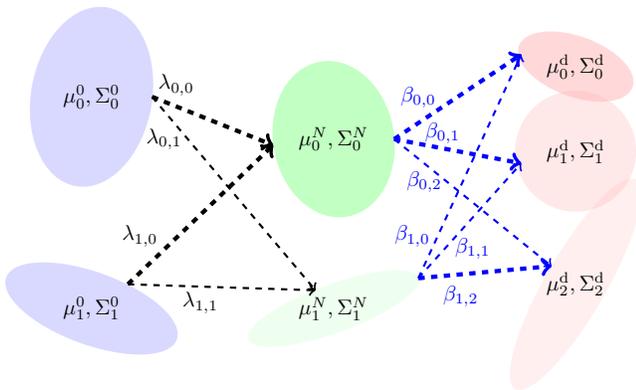

\section{Block Coordinate Descent Solution Procedure}\label{s:solution-procedure}
In Section \ref{s:formulation-nlp}, we showed how Problems \ref{prob:soft-GMM-density-steering}-\ref{prob:step-cost-constrained-GMM-density-steering} can be reduced to NLPs. However, unlike Problem \ref{prob:hard-GMM-density-steering}, they cannot be directly reduced to an LP or any other standard form of convex optimization problems. Nevertheless, we observed that when one set of variables is fixed, the remaining ones yield a convex optimization problem for each of these problems. This observation allows us to design a BCD-based \cite{p:tseng2001block} method to solve these problems by using convex optimization techniques. 
This section details the algorithmic solution procedure to achieve this goal.

The BCD algorithms are used to solve problems of the form:
\begin{subequations}\label{eq:bilinear-sample-problem}
\begin{align}
    \min_{x\in \R{n}, y\in \R{m}} & ~ f_0(x, y) \label{eq:bilinear-sample-problem-objective}\\
    \text{s.t.} & ~ f_i(x, y) \leq 0, ~~ \forall i \in \curly{1, \dots, N_c} \label{eq:bilinear-sample-problem-constr}
\end{align}
\end{subequations}
where the function $f_{i}(x, y)$ is convex with respect to $x \in \R{n}$ when $y \in \R{m}$ is fixed and convex with respect to $y \in \R{m}$ when $x \in \R{n}$ is fixed, for all $i \in \curly{0, \dots, N_c}$.
Note that this does not mean that $f_i(x, y)$ is jointly convex in $x, y$.
To see this, consider the function $f(x, y) = xy$ where $x, y \in \R{}$. 
For each fixed $x$ or $y$, the function $f(x, y)$ is a linear function of the other variable. 
However, its Hessian $H_{f}(\Bar{x}, \Bar{y}) = \left[ \begin{smallmatrix}
    0 & 1 \\ 1 & 0
\end{smallmatrix} \right]$ is constant for all $(\Bar{x}, \Bar{y})$ and 
$H_{f}(\Bar{x}, \Bar{y}) \notin \S{2}^{+}$.

The BCD schemes operate by fixing one set of variables and solving for the other, iterating this process over each block variable in a cyclic manner until convergence. 
For the example problem provided in \eqref{eq:bilinear-sample-problem}, the iterations proceed as follows, starting from an initial guess $(x_0, y_0)$:
\begin{subequations}\label{eq:bilinear-sample-iterations}
\begin{align}
    x_{k+1} & = \argmin_{x} f_0(x, y_k) ~ && \text{s.t.} ~ f_i(x, y_k) \leq 0   \label{eq:bilinear-sample-iterations-x} \\
    y_{k+1} & = \argmin_{y} f_0(x_{k+1}, y) ~ && \text{s.t.} ~ f_i(x_{k+1}, y) \leq 0\label{eq:bilinear-sample-iterations-y}
\end{align}
\end{subequations}
Iterations in \eqref{eq:bilinear-sample-iterations} continue until a convergence criterion, such as $f_0(x_k, y_k) - f_0(x_{k+1}, y_{k+1}) \leq \epsilon$ or $\lVert x_{k+1} - x_{k} \rVert + \lVert y_{k+1} - y_k \rVert \leq \epsilon$, is satisfied. Assuming each sub-problem \eqref{eq:bilinear-sample-iterations-x} and \eqref{eq:bilinear-sample-iterations-y} has a unique global minimizer, the BCD scheme outlined in \eqref{eq:bilinear-sample-iterations} is guaranteed to converge to a local minimizer as shown in \cite[Theorem 2.3]{p:yangyang2013blockcd}. However, since the problem in \eqref{eq:bilinear-sample-problem} can be non-convex, the procedure described in \eqref{eq:bilinear-sample-iterations} cannot ensure convergence to a global minimizer.

Our BCD scheme operates by isolating the variables responsible for bilinearity. 
To address the problem in \eqref{eq:softBilinear} using BCD, we partition the decision variables into two blocks: $\mathcal{S}_\mathrm{soft}^{1} :=  \big(\mathbf{C}, \mathbf{T}, \curly{L_{i, j}}_{i=0, j=0}^{r-1, q-1}, \curly{Y_{i, j}}_{i=0, j=0}^{q-1, t-1}, \curly{\mu_i^N, \Sigma_i^N}_{i=0}^{q-1} \big)$ and $\mathcal{S}_\mathrm{soft}^{2} = ( \boldsymbol{\Tilde{\lambda}}, \boldsymbol{\beta}, \mathbf{p}_{N})$. 
Initially, we set the variables in $\mathcal{S}_{\mathrm{soft}}^{2}$ to fixed values $\boldsymbol{\Tilde{\lambda}} = \boldsymbol{\Tilde{\lambda}}_k$, $\boldsymbol{\beta} = \boldsymbol{\beta}_k$, $\mathbf{p}_N = \mathbf{p}_{N, k}$ and proceed to solve for $\mathcal{S}_{\mathrm{soft}}^{1}$. 
The optimization sub-problem associated with updating $\mathcal{S}_{\mathrm{soft}}^{1}$ is given as follows:
\begin{align} \label{eq:softBilinearSubProblem1}
    \min_{\mathcal{S}_\mathrm{soft}^1} & ~ \mathrm{tr}  (\mathbf{C}\transpose \boldsymbol{\Tilde{\lambda}}_k ) + \kappa \tr{\mathbf{T}\transpose \boldsymbol{\beta}_k}~~  \text{s.t.}  ~ \eqref{eq:softBilinear-constr2}\text{-}\eqref{eq:softBilinear-constr5}~\text{hold}  
\end{align}
Here, the subscript $k$ denotes the values of the variables obtained in the $k$th step of the BCD algorithm. 
Once the problem in \eqref{eq:softBilinearSubProblem1} is solved and the decision variables in $\mathcal{S}_{\mathrm{soft}}^{1}$ are fixed, we proceed to solve the sub-problem associated with updating $\mathcal{S}_{\mathrm{soft}}^{2}$ which is given as:
\begin{align}\label{eq:softBilinearSubProblem2}
    \min_{\mathcal{S}_\mathrm{soft}^{2}} & ~ \mathrm{tr} ( \mathbf{C}_{k+1}\transpose \boldsymbol{\Tilde{\lambda}}) + \kappa \tr{\mathbf{T}_{k+1}\transpose \boldsymbol{\beta} } ~~
    \text{s.t.}  ~ \eqref{eq:softBilinear-constr0}, \eqref{eq:softBilinear-constr1}~\text{hold} 
\end{align}
By iteratively solving sub-problems \eqref{eq:softBilinearSubProblem1} and \eqref{eq:softBilinearSubProblem2}, we can solve problem \eqref{eq:totalBilinear}. Note that sub-problem \eqref{eq:softBilinearSubProblem1} is an SDP whereas sub-problem \eqref{eq:softBilinearSubProblem2} is an LP. 

Furthermore, the partition of the decision variables to solve the problem in \eqref{eq:totalBilinear} is given as $\mathcal{S}_{\mathrm{total}}^{1} := \mathcal{S}_\mathrm{soft}^{1}$ and $\mathcal{S}_\mathrm{total}^{2} := \mathcal{S}_\mathrm{soft}^{2}$. In the BCD approach, the objective function in \eqref{eq:softBilinearSubProblem1} is replaced with $\tr{\mathbf{T}\transpose \boldsymbol{\beta}_k}$ and $ \mathrm{tr} ( \mathbf{C}\transpose \boldsymbol{\Tilde{\lambda}}_k )  \leq \kappa$ is added to the constraints in $\mathcal{S}_\mathrm{total}^1$ update step:
\begin{subequations}\label{eq:totalBilinearSubProblem1}
\begin{align}
    \min_{\mathcal{S}_\mathrm{total}^1} & ~  \tr{\mathbf{T}\transpose \boldsymbol{\beta}_k},~~\text{s.t.}~ \eqref{eq:softBilinear-constr2}\text{-}\eqref{eq:softBilinear-constr5}~\text{hold} \\
    & ~ \mathrm{tr} \big( { \mathbf{C}\transpose \boldsymbol{\Tilde{\lambda}}_k } \big) \leq \kappa \label{eq:totalBilinearSubProblem1-constr}
\end{align}
\end{subequations}
Similarly, the objective function is replaced by $\tr{\mathbf{T}_{k+1}\transpose \boldsymbol{\beta}}$ whereas $\mathrm{tr} (\mathbf{C}_{k+1}\transpose \boldsymbol{\beta}) \leq \kappa$ is added to the constraints for $\mathcal{S}_\mathrm{total}^2$ update step:
\begin{subequations}\label{eq:totalBilinearSubProblem2}
\begin{align}
    \min_{\mathcal{S}_\mathrm{total}^2} & ~ \tr{ \mathbf{T}_{k+1}\transpose \boldsymbol{\beta} },~~\text{s.t.} ~ \eqref{eq:softBilinear-constr0}, \eqref{eq:softBilinear-constr1}~\text{hold}\\
    & ~ \mathrm{tr} (\mathbf{C}_{k+1}\transpose \boldsymbol{\Tilde{\lambda}}) \leq \kappa \label{eq:totalBilinearSubProblem2-constr}
\end{align}
\end{subequations}

Lastly, the decision variables $\mathcal{S}_\mathrm{step}$ of the problem in \eqref{eq:stepBilinear} are separated into blocks 
$\mathcal{S}^{1}_\mathrm{step} = \big(\mathbf{T},\curly{\mu_i^{N}, \Sigma_i^{N}}_{i=0}^{q-1}$, $\{Y_{i, j}\}_{i=0, j=0}^{q-1, t-1}$, $\{\mathbf{\Bar{U}}_{i, j},  \mathbf{M}_{i, j}, \mathbf{Y}_{i, j},  \mathbf{\Sigma}_{i, j}\}_{i=0, j=0}^{r-1, q-1}  \big)$ 
and $\mathcal{S}^{2}_\mathrm{step} = (\boldsymbol{\Tilde{\lambda}}, \boldsymbol{\beta}, \mathbf{p}_N)$ to solve problem \eqref{eq:stepBilinear} using BCD. 
The $\mathcal{S}_\mathrm{step}^{1}$ update step is given as follows:
\begin{align}\label{eq:stepBilinearSubProblem1}
    \min_{\mathcal{S}_{\mathrm{step}}^{1}} & ~ \tr{ \mathbf{T}\transpose \boldsymbol{\beta}_k }~~\text{s.t.}~ ~ \eqref{eq:nlp-soft-GMM-density-steering-constr-1}, \eqref{eq:softBilinear-constr4}, \eqref{eq:softBilinear-constr5}, \eqref{eq:stepBilinear-constr0}\text{-}\eqref{eq:stepBilinear-constr3}
\end{align}
After problem \eqref{eq:stepBilinearSubProblem1} is solved for $\mathcal{S}^{1}_\mathrm{step}$, the update step for $\mathcal{S}_\mathrm{step}^{2}$ with fixed values of $\mathcal{S}_\mathrm{step}^{1}$ is given as follows:
\begin{align}\label{eq:stepBilinearSubProblem2}
    \min_{\mathcal{S}_{\mathrm{step}}^{2}} & ~ \tr{\mathbf{T}_{k+1}\transpose \boldsymbol{\beta}} ~~ \text{s.t.} ~ \eqref{eq:softBilinear-constr0}, \eqref{eq:softBilinear-constr1}~\text{hold}.
\end{align}
In all three solution procedures for problems \eqref{eq:softBilinear}, \eqref{eq:totalBilinear} and \eqref{eq:stepBilinear} described so far,  one can use the following convergence criterion: $\left((f_{k-1} - f_{k})/f_{k} \leq \epsilon\right) ~ \vee \left( f_{k} \leq \epsilon \right),$ where $f_k$ is the value of the objective function at the $k$th step of the BCD.

Our BCD scheme for solving problems \eqref{eq:softBilinear}, \eqref{eq:totalBilinear}, and \eqref{eq:stepBilinear} relies on the feasibility of their associated convex sub-problems. 
While the sub-problems in \eqref{eq:softBilinearSubProblem1} and \eqref{eq:softBilinearSubProblem2} are always feasible for all fixed values of $\mathcal{S}_\mathrm{soft}^{1}$ and $\mathcal{S}_\mathrm{soft}^1$, respectively, this is not the case for the sub-problems in \eqref{eq:totalBilinearSubProblem1} and \eqref{eq:totalBilinearSubProblem2} needed to solve \eqref{eq:totalBilinear}. 
Similarly, sub-problems \eqref{eq:stepBilinearSubProblem1} and \eqref{eq:stepBilinearSubProblem2} may not be feasible for all fixed $\mathcal{S}_{\mathrm{step}}^{2}$ and $\mathcal{S}_\mathrm{step}^1$, respectively.

To avoid artificial infeasibility resulting from fixed variables in the BCD scheme, we initially solve certain auxiliary problems to find an initial feasible solution. 
These auxiliary problems are created by relaxing constraints causing infeasibility, introducing a slack variable $s_\mathrm{slack}$, and minimizing this slack variable instead of the objective function. 
Subsequently, BCD iterations proceed until a feasible solution is obtained. 
Specifically, we replace the right-hand sides of \eqref{eq:totalBilinearSubProblem1-constr} and \eqref{eq:totalBilinearSubProblem2-constr} with $s_\mathrm{slack}$ and substitute $s_\mathrm{slack}$ for the objective functions in \eqref{eq:totalBilinearSubProblem1} and \eqref{eq:totalBilinearSubProblem2}. 
A feasible solution is found when the optimal value of $s_\mathrm{slack}$ is less than $\kappa$. 
Analogously, auxiliary problems for \eqref{eq:stepBilinearSubProblem1} and \eqref{eq:stepBilinearSubProblem2} are established by replacing the right-hand side of constraint $\eqref{eq:stepBilinear-constr3}$ with $s_\mathrm{slack}$ in problems \eqref{eq:stepBilinearSubProblem1} and \eqref{eq:stepBilinearSubProblem2}. 
The feasible point is determined when the optimal value of $s_\mathrm{slack}$ is less than $1$.

Artificial infeasibility can also be observed in problems \eqref{eq:totalBilinear} and \eqref{eq:stepBilinear}, if the number of terminal GMM components, $q$, is less than the number of initial GMM components, $r$. To observe this, consider an instance of the problem in \eqref{eq:totalBilinear} with $r=2$ and $q=1$. In this case, both components of the initial state GMM must be steered to the same state mean $\mu_0^N$ and state covariance matrix $\Sigma_0^N$ and thus,
\eqref{eq:totalBilinear-constr} can be written as:
\begin{align}\label{eq:example-constr-infeasible}
 C_{0, 0} \Tilde{\lambda}_{0, 0} + C_{1, 0} \Tilde{\lambda}_{1, 0} \leq \kappa
\end{align}
where $C_{0,0} = J_\mathrm{mean}^{\star} (\mu_0^0, \mu_0^N) + J_\mathrm{cov}^{\star} (\Sigma_0^0, \Sigma_0^N)$, $C_{1, 0} := J_\mathrm{mean}^{\star}(\mu_1^0, \mu_0^N) + J_\mathrm{cov}^{\star}(\Sigma_1^0, \Sigma_0^N)$ at optimality.
Moreover,  $\Tilde{\lambda}_{0, 0} = p_0^0$, $\Tilde{\lambda}_{1, 0} = p_0^1$ due to the constraints in \eqref{eq:softBilinear-constr0}.
Thus, the constraint in \eqref{eq:example-constr-infeasible} is rewritten as:
\begin{align}
    & p_0^0 \big(J_\mathrm{mean}^{\star} (\mu_0^0, \mu_0^N) + J_\mathrm{cov}^{\star} (\Sigma_0^0, \Sigma_0^N) \big) \nonumber \\
    & + p_1^0 \big(J_\mathrm{mean}^{\star}(\mu_1^0, \mu_0^N) + J_\mathrm{cov}^{\star}(\Sigma_1^0, \Sigma_0^N) \big) \leq \kappa
\end{align}
Thus, if $\lVert \mu_0^0 - \mu_1^0 \rVert$ is large enough, there may not be a pair $(\mu_0^N, \Sigma_0^N)$ such that \eqref{eq:example-constr-infeasible} holds. 

On the other hand, if $q \geq r$ the existence of the randomized policies in the form of \eqref{eq:gmm-policy-definition} is guaranteed under mild conditions. 
The following proposition formally states the conditions for feasibility of Problems \ref{prob:total-cost-constrained-GMM-density-steering} and \ref{prob:step-cost-constrained-GMM-density-steering} over policies $\pi \in \Pi_r$.
\begin{proposition}\label{prop:gmm-steering-feasibility}
    The problem given in \eqref{eq:totalBilinear} is feasible, if $q \geq r$ and there exist $\curly{ \mathbf{\Tilde{U}}_{i} \in \R{mN}, \mathbf{\Tilde{L}}_{i} \in \R{mN \times n} }_{i=0}^{r-1}$ such that
    \begin{align}\label{eq:prop-feasibility-eq1}
        J_\mathrm{mean}( \mathbf{\Tilde{U}}_{i}; \mu_i^0) + J_\mathrm{cov}(\mathbf{\Tilde{L}}_{i}; \Sigma_i^0) \leq \kappa,
    \end{align}
    for all $i \in \curly{0, \dots, r-1}$, where $J_\mathrm{mean}(\mathbf{\Tilde{U}}_{i, j}; \mu_i^0)$ and $J_\mathrm{cov}(\mathbf{\Tilde{L}}_{i, j}, \Sigma_i^0)$ are defined in \eqref{eq:Mean-Steering-Problem} and $\eqref{eq:Cov-Steering-Problem}$, respectively.
    Furthermore, the problem in \eqref{eq:stepBilinear} is feasible if $q \geq r$ and there exist $\curly{\mathbf{\Tilde{U}}_{i}, \mathbf{L}_{i}}_{i=0}^{r-1}$ such that
    \begin{align}\label{eq:prop-feasibility-eq2}
        h_{k}( \mathbf{\Tilde{U}}_{i}, \mathbf{\Tilde{L}}_{i}; \mu_i^0, \Sigma_i^0) \leq \kappa_k 
    \end{align}
    for all $(i, k) \in \curly{0, r-1} \times \curly{0, N-1}$, where $h_k(\mathbf{\Tilde{U}}_{i}, \mathbf{\Tilde{L}}_{i}; \mu_i^0, \Sigma_i^0)$ is defined as in \eqref{eq:h_func_def}.
\end{proposition}
\begin{proof}
The proof is constructive and is done by directly setting the values of the decision variables in \eqref{eq:totalBilinear} and \eqref{eq:stepBilinear}. 
First, we prove the feasibility of the problem in \eqref{eq:totalBilinear}. 
Let $\Tilde{\lambda}_{i, j} = p_i^0$ for all $i \in \curly{0, \dots, r-1}$ with $i = j$, and $\Tilde{\lambda}_{i, j} = 0$, otherwise. Then, by setting $\mathbf{p}_N = \Tilde{\lambda}\transpose \mathbf{1}_r$, the constraint in \eqref{eq:softBilinear-constr0} is satisfied. Note that there exists $\boldsymbol{\beta}$ such that the constraint in \eqref{eq:softBilinear-constr1} holds for any $\mathbf{p}_0$ and $\mathbf{p}_N$. 
Then, let us define $\mathbf{G}_{i} = \Phi_{N:0} + \mathbf{B_N} \mathbf{\Tilde{L}}_i$, $\Sigma^{N}_i = \mathbf{G}_i \Sigma_i^0 \mathbf{G}_i\transpose$, and $\mu_i^N = \Phi_{N:0} \mu_i^0 + \mathbf{B_N} \mathbf{\Tilde{U}}_{i}$.
Now, we set $\mathbf{C}_{i, j} = J_\mathrm{mean}( \mathbf{\Tilde{U}}_{i}; \mu_i^0) + J_\mathrm{cov}(\mathbf{\Tilde{L}}_{i}; \Sigma_i^0)$ for all $i \in \curly{0, \dots, r-1}$ with $i=j$.
Since the optimal value of the right hand side of \eqref{eq:softBilinear-constr2} over $L_{i, j}$ subject to \eqref{eq:softBilinear-constr3} is equal $J_\mathrm{mean}^\star(\mu_i^0, \mu_i^N) + J_\mathrm{cov}^{\star}(\Sigma_i^0, \Sigma_i^N)$, $\mathbf{C}_{i, j}$ satisfy \eqref{eq:softBilinear-constr2} for all $i=j$. 
Furthermore, by setting the rest of $\mathbf{C}_{i, j} =  J_\mathrm{mean}^\star (\mu_i^0, \mu_j^N) + \tr{\Theta_6 \Sigma_i^0} + \tr{\Theta_7 \Sigma_j^N}$ and $L_{i, j} = \bm{0}$, \eqref{eq:softBilinear-constr2} and \eqref{eq:softBilinear-constr3} will be satisfied.
The constraints in \eqref{eq:softBilinear-constr4} and \eqref{eq:softBilinear-constr5} are trivially satisfied, if $Y_{i, j} = 0$ and $\mathbf{T}_{i, j} = \lVert \mu_j^d - \mu_i^N \rVert_2^2 + \tr{\Sigma_i^N + \Sigma_j^d}$.
Finally, the left hand side of the constraint in \eqref{eq:totalBilinear-constr} is equal to $\sum_{i=0}^{r-1} \Tilde{\lambda}_{i, i} C_{i, i}$ since $\Tilde{\lambda}_{i, j} = 0$ by definition for $i \neq j$. 
Since $C_{i, i} \leq \kappa$ from \eqref{eq:prop-feasibility-eq1} and $\Tilde{\lambda}_{i, i} = p_i^0$, \eqref{eq:totalBilinear-constr} holds. 
The proof of the feasibility of \eqref{eq:stepBilinear} follows similarly. 
$\Tilde{\lambda}_{i, j}, \mathbf{T}_{i, j}, \mathbf{p}_N, \boldsymbol{\beta}$ are set using the same method, thus \eqref{eq:softBilinear-constr0}, \eqref{eq:softBilinear-constr1}, \eqref{eq:softBilinear-constr4} and \eqref{eq:softBilinear-constr5} hold. 
Furthermore, we set $\Sigma_i^N$ for all $i \in \{0, \dots, r-1\}$ the same way. 
For all $i \in \curly{0, \dots, r-1}$ with $i = j$, we let $\mathbf{Y}_{i, j} = \mathbf{\Tilde{L}}_{i} \Sigma_i^0$ and $\mathbf{M}_{i, j} = \mathbf{\Tilde{L}}_{i} \Sigma_i^0 \mathbf{\Tilde{L}}_{i}\transpose$, and thus \eqref{eq:stepBilinear-constr0} and \eqref{eq:stepBilinear-constr2} hold for all $i=j$. 
For all other pairs $(i, j)$, $\mathbf{M}_{i, j}$, $\mathbf{\Tilde{L}}_{i, j}$ can be chosen arbitrarily such that \eqref{eq:stepBilinear-constr0} and \eqref{eq:stepBilinear-constr2} hold.
(This can be done since we assumed that the system in \eqref{eq:linear-system-eq} is controllable.) 
Finally, the constraint in \eqref{eq:stepBilinear-constr3} reduces to $\sum_{i=0}^{r-1} p_i^0 h_k (\mathbf{\Tilde{U}}_i, \mathbf{\Tilde{L}}_i; \mu_i^0, \Sigma_i^0) \leq 1$ for all $k$. 
Thus, \eqref{eq:stepBilinear-constr3} holds for all $k$ due to \eqref{eq:prop-feasibility-eq2}. 
\end{proof}

\begin{remark}
    Proposition \ref{prop:gmm-steering-feasibility} asserts that if there are affine feedback control policy parameters that render the linear quadratic constrained optimal control problem feasible for each initial GMM mixture component, then the constrained GMM mixture steering problems will also be feasible.
\end{remark}

\section{Error bounds on GMM Approximations}\label{s:error}
Throughout the paper, we focus on probability distributions represented as GMMs. 
One appealing aspect of GMMs is their universal approximation property \cite[Chapter 3]{b:stergiopoulos2017advSignalProHandbook}, which asserts that any smooth PDF can be approximated to any desired level of precision by a GMM with a sufficiently high number of mixture components. 
In this section, we investigate how the error in the GMM approximation of the initial state distribution evolves over time when the policy defined in \eqref{eq:gmm-policy-definition} is applied to the linear system described in \eqref{eq:linear-system-eq}.

First, let $x_0, x_0^a$, $x_N$ and $x_N^a$ be random variables representing the true and approximated initial states and the true and approximated terminal states, respectively.
In addition, the associated PDFs evaluated at an arbitrary point $x' \in \R{n}$ are denoted as $\P[x_0]{x'}$, $\P[x_0^a]{x'}$, $\P[x_N]{x'}$ and $\P[x_N^a]{x'}$, respectively. 
Moreover, the approximation error for the initial and terminal PDFs are denoted as $e_0(x')$, $e_N(x')$, such that
\begin{align}
    \P[x_k]{x'} & := \P[x_k^a]{x'}  + e_k(x'), ~~k \in\{ 0, N\}. \label{eq:pdf-error-eq12} 
\end{align}
To analyze the terminal state distribution approximation error $e_N(x')$, we need to derive the relationship between $e_N(x')$ and $e_0(x')$.
First, we plug the expression of $\P[x_0]{x'}$ in \eqref{eq:pdf-error-eq12} into the expression of $\P[x_N]{x'}$ in \eqref{eq:pdf-x-n} to obtain: 
\begin{align}
    \P[x_N]{x'} & = \int_{\R{n}} \int_{\R{m N}}  \P[x_N | x_0 = \Hat{x}_0, U = \Hat{U}]{x'} ~ \P[U| x_0=\Hat{x}]{\Hat{U}}  \nonumber \\
    & \qquad \qquad \times \big( \P[x^a_0]{\Hat{x}_0} + e_0(\Hat{x}_0) \big)  \, \mathrm{d}\Hat{U} \, \mathrm{d}\Hat{x}_0, \label{eq:pdf-x-n-error-eq1}\\
    & = \P[x_N^a]{x'} + \int_{\R{n}}\int_{\R{mN}} \P[x_N | x_0 = \Hat{x}_0, U = \Hat{U}]{x'} \nonumber \\
    & \qquad \times \P[U| x_0=\Hat{x}]{\Hat{U}} e_0 (\Hat{x}_0) ~ \mathrm{d}\Hat{U} \mathrm{d} \Hat{x}_0, \label{eq:pdf-x-n-error-eq2}
\end{align}
The first term in \eqref{eq:pdf-error-eq12} is equal to $\P[x_N^a]{x'}$ in view of Proposition \ref{prop:gmm-steering-proposition}. 
Then, it follows from the identity given in \eqref{eq:pdf-error-eq12} that $e_N(x') = \P[x_N]{x'} - \P[x_N^a]{x'}$ and thus,
\begin{align}\label{eq:en-ito-e0}
    e_N(x') & = \int_{\R{n}} \int_{\R{mN}}  \P[x_N | x_0 = \Hat{x}_0, U = \Hat{U}]{x'} \nonumber \\
    & \qquad \qquad \times \P[U| x_0=\Hat{x}]{\Hat{U}} e_0 (\Hat{x}_0) \mathrm{d}\Hat{x}_0 \mathrm{d}\Hat{U}.
\end{align}

To further simplify the expression of $e_N(x')$ given in \eqref{eq:en-ito-e0}, we plug the expressions of $\P[x_N | x_0 = \Hat{x}_0, U = \Hat{U}]{x'}$ and $\P[U| x_0=\Hat{x}]{\Hat{U}} $ given in, respectively, \eqref{eq:conditional-xn-given-x0-u} and \eqref{eq:conditional-U-given-x0}, into \eqref{eq:en-ito-e0}.
Then, similarly to the proof of Proposition \ref{prop:gmm-steering-proposition}, it can be shown that $e_N(x')$ can be rewritten as:
\begin{align*}
    e_N(x') := \sum_{i, j} \int_{\R{n}} \gamma_{i, j}(\Hat{x}_0) \delta (x' = \mathbf{H}_{i, j} \Hat{x}_0 - h_{i, j}) 
    e_0(\Hat{x}_0) \mathrm{d} \Hat{x}_0, 
\end{align*}
where $\mathbf{H}_{i, j} = \Phi_{N:0} + \mathbf{B_N} \mathbf{L}_{i, j}$, $h_{i, j} = \mathbf{B_N} (\mathbf{L}_{i, j} \mu_i^0 - \mathbf{\Bar{U}}_{i, j})$ and $\gamma(\Hat{x}_0)$ is given as in \eqref{eq:gamma-def}. 
Applying the variable transformation $z_{i, j} := \mathbf{H}_{i, j} \Hat{x}_0 - h_{i, j}$ for all $(i, j)$, expanding $\gamma(\Hat{x}_0)$ and using the standard property of Dirac delta function, we obtain the following expression for $e_N(x')$:
\begin{align}\label{eq:en-ito-e0-3}
    e_N(x') := \sum_{i=0, j=0}^{r-1, q-1} p_i^0 \lambda_{i, j} \frac{e_0( \mathbf{H}_{i, j}^{-1} (x' + h_{i, j}) )}{\P[x_0^a]{\mathbf{H}_{i, j}^{-1} (x' + h_{i, j})}} \nonumber \\
    \times \P[\mathcal{N}]{x'; \mu_j^N, \Sigma_j^N}.
\end{align}
For notational brevity, we can write \eqref{eq:en-ito-e0-3} as $e_N(x') = \sum_{i=0, j=0}^{r-1, q-1} g_{i, j}(x')$ where each $g_{i, j}(x')$ is a term in the summation in \eqref{eq:en-ito-e0-3}. 
The decomposed expression of $e_N(x')$ given in \eqref{eq:en-ito-e0-3} will be used in the subsequent analysis.

To find a bound, we first assume that the absolute value of the ratio of the initial error and the approximated GMM mixture PDF is upper bounded by some $\epsilon > 0$. 
Then, we show that this upper bound holds for the terminal state error term $e_N(x')$. 
The following proposition formally states the previous claim.
\begin{proposition}\label{prop:error-bound1}
    If the GMM approximation error of the PDF of the initial state, $e_0(x')$, satisfy $\left| {e_0(x')}/{\P[x_0^a]{x'}} \right| \leq \epsilon$, $\forall x'$, then the GMM approximation error of the PDF of the terminal state, $e_N(x')$, will also satisfy $\left| {e_N(x')}/{\P[x_N^a]{x'}} \right| \leq \epsilon$, $\forall x'$.
\end{proposition}
\begin{proof}
Let $\Bar{g}_{i, j}(x') := p_i^0 \lambda_{i, j} \P[\mathcal{N}]{x'; \mu_j^N, \Sigma_j^N} \epsilon$ and $\underline{g}_{i, j} := - \Bar{g}_{i, j}(x')$. 
Observe that $\underline{g}_{i, j}(x') \leq g_{i, j}(x') \leq \Bar{g}_{i, j}(x')$, $\forall x' \in \R{n}$ and $\forall (i, j)$ pairs since $\left| {e_0(x')}/{\P[x_0^a]{x'}} \right| \leq \epsilon$, $\forall x' \in \R{n}$.
Now, summing the terms $\Bar{g}_{i, j}(x')$ over all $(i, j)$, we obtain:
\begin{align*}
   \sum_{i=0, j=0}^{r-1, q-1} \Bar{g}_{i, j}(x') & = \sum_{j=0}^{q-1} p_j^{N} \P[\mathcal{N}]{x'; \mu_j^N, \Sigma_j^N} \epsilon 
    = \P[x_N^a]{x'} \epsilon 
\end{align*}
Similarly, we have $\sum_{i, j} \underline{g}_{i, j}(x') = - \P[x_N^a]{x'} \epsilon$. 
Thus, we have that $-\P[x_N^a]{x'} \epsilon \leq e_N(x') \leq \P[x_N^a]{x'} \epsilon$.
Furthermore, dividing both sides of the latter inequality by $\P[x_N^a]{x'}$, we obtain $-\epsilon \leq e_N(x')/\P[x_N^a]{x'} \leq \epsilon$ which completes the proof.
\end{proof}

Next, we assume that $|e_0(x')| \leq \epsilon$, $\forall x'$, where $\epsilon_0 > 0$, and derive an upper bound for $|e_N(x')|$ in terms of $\epsilon_0$ for all $x'$. 
\begin{proposition}\label{prop:error-bound2}
    If the GMM approximation error of the PDF of the initial state distribution $e_0(x')$ satisfy $\left| e_0(x') \right| \leq \epsilon_0$, then the GMM approximation error of the terminal state distribution $e_N(x')$ satisfies: 
    \begin{align}\label{eq:eps-n-def}
        |e_N(x')| \leq \sum_{i=0, j=0}^{r-1, q-1} \lambda_{i, j} \sqrt{\det (\Sigma_i^0) / \det (\Sigma_j^N)  }  \epsilon_0
    \end{align}
\end{proposition}
\begin{proof}
First, we observe that $e_N(x')$, which is given in \eqref{eq:en-ito-e0-3}, can be written alternatively as follows:
\begin{align}\label{eq:en-ito-e0-4}
     e_N(x') = \sum_{i=0, j=0}^{r-1, q-1} \lambda_{i, j} \frac{p_i^0 \P[\mathcal{N}]{z_{i, j}(x'); \mu_i^0, \Sigma_i^0}}{\P[x_0^a]{z_{i, j}(x')}} \nonumber \\
     \times e_0(z_{i, j}(x')) \det (\mathbf{H}_{i, j}^{-1}),
\end{align}
where $z_{i, j}(x') = \mathbf{H}_{i, j}^{-1} (x' + h_{i, j})$.
Observe that $e_N(x')$ can be decomposed as $\sum_{i=0, j=0}^{r-1, q-1} g_{i, j}(x')$ where
\begin{align}
    g_{i, j}(x') := \lambda_{i, j} \frac{p_i^0 \P[\mathcal{N}]{z_{i, j}(x'); \mu_i^0, \Sigma_i^0}}{\P[x_0^a]{z_{i, j}(x')}} \frac{e_0(z_{i, j}(x'))}{\det (\mathbf{H}_{i, j})}. 
\end{align}
We observe that $\frac{p_i^0 \P[\mathcal{N}]{z_{i, j}(x'); \mu_i^0, \Sigma_i^0}}{\P[x_0^a]{z_{i, j}(x')}} \in [0,1]$ since the numerator is a component of the Gaussian mixture in the denominator. 
Moreover, we have $|e_0(z_{i, j}(x'))| \leq \epsilon_0$ from the initial assumption of Proposition \ref{prop:error-bound2}. Thus, $\underline{g}_{i, j}(x') \leq g_{i, j}(x') \leq \Bar{g}_{i, j}(x')$ where $\Bar{g}_{i, j} := \lambda_{i, j} \epsilon_0 \det (\mathbf{H}_{i, j}^{-1})$ and $\underline{g}_{i, j}(x') = - \Bar{g}_{i, j}(x')$.
Summing $g_{i, j}(x')$, $\forall (i, j)$, we obtain that $\sum_{i, j} \underline{g}_{i, j}(x') \leq e_N(x') \leq \sum_{i, j}\Bar{g}_{i, j}$; thus $|e_N(x')| \leq \sum_{i, j} \Bar{g}_{i, j}(x')$.
In addition, $\Sigma_j^N = \mathbf{H}_{i, j} \Sigma_i^0 \mathbf{H}_{i, j}\transpose$ $\forall (i, j)$. 
We obtain $\det (\mathbf{H}_{i, j})^{2} = \det(\Sigma_j^N)/\det(\Sigma_i^0)$ and thus, $\det(\mathbf{H}_{i, j}^{-1}) = \sqrt{\det(\Sigma_i^0)/\det(\Sigma_j^N)}$, which completes the proof.
\end{proof}

\begin{remark}
    Proposition \ref{prop:error-bound1} guarantees that the ratio of the approximation error magnitude to the PDF of the GMM evaluated at $x' \in \R{n}$ remains constant for all $x'$ after the policy in \eqref{eq:gmm-policy-definition} is applied. 
    In contrast, Proposition \ref{prop:error-bound2} provides an expression for the absolute approximation error of the PDF of the terminal state GMM in terms of the absolute value of the initial state approximation error.
\end{remark}

\section{Numerical Simulations}\label{s:numerical-simulations}
In this section, we present the results of numerical experiments which were conducted on a Mac M1 with 8 GB of RAM. The first numerical experiment (Section \ref{s:num-exp-prob1}) focuses on Problem \ref{prob:hard-GMM-density-steering}, whereas 
the second set of numerical experiments (Section \ref{s:num-exp-prob23}) correspond to Problems \ref{prob:soft-GMM-density-steering}-\ref{prob:total-cost-constrained-GMM-density-steering}. 
Lastly, the third set of numerical experiments (Section \ref{s:num-exp-prob4}) correspond to Problem \ref{prob:step-cost-constrained-GMM-density-steering}.

\subsection{2D Single Integrator - Problem \ref{prob:hard-GMM-density-steering}}\label{s:num-exp-prob1}
In this experiment, the system dynamics matrices are taken as $A_k = I_2$, $B_k = \Delta t I_2$, $\Delta t = 1.0$, $x_k, u_k \in \R{2}$ for all $k \in \curly{0, \dots, N}$ where $N=10$. In addition, $R_k = I_2$, $Q_k = \bm{0}$ and $x_k' = \bm{0}$. The initial state follows a uniform distribution over the set $\mathcal{S}_{\mathrm{init}} := \curly{ p_x, p_y \in \R{} ~|~ (p_x, p_y) \in [-1, 8] \times [-1, 4] }$. The terminal state distribution is a uniform distribution over a `UT' shape set. The actual initial and desired distributions are illustrated in Fig.~\ref{fig:init-des-actual-ex1}, whereas the approximated initial and desired distributions are shown in Fig.~\ref{fig:init-des-approx-ex1}.

\begin{figure}[ht]
    \centering
    \begin{subfigure}{0.45\linewidth}
        \centering
        \resizebox{\linewidth}{!}{
\begin{tikzpicture}

\definecolor{darkgray176}{RGB}{176,176,176}

\begin{axis}[
axis equal image,
tick align=outside,
tick pos=left,
x grid style={darkgray176},
xmin=-5, xmax=10,
xtick style={color=black},
y grid style={darkgray176},
ymin=-2, ymax=5,
ytick style={color=black}, 
ticklabel style = {font=\Large} 
]
\addplot graphics [includegraphics cmd=\pgfimage,xmin=-5, xmax=10, ymin=-2, ymax=5] {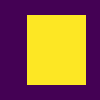};
\end{axis}

\end{tikzpicture}
        }
        \subcaption{Initial}\label{fig:initial_dist_actual_ex1}
    \end{subfigure}
    \hfill
    \begin{subfigure}{0.45\linewidth}
        \centering
        \resizebox{\linewidth}{!}{
            \begin{tikzpicture}

\definecolor{darkgray176}{RGB}{176,176,176}

\begin{axis}[
axis equal image,
tick align=outside,
tick pos=left,
x grid style={darkgray176},
xmin=-5, xmax=10,
xtick style={color=black},
y grid style={darkgray176},
ymin=-2, ymax=5,
ytick style={color=black}, 
ticklabel style={font=\Large}
]
\addplot graphics [includegraphics cmd=\pgfimage,xmin=-5, xmax=10, ymin=-2, ymax=5] {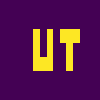};
\end{axis}

\end{tikzpicture}
        }
        \centering
        \subcaption{Desired}\label{fig:desired_dist_actual_ex1}
    \end{subfigure}
    \caption{Actual densities.}
    \label{fig:init-des-actual-ex1}
\end{figure}

\begin{figure}[ht]
    \centering
    \begin{subfigure}{0.45\linewidth}
        \centering
        \resizebox{\linewidth}{!}{
\begin{tikzpicture}

\definecolor{darkgray176}{RGB}{176,176,176}

\begin{axis}[
axis equal image,
tick align=outside,
tick pos=left,
x grid style={darkgray176},
xmin=-5, xmax=10,
xtick style={color=black},
y grid style={darkgray176},
ymin=-2, ymax=5,
ytick style={color=black}, 
ticklabel style = {font=\Large} 
]
\addplot graphics [includegraphics cmd=\pgfimage,xmin=-5, xmax=10, ymin=-2, ymax=5] {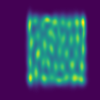};
\end{axis}

\end{tikzpicture}
        }
        \subcaption{Initial}\label{fig:initial_dist_approx_ex1}
    \end{subfigure}
    \hfill
    \begin{subfigure}{0.45\linewidth}
        \centering
        \resizebox{\linewidth}{!}{
            \begin{tikzpicture}

\definecolor{darkgray176}{RGB}{176,176,176}

\begin{axis}[
axis equal image,
tick align=outside,
tick pos=left,
x grid style={darkgray176},
xmin=-5, xmax=10,
xtick style={color=black},
y grid style={darkgray176},
ymin=-2, ymax=5,
ytick style={color=black}, 
ticklabel style = {font=\Large} 
]
\addplot graphics [includegraphics cmd=\pgfimage,xmin=-5, xmax=10, ymin=-2, ymax=5] {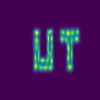};
\end{axis}

\end{tikzpicture}
        }
        \centering
        \subcaption{Desired}\label{fig:desired_dist_approx_ex1}
    \end{subfigure}
    \caption{Approximated densities.}
    \label{fig:init-des-approx-ex1}
\end{figure}

For the initial state PDF approximation, the number of GMM components is chosen as 40, while for the desired state PDF as 30 (Figs \ref{fig:init-des-actual-ex1} and \ref{fig:init-des-approx-ex1}). 
The evolution of the PDF of $x_k$ under the policy given in \eqref{eq:gmm-policy-definition}, obtained by solving the LP in \eqref{eq:hardLP}, is shown for time steps $\{ {2, 5, 8, 10} \}$ in Figure \ref{fig:evolution-state-density-ex1}, from left to right.
For both the initial GMM approximation using the EM algorithm and the illustration of the state PDF evolution, 10,000 samples are generated. 

In Fig.~\ref{fig:approx_init_term_dist_vs_r}, the PDF of the initial state distribution is approximated by a GMM whose number of Gaussian components, $r$, varies. Specifically, from left to right, $r\in\{2, 5, 15, 30\}$. 
The top row (Figs.~\ref{subfig:approx_init_vs_r2}–\ref{subfig:approx_init_vs_r30}) illustrates the GMM approximations of the PDF of the initial state shown in Fig.~\ref{fig:initial_dist_actual_ex1} for different values of $r$. The bottom row (Figs.~\ref{subfig:term_dist_vs_r2}–\ref{subfig:term_dist_vs_r30}) shows the terminal PDFs after the optimal policies are applied to the system.

We also investigated the effect of the number of components in the GMM approximation of the initial state distribution on the optimal value of 
Problem \ref{prob:hard-GMM-density-steering} and the computation time.
As observed in Figure \ref{subfig:opt_val_vs_num_gmm}, the optimal value of the problem decreases until $r=30$, since the terminal distribution is also approximated with a GMM with 30 components. 
On the other hand, the computation time increases linearly, which is expected since the problem size increases linearly with $r$ as shown in Figure \ref{subfig:time_vs_num_gmm}. 
Notably, solving the LP in \eqref{eq:hardLP} takes only a fraction of a second by the solver; however, building the problem and evaluating each $\mathbf{C}_{i, j}$ takes most of the time.

\begin{figure*}[ht]
    \centering
    \begin{subfigure}{0.23\linewidth}
        \centering
        \resizebox{\linewidth}{!}{
            \begin{tikzpicture}

\definecolor{darkgray176}{RGB}{176,176,176}

\begin{axis}[
axis equal image,
tick align=outside,
tick pos=left,
x grid style={darkgray176},
xmin=-5, xmax=10,
xtick style={color=black},
y grid style={darkgray176},
ymin=-2, ymax=5,
ytick style={color=black}, 
ticklabel style={font=\Large}
]
\addplot graphics [includegraphics cmd=\pgfimage,xmin=-5, xmax=10, ymin=-2, ymax=5] {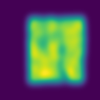};
\end{axis}

\end{tikzpicture}
        }
        \subcaption{$t=2$}
    \end{subfigure}
    \hfill
    \begin{subfigure}{0.23\linewidth}
        \centering
        \resizebox{\linewidth}{!}{
            \begin{tikzpicture}

\definecolor{darkgray176}{RGB}{176,176,176}

\begin{axis}[
axis equal image,
tick align=outside,
tick pos=left,
x grid style={darkgray176},
xmin=-5, xmax=10,
xtick style={color=black},
y grid style={darkgray176},
ymin=-2, ymax=5,
ytick style={color=black}, 
ticklabel style={font=\Large}
]
\addplot graphics [includegraphics cmd=\pgfimage,xmin=-5, xmax=10, ymin=-2, ymax=5] {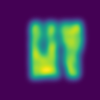};
\end{axis}

\end{tikzpicture}
        }
        \subcaption{$t=5$}
    \end{subfigure}
    \hfill
    \begin{subfigure}{0.23\linewidth}
        \centering
        \resizebox{\linewidth}{!}{
            \begin{tikzpicture}

\definecolor{darkgray176}{RGB}{176,176,176}

\begin{axis}[
axis equal image,
tick align=outside,
tick pos=left,
x grid style={darkgray176},
xmin=-5, xmax=10,
xtick style={color=black},
y grid style={darkgray176},
ymin=-2, ymax=5,
ytick style={color=black}, 
ticklabel style={font=\Large}
]
\addplot graphics [includegraphics cmd=\pgfimage,xmin=-5, xmax=10, ymin=-2, ymax=5] {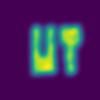};
\end{axis}

\end{tikzpicture}
        }
        \subcaption{$t=8$}
    \end{subfigure}
    \hfill
    \begin{subfigure}{0.23\linewidth}
        \centering
        \resizebox{\linewidth}{!}{
            \begin{tikzpicture}

\definecolor{darkgray176}{RGB}{176,176,176}

\begin{axis}[
axis equal image,
tick align=outside,
tick pos=left,
x grid style={darkgray176},
xmin=-5, xmax=10,
xtick style={color=black},
y grid style={darkgray176},
ymin=-2, ymax=5,
ytick style={color=black}, 
ticklabel style={font=\Large}
]
\addplot graphics [includegraphics cmd=\pgfimage,xmin=-5, xmax=10, ymin=-2, ymax=5] {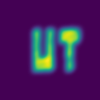};
\end{axis}

\end{tikzpicture}
        }
        \subcaption{$t=10$}
    \end{subfigure}
    \caption{Evolution of state density (PDF).}
    \label{fig:evolution-state-density-ex1}
\end{figure*}

\begin{figure*}
    \centering
    \begin{subfigure}{0.23\linewidth}
        \centering
        \resizebox{\linewidth}{!}{
            \begin{tikzpicture}

\definecolor{darkgray176}{RGB}{176,176,176}

\begin{axis}[
axis equal image,
tick align=outside,
tick pos=left,
x grid style={darkgray176},
xmin=-5, xmax=10,
xtick style={color=black},
y grid style={darkgray176},
ymin=-2, ymax=5,
ytick style={color=black}, 
ticklabel style = {font=\Large} 
]
\addplot graphics [includegraphics cmd=\pgfimage,xmin=-5, xmax=10, ymin=-2, ymax=5] {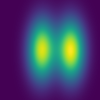};
\end{axis}

\end{tikzpicture}
        }
        \subcaption{$r = 2$}
        \label{subfig:approx_init_vs_r2}
    \end{subfigure}
    \hfill
    \begin{subfigure}{0.23\linewidth}
        \centering
        \resizebox{\linewidth}{!}{
            \begin{tikzpicture}

\definecolor{darkgray176}{RGB}{176,176,176}

\begin{axis}[
axis equal image,
tick align=outside,
tick pos=left,
x grid style={darkgray176},
xmin=-5, xmax=10,
xtick style={color=black},
y grid style={darkgray176},
ymin=-2, ymax=5,
ytick style={color=black}, 
ticklabel style = {font=\Large} 
]
\addplot graphics [includegraphics cmd=\pgfimage,xmin=-5, xmax=10, ymin=-2, ymax=5] {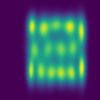};
\end{axis}

\end{tikzpicture}
        }
        \subcaption{$r = 5$}
        \label{subfig:approx_init_vs_r5}
    \end{subfigure}
    \hfill
    \begin{subfigure}{0.23\linewidth}
        \centering
        \resizebox{\linewidth}{!}{
            \begin{tikzpicture}

\definecolor{darkgray176}{RGB}{176,176,176}

\begin{axis}[
axis equal image,
tick align=outside,
tick pos=left,
x grid style={darkgray176},
xmin=-5, xmax=10,
xtick style={color=black},
y grid style={darkgray176},
ymin=-2, ymax=5,
ytick style={color=black}, 
ticklabel style = {font=\Large} 
]
\addplot graphics [includegraphics cmd=\pgfimage,xmin=-5, xmax=10, ymin=-2, ymax=5] {example1/approx_init_N15.png};
\end{axis}

\end{tikzpicture}
        }
        \subcaption{$r = 15$}
        \label{subfig:approx_init_vs_r15}
    \end{subfigure}
    \hfill
    \begin{subfigure}{0.23\linewidth}
        \centering
        \resizebox{\linewidth}{!}{
            \begin{tikzpicture}

\definecolor{darkgray176}{RGB}{176,176,176}

\begin{axis}[
axis equal image,
tick align=outside,
tick pos=left,
x grid style={darkgray176},
xmin=-5, xmax=10,
xtick style={color=black},
y grid style={darkgray176},
ymin=-2, ymax=5,
ytick style={color=black}, 
ticklabel style = {font=\Large} 
]
\addplot graphics [includegraphics cmd=\pgfimage,xmin=-5, xmax=10, ymin=-2, ymax=5] {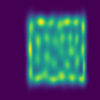};
\end{axis}

\end{tikzpicture}
        }
        \subcaption{$r = 30$}
        \label{subfig:approx_init_vs_r30}
    \end{subfigure}
    \hfill
    \begin{subfigure}{0.23\linewidth}
        \centering
        \resizebox{\linewidth}{!}{
            \begin{tikzpicture}

\definecolor{darkgray176}{RGB}{176,176,176}

\begin{axis}[
axis equal image,
tick align=outside,
tick pos=left,
x grid style={darkgray176},
xmin=-5, xmax=10,
xtick style={color=black},
y grid style={darkgray176},
ymin=-2, ymax=5,
ytick style={color=black}
]
\addplot graphics [includegraphics cmd=\pgfimage,xmin=-5, xmax=10, ymin=-2, ymax=5] {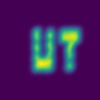};
\end{axis}

\end{tikzpicture}
        }
        \subcaption{$r = 2$}
        \label{subfig:term_dist_vs_r2}
    \end{subfigure}
    \hfill
    \begin{subfigure}{0.23\linewidth}
        \centering
        \resizebox{\linewidth}{!}{
            \begin{tikzpicture}

\definecolor{darkgray176}{RGB}{176,176,176}

\begin{axis}[
axis equal image,
tick align=outside,
tick pos=left,
x grid style={darkgray176},
xmin=-5, xmax=10,
xtick style={color=black},
y grid style={darkgray176},
ymin=-2, ymax=5,
ytick style={color=black}
]
\addplot graphics [includegraphics cmd=\pgfimage,xmin=-5, xmax=10, ymin=-2, ymax=5] {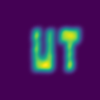};
\end{axis}

\end{tikzpicture}
        }
        \subcaption{$r = 5$}
        \label{subfig:term_dist_vs_r5}
    \end{subfigure}
    \hfill
    \begin{subfigure}{0.23\linewidth}
        \centering
        \resizebox{\linewidth}{!}{
            \begin{tikzpicture}

\definecolor{darkgray176}{RGB}{176,176,176}

\begin{axis}[
axis equal image,
tick align=outside,
tick pos=left,
x grid style={darkgray176},
xmin=-5, xmax=10,
xtick style={color=black},
y grid style={darkgray176},
ymin=-2, ymax=5,
ytick style={color=black}
]
\addplot graphics [includegraphics cmd=\pgfimage,xmin=-5, xmax=10, ymin=-2, ymax=5] {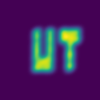};
\end{axis}

\end{tikzpicture}
        }
        \subcaption{$r = 15$}
        \label{subfig:term_dist_vs_r15}
    \end{subfigure}
    \hfill
    \begin{subfigure}{0.23\linewidth}
        \centering
        \resizebox{\linewidth}{!}{
            \begin{tikzpicture}

\definecolor{darkgray176}{RGB}{176,176,176}

\begin{axis}[
axis equal image,
tick align=outside,
tick pos=left,
x grid style={darkgray176},
xmin=-5, xmax=10,
xtick style={color=black},
y grid style={darkgray176},
ymin=-2, ymax=5,
ytick style={color=black}
]
\addplot graphics [includegraphics cmd=\pgfimage,xmin=-5, xmax=10, ymin=-2, ymax=5] {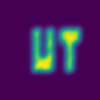};
\end{axis}

\end{tikzpicture}
        }
        \subcaption{$r = 30$}
        \label{subfig:term_dist_vs_r30}
    \end{subfigure}
    \caption{Approximated initial and actual terminal densities for varying $q$ (the number of terminal GMM components).}
    \label{fig:approx_init_term_dist_vs_r}
\end{figure*}
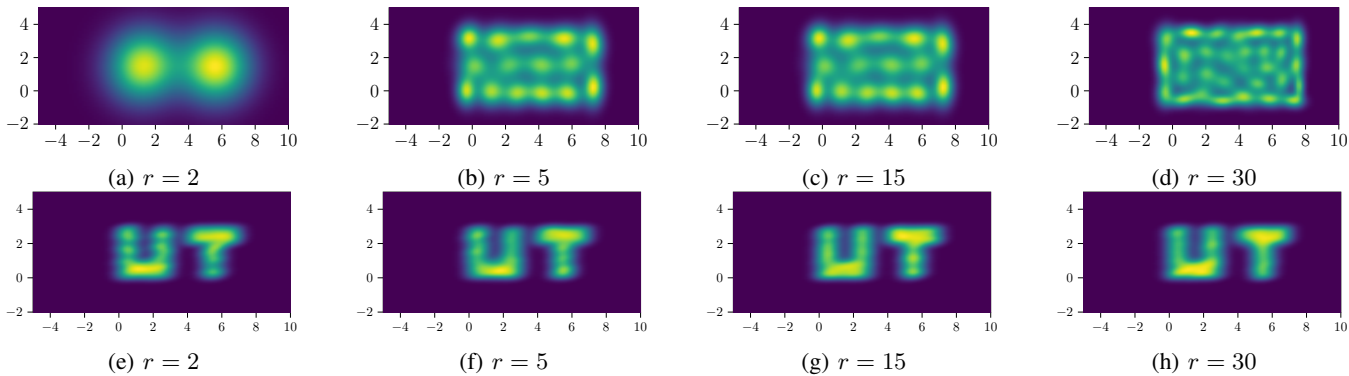

\begin{figure}[ht]
    \centering
    \begin{subfigure}{0.9\linewidth}
    \resizebox{\linewidth}{!}{
        \begin{tikzpicture}
    \begin{axis}[
        xlabel={Num. GMM Comp.},
        ylabel={Opt. Val.},
        grid=major,
        height=4cm,
        width=10cm,
        legend style={at={(0.5,-0.15)},anchor=north,legend columns=-1},
    ]
    \pgfplotstableread{
        x y
        5 0.25684451406133385
        10 0.2059592792914832
        15 0.14088232365860848
        20 0.13345818822882277
        25 0.11924824174212871
        30 0.11843972105695155
        35 0.11914294105126935
        40 0.11439128003306345
        45 0.1110212947029618
        50 0.10811556439378005
    }\loadedtable;

    \addplot[
        color=blue,
        mark=*,
        ] 
    table[x=x,y=y] {\loadedtable};
    \end{axis}
\end{tikzpicture}
    }
    \subcaption{Optimum Value vs $r$}
    \label{subfig:opt_val_vs_num_gmm}
    \end{subfigure}
    \hfill
    \begin{subfigure}{0.9\linewidth}
    \resizebox{\linewidth}{!}{
        \begin{tikzpicture}
    \begin{axis}[
        xlabel={Num. GMM Comp.},
        ylabel={Time (s)},
        grid=major,
        height=4cm,
        width=10cm,
        legend style={at={(0.5,-0.15)},anchor=north,legend columns=-1},
    ]
    \pgfplotstableread{
        x y
        5 0.10639190673828125
        10 0.18224191665649414
        15 0.22929000854492188
        20 0.26682591438293457
        25 0.342832088470459
        30 0.37670397758483887
        35 0.44867467880249023
        40 0.511260986328125
        45 0.5750641822814941
        50 0.6315748691558838
    }\loadedtable;

    \addplot[
        color=blue,
        mark=*,
        ] 
    table[x=x,y=y] {\loadedtable};
    \end{axis}
\end{tikzpicture}
    }
    \subcaption{Computation Time vs $r$}
    \label{subfig:time_vs_num_gmm}
    \end{subfigure}
    \caption{Optimal Value and Computation time vs the number of components, $r$, of initial GMM}
    \label{fig:opt_val_time_vs_num_gmm}
\end{figure}
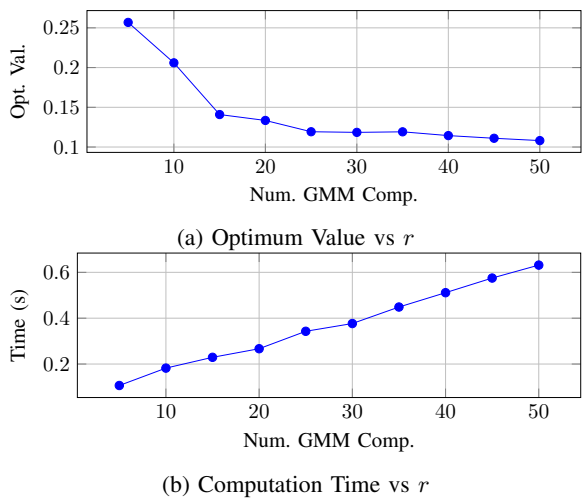

\subsection{2D Single Integrator - Problems \ref{prob:soft-GMM-density-steering} and \ref{prob:total-cost-constrained-GMM-density-steering}}\label{s:num-exp-prob23}
Furthermore, we solve Problems \ref{prob:soft-GMM-density-steering} and \ref{prob:total-cost-constrained-GMM-density-steering} for the 2D single integrator system defined in Section \ref{s:num-exp-prob1}. 
In these numerical experiments, we take $Q_k = I_2$. The initial state distribution is represented as a GMM with 3 components, while the desired distribution is a ``C''-shaped uniform distribution over 2D Euclidean space. 
The desired distribution is approximated using the EM algorithm as a GMM with 10 components. 
The initial and approximated desired densities are shown in Fig.~\ref{fig:init-des-ex2}.

\begin{figure}
    \centering
    \begin{subfigure}{0.48\linewidth}
        \centering
        \resizebox{\linewidth}{!}{
            \begin{tikzpicture}

\definecolor{darkgray176}{RGB}{176,176,176}

\begin{axis}[
axis equal image,
tick align=outside,
tick pos=left,
x grid style={darkgray176},
xmin=-10, xmax=20,
xtick style={color=black},
y grid style={darkgray176},
ymin=-10, ymax=20,
ytick style={color=black}
]
\addplot graphics [includegraphics cmd=\pgfimage,xmin=-10, xmax=20, ymin=-10, ymax=20] {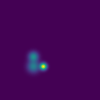};
\end{axis}

\end{tikzpicture}
        }
        \subcaption{Initial State PDF}
        \label{subfig:init-ex2}
    \end{subfigure}
    \hfill
    \begin{subfigure}{0.48\linewidth}
        \centering
        \resizebox{\linewidth}{!}{
            \begin{tikzpicture}

\definecolor{darkgray176}{RGB}{176,176,176}

\begin{axis}[
axis equal image,
tick align=outside,
tick pos=left,
x grid style={darkgray176},
xmin=-10, xmax=20,
xtick style={color=black},
y grid style={darkgray176},
ymin=-10, ymax=20,
ytick style={color=black}
]
\addplot graphics [includegraphics cmd=\pgfimage,xmin=-10, xmax=20, ymin=-10, ymax=20] {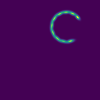};
\end{axis}

\end{tikzpicture}
        }
        \subcaption{Desired PDF}
        \label{subfig:des-ex2}
    \end{subfigure}
    \caption{Initial and desired PDFs}
    \label{fig:init-des-ex2}
\end{figure}


First, we solve Problem \ref{prob:soft-GMM-density-steering} for varying parameters $\kappa \in \{ 0.05, 0.2, 2.5, 50.0 \}$. The terminal state PDFs obtained after applying the optimal policy for each $\kappa$ are shown in Fig.~\ref{fig:prob2-terminal-dist-vs-kappa}. We observe therein that as $\kappa$ increases, the discrepancy between the terminal and desired distributions is penalized more heavily, resulting in a terminal state distribution that resembles the desired distribution more closely. 
Furthermore, 
when $\kappa$ is small, the terminal distribution tends to be less dispersed.

Even with $\kappa = 50.0$, the terminal density shown in Fig.~\ref{subfig:prob2-term-dist-kappa50p0} does not exactly match the desired distribution shown in Fig.~\ref{subfig:des-ex2}. 
This is because the terminal GMM has 5 components whereas the desired GMM has 10. 
The terminal distribution shown in Figure \ref{subfig:prob2-term-dist-kappa50p0} is the one that minimizes the GMM-Wasserstein distance. 
This can be verified in Figure \ref{fig:GMM-Wass-vs-kappa}, where the value of the GMM-Wasserstein distance between the terminal and desired distributions is plotted against $\kappa$. 
As $\kappa$ increases, the GMM-Wasserstein distance between the terminal state distribution and the desired one decreases but converges to $1.09$.

\begin{figure*}
    \centering
    \begin{subfigure}{0.24\linewidth}
        \centering
        \resizebox{\linewidth}{!}{
            \begin{tikzpicture}

\definecolor{darkgray176}{RGB}{176,176,176}

\begin{axis}[
axis equal image,
tick align=outside,
tick pos=left,
x grid style={darkgray176},
xmin=-10, xmax=20,
xtick style={color=black},
y grid style={darkgray176},
ymin=-10, ymax=20,
ytick style={color=black}
]
\addplot graphics [includegraphics cmd=\pgfimage,xmin=-10, xmax=20, ymin=-10, ymax=20] {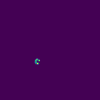};
\end{axis}

\end{tikzpicture}
        }
        \subcaption{$\kappa = 0.05$}
        \label{subfig:prob2-term-dist-kappa0p05}
    \end{subfigure}
    \hfill
    \begin{subfigure}{0.24\linewidth}
        \centering
        \resizebox{\linewidth}{!}{
            \begin{tikzpicture}

\definecolor{darkgray176}{RGB}{176,176,176}

\begin{axis}[
axis equal image,
tick align=outside,
tick pos=left,
x grid style={darkgray176},
xmin=-10, xmax=20,
xtick style={color=black},
y grid style={darkgray176},
ymin=-10, ymax=20,
ytick style={color=black}
]
\addplot graphics [includegraphics cmd=\pgfimage,xmin=-10, xmax=20, ymin=-10, ymax=20] {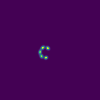};
\end{axis}

\end{tikzpicture}
        }
        \subcaption{$\kappa = 0.2$}
        \label{subfig:prob2-term-dist-kappa0p20}
    \end{subfigure}
    \hfill
    \begin{subfigure}{0.24\linewidth}
        \centering
        \resizebox{\linewidth}{!}{
            \begin{tikzpicture}

\definecolor{darkgray176}{RGB}{176,176,176}

\begin{axis}[
axis equal image,
tick align=outside,
tick pos=left,
x grid style={darkgray176},
xmin=-10, xmax=20,
xtick style={color=black},
y grid style={darkgray176},
ymin=-10, ymax=20,
ytick style={color=black}
]
\addplot graphics [includegraphics cmd=\pgfimage,xmin=-10, xmax=20, ymin=-10, ymax=20] {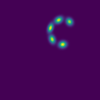};
\end{axis}

\end{tikzpicture}
        }
        \subcaption{$\kappa = 2.5$}
        \label{subfig:prob2-term-dist-kappa2p5}
    \end{subfigure}
    \hfill
    \begin{subfigure}{0.24\linewidth}
        \centering
        \resizebox{\linewidth}{!}{
            \begin{tikzpicture}

\definecolor{darkgray176}{RGB}{176,176,176}

\begin{axis}[
axis equal image,
tick align=outside,
tick pos=left,
x grid style={darkgray176},
xmin=-10, xmax=20,
xtick style={color=black},
y grid style={darkgray176},
ymin=-10, ymax=20,
ytick style={color=black}
]
\addplot graphics [includegraphics cmd=\pgfimage,xmin=-10, xmax=20, ymin=-10, ymax=20] {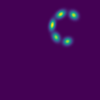};
\end{axis}

\end{tikzpicture}
        }
        \subcaption{$\kappa = 50.0$}
        \label{subfig:prob2-term-dist-kappa50p0}
    \end{subfigure}
    \caption{Terminal State PDF vs $\kappa$ for Problem \ref{prob:soft-GMM-density-steering}}
    \label{fig:prob2-terminal-dist-vs-kappa}
\end{figure*}

\begin{figure}
    \centering
    \resizebox{\linewidth}{!}{
        \begin{tikzpicture}
    \begin{loglogaxis}[
        xlabel={$\kappa$},
        ylabel={$W_{\mathrm{GMM}}^{2}(\rho_N, \rho_d)$},
        grid=major,
        height=4cm,
        width=10cm,
        legend style={at={(0.5,-0.15)},anchor=north,legend columns=-1},
    ]
    \pgfplotstableread{
        x y
        0.01 228.76878149662977
        0.05 189.65280295406347
        0.2 103.20626249911551
        0.5 45.21065941232787
        1.0 19.14328475705691
        2.5 5.144577223512199
        5.0 2.6094132672404973
        10.0 1.680652557329921
        25.0 1.1265913445130804
        50.0 1.1158754531253126
        100.0 1.091783036146476
        250.0 1.324877832989902
        1000.0 1.099676098889404
    }\loadedtable;

    \addplot[
        color=blue,
        mark=*,
        ] 
    table[x=x,y=y] {\loadedtable};
    \end{loglogaxis}
\end{tikzpicture}
    }
    \caption{GMM-Wasserstein Distance vs $\kappa$ for Problem \ref{prob:soft-GMM-density-steering}}
    \label{fig:GMM-Wass-vs-kappa}
\end{figure}
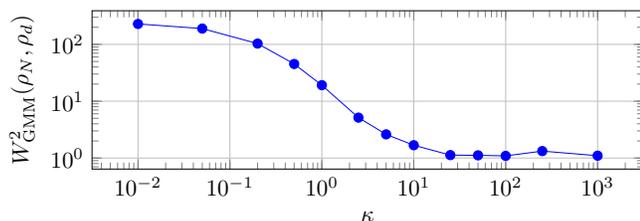

We also solve Problem \ref{prob:total-cost-constrained-GMM-density-steering} with varying parameter $\kappa$.
The terminal state distributions for $\kappa \in \{ 10.0, 25.0, 50.0, 100.0 \}$ are shown in Figure \ref{fig:prob3-terminal-dist-vs-kappa}.
In Problem \ref{prob:total-cost-constrained-GMM-density-steering}, parameter $\kappa$ represents the total quadratic cost that can be accumulated over problem horizon.
As $\kappa$ is increased from $10.0$ to $100.0$, the total cost constraint in \eqref{eq:nlp-total-cost-GMM-density-steering-constr} is implicitly relaxed. 
This causes the terminal state distribution to converge to the one in Figure \ref{subfig:prob2-term-dist-kappa50p0}.

\begin{figure*}
    \centering
    \begin{subfigure}{0.24\linewidth}
        \centering
        \resizebox{\linewidth}{!}{
\begin{tikzpicture}

\definecolor{darkgray176}{RGB}{176,176,176}

\begin{axis}[
axis equal image,
tick align=outside,
tick pos=left,
x grid style={darkgray176},
xmin=-10, xmax=20,
xtick style={color=black},
y grid style={darkgray176},
ymin=-10, ymax=20,
ytick style={color=black}
]
\addplot graphics [includegraphics cmd=\pgfimage,xmin=-10, xmax=20, ymin=-10, ymax=20] {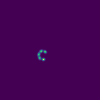};
\end{axis}

\end{tikzpicture}
        }
        \subcaption{$\kappa = 10.0$}
        \label{subfig:prob3-term-dist-kappa10p0}
    \end{subfigure}
    \hfill
    \begin{subfigure}{0.24\linewidth}
        \centering
        \resizebox{\linewidth}{!}{
\begin{tikzpicture}

\definecolor{darkgray176}{RGB}{176,176,176}

\begin{axis}[
axis equal image,
tick align=outside,
tick pos=left,
x grid style={darkgray176},
xmin=-10, xmax=20,
xtick style={color=black},
y grid style={darkgray176},
ymin=-10, ymax=20,
ytick style={color=black}
]
\addplot graphics [includegraphics cmd=\pgfimage,xmin=-10, xmax=20, ymin=-10, ymax=20] {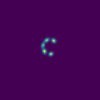};
\end{axis}

\end{tikzpicture}
        }
        \subcaption{$\kappa = 25.0$}
        \label{subfig:prob3-term-dist-kappa25p0}
    \end{subfigure}
    \hfill
    \begin{subfigure}{0.24\linewidth}
        \centering
        \resizebox{\linewidth}{!}{
\begin{tikzpicture}

\definecolor{darkgray176}{RGB}{176,176,176}

\begin{axis}[
axis equal image,
tick align=outside,
tick pos=left,
x grid style={darkgray176},
xmin=-10, xmax=20,
xtick style={color=black},
y grid style={darkgray176},
ymin=-10, ymax=20,
ytick style={color=black}
]
\addplot graphics [includegraphics cmd=\pgfimage,xmin=-10, xmax=20, ymin=-10, ymax=20] {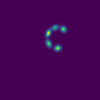};
\end{axis}

\end{tikzpicture}
        }
        \subcaption{$\kappa = 50.0$}
        \label{subfig:prob3-term-dist-kappa50p0}
    \end{subfigure}
    \hfill
    \begin{subfigure}{0.24\linewidth}
        \centering
        \resizebox{\linewidth}{!}{
\begin{tikzpicture}

\definecolor{darkgray176}{RGB}{176,176,176}

\begin{axis}[
axis equal image,
tick align=outside,
tick pos=left,
x grid style={darkgray176},
xmin=-10, xmax=20,
xtick style={color=black},
y grid style={darkgray176},
ymin=-10, ymax=20,
ytick style={color=black}
]
\addplot graphics [includegraphics cmd=\pgfimage,xmin=-10, xmax=20, ymin=-10, ymax=20] {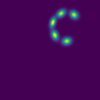};
\end{axis}

\end{tikzpicture}
        }
        \subcaption{$\kappa = 100.0$}
        \label{subfig:prob3-term-dist-kappa100p0}
    \end{subfigure}
    \hfill
    \caption{Terminal State PDF vs $\kappa$ for Problem \ref{prob:total-cost-constrained-GMM-density-steering}}
    \label{fig:prob3-terminal-dist-vs-kappa}
\end{figure*}

\subsection{Control of Drone Swarm - Problem 4}\label{s:num-exp-prob4}
In this section, we apply the GMM density steering techniques we developed to the problem of drone swarm trajectory optimization. 
We assume that practical constraints, such as maximum speed and maximum thrust per propeller, are enforced at all times during operation. 
Thus, the drone swarm trajectory optimization problem can be formulated as an instance of Problem \ref{prob:step-cost-constrained-GMM-density-steering}.

To obtain smooth trajectories for the drones to follow, we consider two dimensional double integrator dynamics:
$A_k = \Big[ \begin{smallmatrix} 
        I_2 & \Delta t \\
        \bm{0} & I_2 
    \end{smallmatrix} \Big],~B_k = \Big[ \begin{smallmatrix} 
        (\Delta t)^2 / 2 I_2 \\
        I_2 \end{smallmatrix} \Big] $
$x_k = [p_k^{x}, p_k^{y}, v_k^{x}, v_k^{y}]\transpose \in \R{4}$, $u_k = [a_k^x, a_k^y]\transpose \in \R{2}$ for all $k \in \curly{0, \dots, N-1}$ with $N = 8$.
In this example, $p_k := [p_k^x, p_k^y]\transpose \in \R{2}$, $v_k := [v_k^x, v_k^y]\transpose \in \R{2}$ and $a_k := [a_k^x, a_k^y]\transpose \in \R{2}$ represent the position, velocity, and acceleration, respectively. 
Additionally, two separate constraints are enforced for the state $x_k$ and control input $u_k$. 
Specifically, the acceleration and velocity satisfy the following constraints: $\E{a_k\transpose a_k} \leq a_{\mathrm{max}}^2$ and $\E{v_k\transpose v_k} \leq v_\mathrm{max}^2$, respectively (whereas the position $p_k$ is unbounded). To enforce these upper bound constraints, we set  $R_k = a_{\mathrm{max}}^{-2} I_2$, $Q_k = \bdiag{\bm{0}, v_\mathrm{max}^{-1} I_2 }$, $x_k' = \bm{0}$ and $\kappa_k = 1$ for all $k \in \curly{0, \dots, N-1}$ where $a_\mathrm{max} = 0.2 \mathrm{m/s^2}$ and $v_\mathrm{max} = 1.0 \mathrm{m/s}$.
Both maximum speed and acceleration constraints are enforced separately. Finally, $\Delta t$ is varied as a parameter, with $\Delta t \in \curly{1.0, 1.25, 1.5, 1.75, 2.0, 2.5, 3.0}$ to control the total time horizon in the experiments. The initial and desired PDFs of the drone positions are shown in Fig.~\ref{fig:init-des-pos-pdf-drones}. 


\begin{figure}
    \centering
    \begin{subfigure}{0.49\linewidth}
        \centering
        \resizebox{\linewidth}{!}{
            \begin{tikzpicture}

\definecolor{darkgray176}{RGB}{176,176,176}

\begin{axis}[
axis equal image,
tick align=outside,
tick pos=left,
x grid style={darkgray176},
xmin=-10, xmax=20,
xtick style={color=black},
y grid style={darkgray176},
ymin=-10, ymax=20,
ytick style={color=black}, 
xlabel={$p^x$}, 
ylabel={$p^y$}, 
ticklabel style={font=\Large},
label style={font=\Large},
]
\addplot graphics [includegraphics cmd=\pgfimage,xmin=-10, xmax=20, ymin=-10, ymax=20] {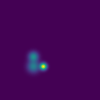};
\end{axis}

\end{tikzpicture}
        }
        \caption{Initial PDF}
        \label{subfig:init-pos-pdf-drones}
    \end{subfigure}
    \hfill
    \begin{subfigure}{0.49\linewidth}
        \centering
        \resizebox{\linewidth}{!}{
            \begin{tikzpicture}

\definecolor{darkgray176}{RGB}{176,176,176}

\begin{axis}[
axis equal image,
tick align=outside,
tick pos=left,
x grid style={darkgray176},
xmin=-10, xmax=20,
xtick style={color=black},
y grid style={darkgray176},
ymin=-10, ymax=20,
ytick style={color=black}, 
xlabel={$p^x$},
ylabel={$p^y$}, 
ticklabel style={font=\Large}, 
label style={font=\Large},
]
\addplot graphics [includegraphics cmd=\pgfimage,xmin=-10, xmax=20, ymin=-10, ymax=20] {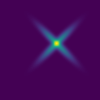};
\end{axis}

\end{tikzpicture}
        }
        \caption{Desired PDF}
        \label{subfig:des-pos-pdf-drones}
    \end{subfigure}
    \caption{Initial and Desired PDFs of Drone Positions}
    \label{fig:init-des-pos-pdf-drones}
\end{figure}

For drone simulations, we used \texttt{gym-pybullet-drones} \cite{p:panerati2021gym-pybullet-drones}. 
The optimal policy obtained after solving the problem in \eqref{eq:stepBilinear} returns a sequence of acceleration commands given an initial state. 
The drones then follow this trajectory using cascaded PID controllers. 
The maximum speed, $v_\mathrm{max}$, and acceleration, $a_\mathrm{max}$, are selected such that the underlying PID controllers can easily track the output trajectory. 
In all of these experiments, 20 drones are used. 
The drones' initial positions 
are sampled from the GMM distribution shown in Figure \ref{subfig:init-pos-pdf-drones}. 

In Figure \ref{fig:traj-vel-drones}, the trajectories of the drones in the $x$-$y$ plane 
are depicted for $\Delta t \in \{1.0, 1.5, 3.0\}$. The problem horizon parameter $N=8$ (fixed) whereas the varying $\Delta t$ determines the time frame of the trajectories. The speed and acceleration limits prevent the drones from reaching the desired distribution within a given time, as shown in Fig.~\ref{subfig:traj-dt1} for $\Delta t = 1.0\mathrm{s}$. 
When $\Delta t = 1.5\mathrm{s}$, the terminal positions of the drones are closer to the desired distribution, as shown in Fig.~\ref{subfig:traj-dt1p5}. 
With $\Delta t = 3.0\mathrm{s}$, there is ample time to reach the desired positions (terminal GMM matches the desired GMM). Consequently, there exist multiple sets of optimal policy parameters that solve Problem \ref{prob:step-cost-constrained-GMM-density-steering} when $\Delta t = 3.0\mathrm{s}$. 
The optimal policy parameters returned by the BCD procedure for $\Delta t = 3.0\mathrm{s}$ is not as strict as the other cases in maximizing the speed to reach the desired positions of the drones. That is why the trajectories are more dispersed and irregular in Figures \ref{subfig:traj-dt3}, \ref{subfig:vx-dt3} and \ref{subfig:vy-dt3}.


The terminal positions of the drones form an ``X'' shape, as shown in Figure \ref{subfig:des-pos-pdf-drones}. 
In this experiment, although the terminal distribution matches the desired one exactly, the terminal positions of the drones do not appear to precisely align with the desired distribution because only 20 sampled drone positions were used. 
Additionally, Table \ref{tab:wass-vs-dt-problem4} shows the optimal value of Problem \ref{prob:step-cost-constrained-GMM-density-steering} versus $\Delta t$.
Due to the maximum speed constraint, it is not possible to reach the desired distribution within a finite time if $\Delta t = 1.0 \mathrm{s}$, resulting in a large incurred cost. 
As $\Delta t$ increases, the terminal GMM-Wasserstein cost function decreases. 
With $\Delta t = 2.0 \mathrm{s}$ providing sufficient time for the drones to reach the desired distribution, any $\Delta t$ greater than $2.0\mathrm{s}$ will yield the same optimal cost of zero.
\begin{figure}
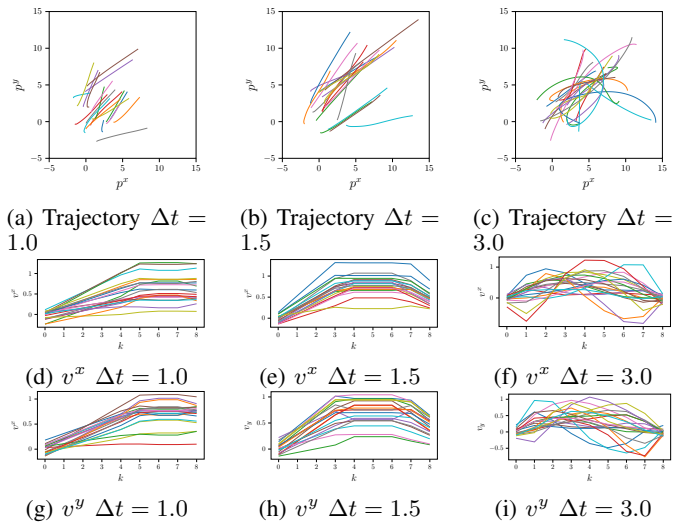

    \centering
    \begin{subfigure}{0.3\linewidth}
        \resizebox{\linewidth}{!}{
            \input{example4/drone_traj_dt1}
        }
        \caption{Trajectory $\Delta t = 1.0$}
        \label{subfig:traj-dt1}
    \end{subfigure}
    \hfill
    \begin{subfigure}{0.3\linewidth}
        \resizebox{\linewidth}{!}{
            \input{example4/drone_traj_dt1p5}
        }
        \caption{Trajectory $\Delta t = 1.5$}
        \label{subfig:traj-dt1p5}
    \end{subfigure}
    \hfill
    \begin{subfigure}{0.3\linewidth}
        \resizebox{\linewidth}{!}{
            \input{example4/drone_traj_dt3}
        }
        \caption{Trajectory $\Delta t = 3.0$}
        \label{subfig:traj-dt3}
    \end{subfigure}
     \hfill
    \begin{subfigure}{0.3\linewidth}
        \resizebox{\linewidth}{!}{
\begin{tikzpicture}

\definecolor{crimson2143940}{RGB}{214,39,40}
\definecolor{darkgray176}{RGB}{176,176,176}
\definecolor{darkorange25512714}{RGB}{255,127,14}
\definecolor{darkturquoise23190207}{RGB}{23,190,207}
\definecolor{forestgreen4416044}{RGB}{44,160,44}
\definecolor{goldenrod18818934}{RGB}{188,189,34}
\definecolor{gray127}{RGB}{127,127,127}
\definecolor{mediumpurple148103189}{RGB}{148,103,189}
\definecolor{orchid227119194}{RGB}{227,119,194}
\definecolor{sienna1408675}{RGB}{140,86,75}
\definecolor{steelblue31119180}{RGB}{31,119,180}

\begin{axis}[
tick align=outside,
tick pos=left,
x grid style={darkgray176},
xmin=-0.4, xmax=8.4,
xtick style={color=black},
y grid style={darkgray176},
ymin=-0.312520969904245, ymax=1.3415741030484,
ytick style={color=black}, 
xlabel={$k$}, 
ylabel={$v^x$},
height=5cm,
width=10cm,
label style={font=\Large},
]
\addplot [semithick, steelblue31119180]
table {%
0 0.0550101731341287
1 0.175250874357795
2 0.2952599238746
3 0.414814359103032
4 0.533065234743006
5 0.620189524275435
6 0.60473270714852
7 0.604727282489744
8 0.596923515129294
};
\addplot [semithick, darkorange25512714]
table {%
0 0.00722206992816003
1 0.177562724264152
2 0.347917554979864
3 0.518294784368083
4 0.688694792029273
5 0.848866111013569
6 0.851417947775886
7 0.851410298277758
8 0.855778741262448
};
\addplot [semithick, forestgreen4416044]
table {%
0 -0.23733483022458
1 -0.0701997617134199
2 0.0981315710422686
3 0.268789214283881
4 0.44593150434415
5 0.730107951380715
6 0.816753276622867
7 0.816710469837209
8 0.692403913940722
};
\addplot [semithick, crimson2143940]
table {%
0 -0.119640221828024
1 -0.0155626729660001
2 0.0891231999793277
3 0.194992246751909
4 0.304154977922475
5 0.466485066218115
6 0.510829630171447
7 0.510830111635167
8 0.490711853177514
};
\addplot [semithick, mediumpurple148103189]
table {%
0 -0.0947225654175252
1 -0.00446848932528711
2 0.0862709008528281
3 0.177951156601498
4 0.272249134551315
5 0.408332118750185
6 0.443630452315672
7 0.443633728080061
8 0.392512999438248
};
\addplot [semithick, sienna1408675]
table {%
0 0.0797443368754139
1 0.147926260519329
2 0.215742647743403
3 0.282843253504733
4 0.347921208042168
5 0.372096954552908
6 0.346709061464015
7 0.346732474093002
8 0.40644146585637
};
\addplot [semithick, orchid227119194]
table {%
0 0.0389406547352383
1 0.121989483669839
2 0.204873987127887
3 0.28743551341896
4 0.369071717017548
5 0.428755471280698
6 0.417841094571876
7 0.417852978877302
8 0.443663506629783
};
\addplot [semithick, gray127]
table {%
0 -0.111434929514298
1 0.0128146563593265
2 0.137640261159857
3 0.263582389760772
4 0.392626802063044
5 0.570109217931649
6 0.612123717543801
7 0.612110948948992
8 0.534331891130711
};
\addplot [semithick, goldenrod18818934]
table {%
0 -0.0747570956751877
1 -0.0575522408736454
2 -0.039980424716467
3 -0.021694677420557
4 -0.00141071139411213
5 0.0539964396794004
6 0.0803089585240024
7 0.080307248375939
8 0.0745390211951815
};
\addplot [semithick, darkturquoise23190207]
table {%
0 0.114469355905336
1 0.328587697469674
2 0.542213423754987
3 0.754873785159865
4 0.964776317471412
5 1.11134067688534
6 1.07830273486543
7 1.07831909372285
8 1.13216467001502
};
\addplot [semithick, steelblue31119180]
table {%
0 0.0159258923362681
1 0.169948451743469
2 0.323938557858994
3 0.477860246238761
4 0.63155432988686
5 0.771499748396444
6 0.770545219670121
7 0.77052711884748
8 0.769866439406663
};
\addplot [semithick, darkorange25512714]
table {%
0 -0.234813047076663
1 -0.136474338011319
2 -0.0369700106374521
3 0.0648002295253394
4 0.172902254737831
5 0.389567168796765
6 0.473303949436765
7 0.473283650310817
8 0.334376904817373
};
\addplot [semithick, forestgreen4416044]
table {%
0 -0.00377594188813041
1 0.24981815592375
2 0.503503667855649
3 0.757359685560839
4 1.01163353258956
5 1.25763615497564
6 1.26638796336873
7 1.26636988442423
8 1.24040975701653
};
\addplot [semithick, crimson2143940]
table {%
0 -0.0273655960316563
1 0.0644785769150251
2 0.156481706005521
3 0.2487916591026
4 0.341940431549653
5 0.444545863242224
6 0.456651142784963
7 0.456640259770788
8 0.433061942951349
};
\addplot [semithick, mediumpurple148103189]
table {%
0 0.0653773071747862
1 0.096789974547176
2 0.127895938866059
3 0.158402053328642
4 0.187219139081136
5 0.183472568926933
6 0.162094427986446
7 0.162138305207057
8 0.261111256248255
};
\addplot [semithick, sienna1408675]
table {%
0 0.0466337988148973
1 0.290959490774968
2 0.53512923562926
3 0.778989229311187
4 1.02192275629128
5 1.23280791655325
6 1.22399015661133
7 1.22399870781579
8 1.24041065281488
};
\addplot [semithick, orchid227119194]
table {%
0 -0.00822658500687279
1 0.137320550715747
2 0.282949657684832
3 0.428733575057959
4 0.574917717082456
5 0.719258587037231
6 0.726393201621559
7 0.726383703962883
8 0.721300782508424
};
\addplot [semithick, gray127]
table {%
0 0.0518411873685862
1 0.202673963081042
2 0.353299444857656
3 0.50351636334321
4 0.652551636852749
5 0.770762808638493
6 0.757440552688045
7 0.757466911809112
8 0.807250712095347
};
\addplot [semithick, goldenrod18818934]
table {%
0 0.0694019607209628
1 0.240921650847164
2 0.412155884808414
3 0.582826995287577
4 0.751878313921546
5 0.880872465257631
6 0.862035791825429
7 0.86204975817884
8 0.880836647296192
};
\addplot [semithick, darkturquoise23190207]
table {%
0 0.0194676762040019
1 0.0883209976902768
2 0.157100123794137
3 0.225732299564111
4 0.293936211188774
5 0.35007924281171
6 0.345442124540244
7 0.345448514464918
8 0.363028671530601
};
\end{axis}

\end{tikzpicture}
        }
        \caption{$v^x$ $\Delta t = 1.0$}
        \label{subfig:vx-dt1}
    \end{subfigure}
    \hfill
    \begin{subfigure}{0.3\linewidth}
        \resizebox{\linewidth}{!}{
\begin{tikzpicture}

\definecolor{crimson2143940}{RGB}{214,39,40}
\definecolor{darkgray176}{RGB}{176,176,176}
\definecolor{darkorange25512714}{RGB}{255,127,14}
\definecolor{darkturquoise23190207}{RGB}{23,190,207}
\definecolor{forestgreen4416044}{RGB}{44,160,44}
\definecolor{goldenrod18818934}{RGB}{188,189,34}
\definecolor{gray127}{RGB}{127,127,127}
\definecolor{mediumpurple148103189}{RGB}{148,103,189}
\definecolor{orchid227119194}{RGB}{227,119,194}
\definecolor{sienna1408675}{RGB}{140,86,75}
\definecolor{steelblue31119180}{RGB}{31,119,180}

\begin{axis}[
tick align=outside,
tick pos=left,
x grid style={darkgray176},
xmin=-0.4, xmax=8.4,
xtick style={color=black},
y grid style={darkgray176},
ymin=-0.213980401334516, ymax=1.39193535145516,
ytick style={color=black}, 
xlabel={$k$}, 
ylabel={$v^x$},
height=5cm,
width=10cm,
label style={font=\Large},
]
\addplot [semithick, steelblue31119180]
table {%
0 0.137035507543457
1 0.531004676194744
2 0.924972225068417
3 1.31893918087381
4 1.31342709002138
5 1.31346812413321
6 1.31350004894195
7 1.28655549619054
8 0.892561967308788
};
\addplot [semithick, darkorange25512714]
table {%
0 0.0317988905724039
1 0.235841788006789
2 0.439885343454882
3 0.643928958239043
4 0.680247382542162
5 0.680267213897707
6 0.680282211681041
7 0.690567036341662
8 0.486515940446927
};
\addplot [semithick, forestgreen4416044]
table {%
0 0.0341673468395102
1 0.232253608841978
2 0.430340604794292
3 0.62842768980498
4 0.660408245077279
5 0.660431173047898
6 0.660448538287831
7 0.683846396515356
8 0.485750646924545
};
\addplot [semithick, crimson2143940]
table {%
0 -0.137689810091937
1 0.00699234905782861
2 0.151675836456979
3 0.29635976862239
4 0.482159985501355
5 0.482135855322917
6 0.482119131239069
7 0.386867376369744
8 0.242197454785945
};
\addplot [semithick, mediumpurple148103189]
table {%
0 -0.0798255764557595
1 0.189562715946017
2 0.45895065814426
3 0.728338955971653
4 0.897773312786665
5 0.897819889760801
6 0.897871158930062
7 0.717038451208295
8 0.447607785788397
};
\addplot [semithick, sienna1408675]
table {%
0 0.00251373166002001
1 0.251864522624531
2 0.501215388961645
3 0.750566093766588
4 0.831127376991174
5 0.831119287681202
6 0.83111304945878
7 0.727871386682056
8 0.478524145332501
};
\addplot [semithick, orchid227119194]
table {%
0 -0.0365095477056627
1 0.148395283487719
2 0.333300982658854
3 0.518206871744725
4 0.616356505803235
5 0.616356868888152
6 0.616357724917797
7 0.575900977736548
8 0.390996892092357
};
\addplot [semithick, gray127]
table {%
0 -0.0260467264833972
1 0.294863778210997
2 0.615773380870613
3 0.936683079596739
4 1.0695609604495
5 1.06959141375349
6 1.06962592888184
7 0.890722091220028
8 0.569782917610736
};
\addplot [semithick, goldenrod18818934]
table {%
0 -0.0957302914550818
1 0.130482304162701
2 0.356695221987816
3 0.582908143388314
4 0.753851511245468
5 0.753813817171186
6 0.753785995800971
7 0.570042206431213
8 0.343846672828527
};
\addplot [semithick, darkturquoise23190207]
table {%
0 -0.100559033132389
1 0.145198252497903
2 0.390955507279641
3 0.636713050290763
4 0.818999256434309
5 0.818964347017764
6 0.818937172719079
7 0.67590511616322
8 0.430168572912888
};
\addplot [semithick, steelblue31119180]
table {%
0 0.105987873957632
1 0.409540760397132
2 0.713092943472034
3 1.01664493369861
4 1.01201056759946
5 1.01204155830861
6 1.01206524702251
7 0.994078444007636
8 0.690507496811685
};
\addplot [semithick, darkorange25512714]
table {%
0 0.00670528498894615
1 0.221507653159505
2 0.43631053285509
3 0.651113427818781
4 0.716041761463029
5 0.716047933685831
6 0.716052692178721
7 0.677188290769696
8 0.462383640615328
};
\addplot [semithick, forestgreen4416044]
table {%
0 0.0966867001517091
1 0.380227411049685
2 0.663766540530491
3 0.947304809685937
4 0.945194198451846
5 0.945221012796102
6 0.945243428877346
7 0.875752294843125
8 0.592192490826257
};
\addplot [semithick, crimson2143940]
table {%
0 -0.128331717415406
1 0.0905276486620986
2 0.309387414771873
3 0.528247237565236
4 0.729290816587922
5 0.729243179532864
6 0.729208103820172
7 0.518192491051407
8 0.299354719018326
};
\addplot [semithick, mediumpurple148103189]
table {%
0 -0.114392866143491
1 0.142277730060786
2 0.398948257253725
3 0.65561863907117
4 0.855281838935995
5 0.855225772202551
6 0.855183968437407
7 0.596375315146699
8 0.339729501996866
};
\addplot [semithick, sienna1408675]
table {%
0 -0.0580638850510758
1 0.216270175944889
2 0.490603903608308
3 0.764937746683749
4 0.914323346651692
5 0.91429744917168
6 0.914276497594056
7 0.792202376005186
8 0.517884481455526
};
\addplot [semithick, orchid227119194]
table {%
0 -0.140984230753167
1 0.0732853531417387
2 0.287555309554539
3 0.501825592208913
4 0.714063299924214
5 0.714037323238422
6 0.714016703015709
7 0.616735191935876
8 0.402481456814044
};
\addplot [semithick, gray127]
table {%
0 0.0118756110692029
1 0.24426925426995
2 0.476663190189206
3 0.709057047978214
4 0.774655159662192
5 0.774656900669927
6 0.77465815241335
7 0.711950089403695
8 0.47955587446229
};
\addplot [semithick, goldenrod18818934]
table {%
0 0.0554944854765829
1 0.123601523318175
2 0.191709249795957
3 0.259817301182625
4 0.227202222055686
5 0.227209282610485
6 0.227211872302754
7 0.291486139876993
8 0.223378245655003
};
\addplot [semithick, darkturquoise23190207]
table {%
0 0.0276187109935553
1 0.29355276185904
2 0.55948669089207
3 0.825420360997203
4 0.88642201329129
5 0.886417585950868
6 0.886413837460953
7 0.785444128493127
8 0.519511752321564
};
\end{axis}

\end{tikzpicture}
        }
        \caption{$v^x$ $\Delta t = 1.5$}
        \label{subfig:vx-dt1p5}
    \end{subfigure}
    \hfill
    \begin{subfigure}{0.3\linewidth}
        \resizebox{\linewidth}{!}{
\begin{tikzpicture}

\definecolor{crimson2143940}{RGB}{214,39,40}
\definecolor{darkgray176}{RGB}{176,176,176}
\definecolor{darkorange25512714}{RGB}{255,127,14}
\definecolor{darkturquoise23190207}{RGB}{23,190,207}
\definecolor{forestgreen4416044}{RGB}{44,160,44}
\definecolor{goldenrod18818934}{RGB}{188,189,34}
\definecolor{gray127}{RGB}{127,127,127}
\definecolor{mediumpurple148103189}{RGB}{148,103,189}
\definecolor{orchid227119194}{RGB}{227,119,194}
\definecolor{sienna1408675}{RGB}{140,86,75}
\definecolor{steelblue31119180}{RGB}{31,119,180}

\begin{axis}[
tick align=outside,
tick pos=left,
x grid style={darkgray176},
xmin=-0.4, xmax=8.4,
xtick style={color=black},
y grid style={darkgray176},
ymin=-0.926112322165156, ymax=1.32726864712298,
ytick style={color=black}, 
xlabel={$k$}, 
ylabel={$v^x$},
height=5cm,
width=10cm,
label style={font=\Large},
]
\addplot [semithick, steelblue31119180]
table {%
0 0.0609734039683963
1 0.735966489588564
2 0.944802565158457
3 0.814599762759024
4 0.471602772136556
5 0.260016167621306
6 0.253264498524335
7 0.231351319033312
8 -0.0140944837001789
};
\addplot [semithick, darkorange25512714]
table {%
0 -0.0307343033404381
1 -0.000168202762377754
2 0.0927805860458075
3 0.249198361876047
4 0.364894713534431
5 0.347080917643991
6 0.190765697834873
7 -0.00405081973682173
8 -0.0381795569795692
};
\addplot [semithick, forestgreen4416044]
table {%
0 0.0320202098294143
1 0.453135661732291
2 0.590027121743147
3 0.556064228734923
4 0.38637225674284
5 0.185215170485567
6 0.0448684052225095
7 0.00343219282636863
8 -0.00184661566827342
};
\addplot [semithick, crimson2143940]
table {%
0 0.0885279950037087
1 0.285497196628604
2 0.184015061345748
3 0.095938132887378
4 0.11839725442514
5 0.157855848771052
6 0.157156226284082
7 0.118463489354355
8 0.0931030950930205
};
\addplot [semithick, mediumpurple148103189]
table {%
0 0.116681495879461
1 0.430785215391581
2 0.330150921106534
3 0.170861391916172
4 0.0972306170376403
5 0.118403875371012
6 0.199271214318512
7 0.237326385580606
8 0.123338504371987
};
\addplot [semithick, sienna1408675]
table {%
0 -0.00603160637274791
1 0.269543963765059
2 0.409312717928101
3 0.470946539028043
4 0.421182667603093
5 0.2446066465195
6 0.00991611128369388
7 -0.139100494181515
8 -0.0573197450028734
};
\addplot [semithick, orchid227119194]
table {%
0 -0.0463067454503808
1 0.19048435688835
2 0.386122836626556
3 0.542771016810392
4 0.682326215328106
5 0.821307174357618
6 0.831429190733708
7 0.480265683664907
8 -0.0175809993490907
};
\addplot [semithick, gray127]
table {%
0 -0.00567419110940667
1 0.288358427047093
2 0.439315806013596
3 0.511088189868903
4 0.497377134259695
5 0.449547177626827
6 0.396262660994575
7 0.25676902298286
8 0.0278536410745943
};
\addplot [semithick, goldenrod18818934]
table {%
0 -0.0820682415931088
1 0.31498537571999
2 0.627016378413205
3 0.822307293198368
4 0.839678830929585
5 0.75421429904776
6 0.60119582651455
7 0.286983487811036
8 -0.0778341364502027
};
\addplot [semithick, darkturquoise23190207]
table {%
0 0.0145298772663313
1 0.0407183436381346
2 -0.024942071271662
3 0.0868489502621758
4 0.323769895335946
5 0.649452933797477
6 1.07630996268565
7 1.07143402571225
8 0.134330931640694
};
\addplot [semithick, steelblue31119180]
table {%
0 -0.00773779642829369
1 0.106188298595951
2 0.172325216824989
3 0.280903070044581
4 0.430191862811351
5 0.563738083088959
6 0.609843325338527
7 0.415105238689339
8 0.0601863479576911
};
\addplot [semithick, darkorange25512714]
table {%
0 -0.0635506967390325
1 0.269956650020946
2 0.790493994540191
3 0.628091522069636
4 0.0589192649243027
5 -0.414022870721653
6 -0.672400532346799
7 -0.600505744771609
8 0.00180531062527256
};
\addplot [semithick, forestgreen4416044]
table {%
0 -0.0622770166992377
1 0.304411209812735
2 0.572114897568413
3 0.733651665657432
4 0.741662777187847
5 0.631525386292865
6 0.425187109873056
7 0.114076701058665
8 -0.105268254050929
};
\addplot [semithick, crimson2143940]
table {%
0 -0.288810088979591
1 -0.753559164238072
2 -0.0971554802230566
3 0.777369834751121
4 1.22484223942806
5 1.21459350857887
6 0.945743860659519
7 0.376478689641685
8 -0.210611397486423
};
\addplot [semithick, mediumpurple148103189]
table {%
0 -0.073449946260935
1 0.191467294284432
2 0.594473860329002
3 0.638061752198568
4 0.317381155769984
5 -0.207265479576861
6 -0.766040199765741
7 -0.823685914470241
8 -0.123265099402625
};
\addplot [semithick, sienna1408675]
table {%
0 -0.0103228222639983
1 0.363737548430131
2 0.553077159087324
3 0.630010370876328
4 0.594660931704781
5 0.458706895878182
6 0.213553426782811
7 -0.0786707664915801
8 -0.134926508635479
};
\addplot [semithick, orchid227119194]
table {%
0 -0.0690777743541849
1 -0.10035068993358
2 0.0168176285875957
3 0.200682161399091
4 0.405329724822339
5 0.528275547359455
6 0.515770355937337
7 0.315489251084261
8 0.0755679386316278
};
\addplot [semithick, gray127]
table {%
0 -0.165141172903583
1 -0.134347974047233
2 0.228571569768084
3 0.633612785004326
4 0.871964143419595
5 0.883839411587221
6 0.698299144252153
7 0.298034758785144
8 -0.0770242699891883
};
\addplot [semithick, goldenrod18818934]
table {%
0 -0.14133502685869
1 -0.503781382531302
2 -0.0406308445137399
3 0.647670406917724
4 0.914571991560956
5 0.667848404126016
6 0.157421765999669
7 -0.406535644243778
8 -0.204450270053557
};
\addplot [semithick, darkturquoise23190207]
table {%
0 0.0317816287288581
1 0.239533905075113
2 0.265130638556416
3 0.272808275612548
4 0.267551307362636
5 0.167965740459642
6 0.0245786574959893
7 -0.0282572897188782
8 0.0429148568980854
};
\end{axis}

\end{tikzpicture}
        }
        \caption{$v^x$ $\Delta t = 3.0$}
        \label{subfig:vx-dt3}
    \end{subfigure}
    \hfill
    \begin{subfigure}{0.3\linewidth}
        \resizebox{\linewidth}{!}{
\begin{tikzpicture}

\definecolor{crimson2143940}{RGB}{214,39,40}
\definecolor{darkgray176}{RGB}{176,176,176}
\definecolor{darkorange25512714}{RGB}{255,127,14}
\definecolor{darkturquoise23190207}{RGB}{23,190,207}
\definecolor{forestgreen4416044}{RGB}{44,160,44}
\definecolor{goldenrod18818934}{RGB}{188,189,34}
\definecolor{gray127}{RGB}{127,127,127}
\definecolor{mediumpurple148103189}{RGB}{148,103,189}
\definecolor{orchid227119194}{RGB}{227,119,194}
\definecolor{sienna1408675}{RGB}{140,86,75}
\definecolor{steelblue31119180}{RGB}{31,119,180}

\begin{axis}[
tick align=outside,
tick pos=left,
x grid style={darkgray176},
xmin=-0.4, xmax=8.4,
xtick style={color=black},
y grid style={darkgray176},
ymin=-0.199046080904371, ymax=1.15054272925285,
ytick style={color=black},
xlabel={$k$}, 
ylabel={$v^x$},
height=5cm,
width=10cm,
label style={font=\Large},
]
\addplot [semithick, steelblue31119180]
table {%
0 -0.109212737573175
1 0.0344273724888265
2 0.178637656961127
3 0.323953832554834
4 0.472337852285802
5 0.667321215387174
6 0.709163094284384
7 0.709144762041365
8 0.661491546643917
};
\addplot [semithick, darkorange25512714]
table {%
0 0.0742914594869605
1 0.208628519346098
2 0.342644950888775
3 0.476032296159484
4 0.607623961048227
5 0.698294339540942
6 0.676731389488299
7 0.676746604711329
8 0.732174873665812
};
\addplot [semithick, forestgreen4416044]
table {%
0 0.0584899769273919
1 0.217334205070514
2 0.375941491706138
3 0.534081853941938
4 0.690874712736606
5 0.813340570836984
6 0.797972719086093
7 0.798009946156738
8 0.849239662271614
};
\addplot [semithick, crimson2143940]
table {%
0 0.0317942247488632
1 0.0497381792047417
2 0.0675345099251409
3 0.0850405044088493
4 0.101728686265535
5 0.102476535171649
6 0.0920481885380727
7 0.0920420904962653
8 0.099644662743645
};
\addplot [semithick, mediumpurple148103189]
table {%
0 -0.127711483347045
1 0.077404458559581
2 0.283198488265005
3 0.49030564925512
4 0.701046687140209
5 0.964806370863193
6 1.01477410637595
7 1.01475269925149
8 0.906453169507132
};
\addplot [semithick, sienna1408675]
table {%
0 -0.0150348770453744
1 0.203262611724669
2 0.421696392371089
3 0.640387974893909
4 0.859750998002832
5 1.07763934661671
6 1.08919778333661
7 1.08917515983832
8 1.04397669301259
};
\addplot [semithick, orchid227119194]
table {%
0 0.0250725415178546
1 0.176295758650211
2 0.32744143883993
3 0.478431122610853
4 0.628947699826579
5 0.76144500992705
6 0.757300138005858
7 0.757306053424451
8 0.767762554923672
};
\addplot [semithick, gray127]
table {%
0 -0.0555275424758569
1 0.108441976155825
2 0.272728594584194
3 0.437625872248469
4 0.604196978160348
5 0.790848899127563
6 0.814777525130682
7 0.814774171880139
8 0.75654309727267
};
\addplot [semithick, goldenrod18818934]
table {%
0 -0.0499380163194122
1 0.0154996330356143
2 0.0811983571725972
3 0.147402510709212
4 0.215009240663183
5 0.303933869472801
6 0.32320574190223
7 0.323220416385117
8 0.355885076578799
};
\addplot [semithick, darkturquoise23190207]
table {%
0 0.0880553448585889
1 0.2302246774957
2 0.372009545532476
3 0.513039875090818
4 0.651920586091554
5 0.742974765666527
6 0.716829423774593
7 0.71684716067548
8 0.731077506821194
};
\addplot [semithick, steelblue31119180]
table {%
0 0.180657610480516
1 0.313101986837293
2 0.444710916084032
3 0.574686951165856
4 0.700052753398428
5 0.733626530595458
6 0.675375112282393
7 0.675403924574414
8 0.778089545002025
};
\addplot [semithick, darkorange25512714]
table {%
0 -0.137701134988134
1 0.0617708830464752
2 0.261968067636196
3 0.463569615347627
4 0.669061968282117
5 0.932582490920032
6 0.985832852115542
7 0.985824462616838
8 0.870265606630727
};
\addplot [semithick, forestgreen4416044]
table {%
0 0.0534281656890528
1 0.108902314455447
2 0.16413390512033
3 0.218891040718044
4 0.272305274602521
5 0.297926748751359
6 0.281271724548593
7 0.28129630079363
8 0.354138272698203
};
\addplot [semithick, crimson2143940]
table {%
0 -0.067018176091417
1 0.0846094986226037
2 0.236605343594918
3 0.389313586886625
4 0.543984077593489
5 0.724727465224724
6 0.752249231078691
7 0.752231654968653
8 0.710173730518301
};
\addplot [semithick, mediumpurple148103189]
table {%
0 0.0349602110522635
1 0.187115874528997
2 0.339146208717875
3 0.490927975035877
4 0.641976352428534
5 0.770234106292714
6 0.762664880400764
7 0.762664437316891
8 0.756811320265815
};
\addplot [semithick, sienna1408675]
table {%
0 0.110249249896371
1 0.274390765696632
2 0.438046895726686
3 0.600750900656869
4 0.760746163350625
5 0.861411188824149
6 0.828230967844085
7 0.828236890553144
8 0.83586449660642
};
\addplot [semithick, orchid227119194]
table {%
0 0.0842923497831349
1 0.235812061091992
2 0.386967986954145
3 0.537409710732159
4 0.685812893164175
5 0.787821658608832
6 0.763299336013971
7 0.76331368422649
8 0.810302602750571
};
\addplot [semithick, gray127]
table {%
0 -0.132664239399481
1 0.0387545342238532
2 0.210864359778809
3 0.384315980757816
4 0.561490168972494
5 0.795394052180487
6 0.846142816757683
7 0.846117126961999
8 0.781331683583949
};
\addplot [semithick, goldenrod18818934]
table {%
0 -0.0795350293463504
1 0.0374251639015757
2 0.154804482485628
3 0.272995638883631
4 0.393436617773258
5 0.547325916835408
6 0.578151437614257
7 0.578133232898787
8 0.527928463287112
};
\addplot [semithick, darkturquoise23190207]
table {%
0 -0.0978035368775555
1 0.0228112768258667
2 0.143934407114976
3 0.266043666254328
4 0.390890230450534
5 0.557787591476249
6 0.595039896978111
7 0.595026145325827
8 0.558703258256614
};
\end{axis}

\end{tikzpicture}
        }
        \caption{$v^y$ $\Delta t = 1.0$}
        \label{subfig:vy-dt1}
    \end{subfigure}
    \hfill
    \begin{subfigure}{0.3\linewidth}
        \resizebox{\linewidth}{!}{
\begin{tikzpicture}

\definecolor{crimson2143940}{RGB}{214,39,40}
\definecolor{darkgray176}{RGB}{176,176,176}
\definecolor{darkorange25512714}{RGB}{255,127,14}
\definecolor{darkturquoise23190207}{RGB}{23,190,207}
\definecolor{forestgreen4416044}{RGB}{44,160,44}
\definecolor{goldenrod18818934}{RGB}{188,189,34}
\definecolor{gray127}{RGB}{127,127,127}
\definecolor{mediumpurple148103189}{RGB}{148,103,189}
\definecolor{orchid227119194}{RGB}{227,119,194}
\definecolor{sienna1408675}{RGB}{140,86,75}
\definecolor{steelblue31119180}{RGB}{31,119,180}

\begin{axis}[
tick align=outside,
tick pos=left,
x grid style={darkgray176},
xmin=-0.4, xmax=8.4,
xtick style={color=black},
y grid style={darkgray176},
ymin=-0.194552417300779, ymax=1.09976432176741,
ytick style={color=black},
width=10cm, 
height=5cm,
xlabel={$k$},
ylabel={$v_y$},
label style={font=\Large},
]
\addplot [semithick, steelblue31119180]
table {%
0 0.14099210332105
1 0.432117293750064
2 0.723240838803726
3 1.01436376344904
4 0.970574678949176
5 0.970617528520823
6 0.970653316749181
7 0.937603111375883
8 0.646450081720212
};
\addplot [semithick, darkorange25512714]
table {%
0 -0.0928567919002946
1 0.157553931132646
2 0.407964673148157
3 0.658375291479624
4 0.834470306754231
5 0.83442487697066
6 0.834391015815179
7 0.612663497277074
8 0.362272919234291
};
\addplot [semithick, forestgreen4416044]
table {%
0 0.0419937243163502
1 0.330741668013125
2 0.619489209021759
3 0.908236365784605
4 0.962452082628217
5 0.962444886024193
6 0.96243878235288
7 0.839929380078744
8 0.551184006908134
};
\addplot [semithick, crimson2143940]
table {%
0 0.0598998335816807
1 0.285306503049381
2 0.510713574654985
3 0.736120582358427
4 0.751473624469806
5 0.751496530010843
6 0.751513520321556
7 0.759536765669721
8 0.534120213448522
};
\addplot [semithick, mediumpurple148103189]
table {%
0 0.220788037543331
1 0.390008937531207
2 0.559229514952227
3 0.728449543541638
4 0.564522694532554
5 0.564538707603671
6 0.56453753319724
7 0.735204955589897
8 0.56598778130143
};
\addplot [semithick, sienna1408675]
table {%
0 -0.05428137974673
1 0.223100470343536
2 0.500482017471144
3 0.777863283366881
4 0.924397430966404
5 0.92435708542374
6 0.92432662552813
7 0.703534389963206
8 0.426170024214514
};
\addplot [semithick, orchid227119194]
table {%
0 0.0770113223979298
1 0.389296245408165
2 0.701580485385491
3 1.01386420300967
4 1.04093174271886
5 1.04092950760526
6 1.0409267599641
7 0.920889785162316
8 0.608604877439435
};
\addplot [semithick, gray127]
table {%
0 0.172179320869587
1 0.397307617507885
2 0.62243501168219
3 0.847561787323183
4 0.750735221514974
5 0.750744450251651
6 0.750742033670928
7 0.833939258759401
8 0.608814832457617
};
\addplot [semithick, goldenrod18818934]
table {%
0 0.102672594393305
1 0.384218526897398
2 0.665764174568293
3 0.947309447579064
4 0.938579135929543
5 0.938598085423265
6 0.938611406922787
7 0.902682381992904
8 0.621127495966136
};
\addplot [semithick, darkturquoise23190207]
table {%
0 -0.013600670094821
1 0.119801586957737
2 0.253203823179873
3 0.3866062943599
4 0.444551929571419
5 0.444513312134849
6 0.44448184457431
7 0.329644748562409
8 0.196266764017945
};
\addplot [semithick, steelblue31119180]
table {%
0 0.069613480833465
1 0.278962729951145
2 0.48831122862759
3 0.697659434088852
4 0.697938501778828
5 0.697972580798384
6 0.698001158934287
7 0.682998024315088
8 0.473626520202105
};
\addplot [semithick, darkorange25512714]
table {%
0 0.0369067660797094
1 0.272700177012712
2 0.508493849097414
3 0.744287413928805
4 0.786030443802105
5 0.786040641398901
6 0.786048099100342
7 0.748130792225979
8 0.512332997764018
};
\addplot [semithick, forestgreen4416044]
table {%
0 -0.135719838252225
1 -0.0644799327508863
2 0.00676177218176138
3 0.0780047147238196
4 0.237309221954114
5 0.237283710232411
6 0.237264439797211
7 0.153711643352704
8 0.0824871835858873
};
\addplot [semithick, crimson2143940]
table {%
0 -0.105086350466822
1 0.148160370459397
2 0.401407069199605
3 0.654653635142861
4 0.843893968837732
5 0.843842763107085
6 0.843804594226192
7 0.602252370397783
8 0.34902832181419
};
\addplot [semithick, mediumpurple148103189]
table {%
0 -0.080098391166524
1 0.0965677015342338
2 0.273234746596865
3 0.449902044635008
4 0.588824175956086
5 0.588810658114036
6 0.588801366811961
7 0.509150409351732
8 0.332491357213738
};
\addplot [semithick, sienna1408675]
table {%
0 0.00243342491266578
1 0.16553907203083
2 0.328644388286811
3 0.49174981573829
4 0.543581624819602
5 0.54355443722014
6 0.543533203336292
7 0.426506851697831
8 0.263417392133955
};
\addplot [semithick, orchid227119194]
table {%
0 0.0210089803689937
1 0.10492058287005
2 0.188832615015147
3 0.272745096287802
4 0.279629472791415
5 0.279598062038225
6 0.279572700276098
7 0.181220728671465
8 0.0973285481448907
};
\addplot [semithick, gray127]
table {%
0 -0.114328093342045
1 0.0489643062021619
2 0.212257810961054
3 0.375551653727342
4 0.544203825479202
5 0.544181934767976
6 0.544166630702892
7 0.445577638964433
8 0.282296202354999
};
\addplot [semithick, goldenrod18818934]
table {%
0 0.0885278360116505
1 0.373020402017891
2 0.657512435008128
3 0.942004362711092
4 0.948409803299175
5 0.948397404232171
6 0.948384947081402
7 0.906859409859805
8 0.622378436103817
};
\addplot [semithick, darkturquoise23190207]
table {%
0 -0.0223974289756245
1 0.168021960219307
2 0.358442154692467
3 0.548862504398738
4 0.634755422627431
5 0.634759233703817
6 0.634762571304401
7 0.602112152151683
8 0.411691891621534
};
\end{axis}

\end{tikzpicture}
        }
        \caption{$v^y$ $\Delta t = 1.5$}
        \label{subfig:vy-dt1p5}
    \end{subfigure}
    \hfill
    \begin{subfigure}{0.3\linewidth}
        \resizebox{\linewidth}{!}{
\begin{tikzpicture}

\definecolor{crimson2143940}{RGB}{214,39,40}
\definecolor{darkgray176}{RGB}{176,176,176}
\definecolor{darkorange25512714}{RGB}{255,127,14}
\definecolor{darkturquoise23190207}{RGB}{23,190,207}
\definecolor{forestgreen4416044}{RGB}{44,160,44}
\definecolor{goldenrod18818934}{RGB}{188,189,34}
\definecolor{gray127}{RGB}{127,127,127}
\definecolor{mediumpurple148103189}{RGB}{148,103,189}
\definecolor{orchid227119194}{RGB}{227,119,194}
\definecolor{sienna1408675}{RGB}{140,86,75}
\definecolor{steelblue31119180}{RGB}{31,119,180}

\begin{axis}[
tick align=outside,
tick pos=left,
x grid style={darkgray176},
xmin=-0.4, xmax=8.4,
xtick style={color=black},
y grid style={darkgray176},
ymin=-0.854186661214229, ymax=1.15688823217218,
ytick style={color=black},
width=10cm, 
height=5cm,
xlabel={$k$},
ylabel={$v_y$},
label style={font=\Large},
]
\addplot [semithick, steelblue31119180]
table {%
0 0.211886664436584
1 0.611255455779625
2 0.352164386840402
3 -0.0474830593893203
4 -0.353670993308824
5 -0.482982697063708
6 -0.354178126863527
7 0.0234045866778179
8 0.197267041744591
};
\addplot [semithick, darkorange25512714]
table {%
0 -0.0635127682906818
1 0.0573688455702313
2 0.300918842440779
3 0.542467886563157
4 0.640251033662128
5 0.560832092726505
6 0.324476779091435
7 0.0150927865182913
8 -0.0738683705677758
};
\addplot [semithick, forestgreen4416044]
table {%
0 0.0104277227582813
1 0.320895788171986
2 0.452689059648535
3 0.433333591320252
4 0.276715952211062
5 0.104315207285368
6 0.00721096442486331
7 -0.0021031373218281
8 -0.0166435211317652
};
\addplot [semithick, crimson2143940]
table {%
0 -0.0359556480049871
1 0.0771028966499423
2 0.209513581144774
3 0.321333889198661
4 0.399693254348903
5 0.488065360302612
6 0.522646749685346
7 0.323428136463095
8 -0.0134199746369474
};
\addplot [semithick, mediumpurple148103189]
table {%
0 0.0559452175086069
1 0.320498856364566
2 0.332007663893455
3 0.246529091988228
4 0.161759588721381
5 0.182360484329594
6 0.270022930587179
7 0.248096990036869
8 0.0430061159195221
};
\addplot [semithick, sienna1408675]
table {%
0 0.085122803362797
1 0.266167901851816
2 0.229901601216313
3 0.15804170774369
4 0.126406644037711
5 0.128887226043557
6 0.134035008625794
7 0.111115189856999
8 0.0836313237010341
};
\addplot [semithick, orchid227119194]
table {%
0 -0.040244433011393
1 0.528215216858275
2 0.844718491370591
3 0.967471929495689
4 0.868817039082025
5 0.677356553468307
6 0.44623009324981
7 0.123403946385906
8 -0.150699340077055
};
\addplot [semithick, gray127]
table {%
0 -0.0832302950509605
1 0.00501787513332816
2 0.193758212657472
3 0.400425970556672
4 0.573905069972967
5 0.663422930286861
6 0.623938609488205
7 0.361070304886961
8 0.0281140638858298
};
\addplot [semithick, goldenrod18818934]
table {%
0 -0.0684549848617194
1 0.0969360846780506
2 0.297562803439685
3 0.489060147559382
4 0.663141509461687
5 0.82463750562882
6 0.872985905564323
7 0.565234059493656
8 0.0346683250814527
};
\addplot [semithick, darkturquoise23190207]
table {%
0 0.153661000859754
1 0.963807243874685
2 0.922092573328773
3 0.368097471916108
4 -0.199274880717443
5 -0.515831528372413
6 -0.645183232852827
7 -0.477497662735395
8 0.0607681163502829
};
\addplot [semithick, steelblue31119180]
table {%
0 0.0484453230449048
1 0.196938893189877
2 0.162803585311201
3 0.109249914410713
4 0.107681165462808
5 0.134573101206403
6 0.121337803409373
7 0.0431637646104174
8 0.0152396220087655
};
\addplot [semithick, darkorange25512714]
table {%
0 -0.0715384094157826
1 0.146012022671293
2 0.601527167452727
3 0.731966337030066
4 0.505393810389217
5 0.124723006987753
6 -0.404413519385438
7 -0.762774166060301
8 -0.096281495379518
};
\addplot [semithick, forestgreen4416044]
table {%
0 0.0195021592640474
1 0.315205281004171
2 0.430459192506266
3 0.471186645635436
4 0.485642198973207
5 0.522121655088578
6 0.520410751290652
7 0.323285665858344
8 0.0153457134200809
};
\addplot [semithick, crimson2143940]
table {%
0 0.0473398963192906
1 0.423027749345897
2 0.451788201296392
3 0.347243498308938
4 0.209664423217158
5 -0.0926284624014828
6 -0.595482708337457
7 -0.728435733545446
8 -0.0437826663767347
};
\addplot [semithick, mediumpurple148103189]
table {%
0 -0.194415367160609
1 -0.305385306996624
2 0.237555388513193
3 0.833233456129068
4 1.06547573701825
5 0.92882830706789
6 0.663950068858334
7 0.334706704900109
8 -0.0970435569248346
};
\addplot [semithick, sienna1408675]
table {%
0 0.106634468925962
1 0.596856816073075
2 0.653405578637501
3 0.541413493384019
4 0.361147438318935
5 0.21447785191637
6 0.0963195285012833
7 -0.0203665358392831
8 -0.0540150779617606
};
\addplot [semithick, orchid227119194]
table {%
0 0.0287625191485909
1 0.317979191260926
2 0.411529632204945
3 0.420859255541349
4 0.346139343764482
5 0.233958004988694
6 0.157758064595588
7 0.117317482699989
8 0.049171360610521
};
\addplot [semithick, gray127]
table {%
0 -0.0298412792321505
1 0.105376548241758
2 0.230987326667478
3 0.311382344161904
4 0.294051821952697
5 0.118082885281962
6 -0.180079177049261
7 -0.349093513746464
8 -0.105837147177121
};
\addplot [semithick, goldenrod18818934]
table {%
0 -0.11764020731545
1 -0.0227601369554138
2 0.597609195400505
3 0.897879242341027
4 0.69877063562885
5 0.320200454676741
6 0.132856284268617
7 -0.00477122258390333
8 -0.0633087628635044
};
\addplot [semithick, darkturquoise23190207]
table {%
0 -0.0942148002849645
1 0.126663247702802
2 0.439637448727324
3 0.628531074841059
4 0.60295422852721
5 0.482354077057954
6 0.310637961562092
7 0.0587359842749522
8 -0.123145239906899
};
\end{axis}

\end{tikzpicture}
        }
        \caption{$v^y$ $\Delta t = 3.0$}
        \label{subfig:vy-dt3}
    \end{subfigure}
    \caption{Drones' trajectories (position space) and velocity components versus time.}
    \label{fig:traj-vel-drones}
\end{figure}

\begin{table}[ht]
    \centering
    \begin{tabular}{|c|c|c|c|c|c|c|c|}
        \hline
        $\Delta t$ & 1.0 & 1.25 & 1.5 & 1.75 & 2.0 & 2.5 & 3.0  \\
        \hline
        Opt. Val & 17.74 & 7.04 & 3.20 & 1.84 & 0.0 & 0.0 & 0.0 \\
        \hline
    \end{tabular}
    \caption{Wasserstein-GMM vs $\Delta t$ for Problem \ref{prob:step-cost-constrained-GMM-density-steering}}
    \label{tab:wass-vs-dt-problem4}
\end{table}

\section{Conclusion}\label{s:conclusion}
In this paper, we studied the optimal multi-modal density steering problem for linear dynamical systems by leveraging GMMs and covariance steering theory. To achieve this, we first formulated the hard-constrained density steering problem as an LP and transformed other constrained problems into bilinear optimization problems. Subsequently, we introduced a block coordinate descent procedure to address these bilinear programs effectively. Finally, we derived upper bounds for the GMM approximation error concerning the terminal state distribution. Possible future research directions include the utilization and comparison of other GMM based distance metrics (such as Cauchy-Schwarz Divergence \cite{p:kampa2011ClosedFormCSDivergenceGMM}) in GMM steering problems and the integration of chance-constraints (e.g. chance-constrained obstacle avoidance) into GMM steering problems.

\bibliographystyle{ieeetr}
\bibliography{density_steering}

\section{Biography Section}

\vspace{-22pt}
\begin{IEEEbiographynophoto}{Isin M. Balci}
received a B.Sc. in Mechanical Engineering from the Bogazici University, Istanbul, Turkey, in 2018 and the M.S. and Ph.D. degrees in Aerospace Engineering from the University of Texas at Austin, Austin, TX, USA, in 2020 and 2024, respectively. He is currently a software engineer at Applied Intuition. His research is mainly focused on control of uncertain and stochastic systems and optimization-based control.
\end{IEEEbiographynophoto}

\begin{IEEEbiographynophoto}{Efstathios Bakolas}
(Senior Member, IEEE) received the Diploma degree in Mechanical Engineering with highest honors from the National Technical University of Athens, Athens, Greece, in 2004 and the M.S. and Ph.D. degrees in Aerospace Engineering from the Georgia Institute of Technology, Atlanta, Atlanta, GA, USA, in 2007 and 2011, respectively. He is currently an Associate Professor with the Department of Aerospace Engineering and Engineering Mechanics, University of Texas at
Austin, Austin, TX, USA. His research is mainly focused on control of uncertain and stochastic systems, data-driven control of complex systems, optimal control theory, decision making and control of autonomous agents and multi-agent networks and differential games.
\end{IEEEbiographynophoto}

\vfill

\end{document}